\documentclass[11pt, twoside]{amsart}

\title[On extended Frobenius structures]{On extended Frobenius structures}

\author[A. Czenky]{Agustina Czenky}
\address{Czenky: Department of Mathematics, University of Southern California,
3620 S. Vermont Ave., KAP 104. Los Angeles, CA 90089, USA}
\email{czenky@usc.edu}

\author[J. Kesten]{Jacob Kesten}
\address{Kesten: Department of Mathematics, The Emery/Weiner School,
9825 Stella Link Road
Houston, TX 77025, USA}
\email{jkesten@emeryweiner.org}

\author[A. Quinonez]{Abiel Quinonez}
\address{Quinonez: Department of Mathematics,  University of Texas,  2515
Speedway C1200,  Austin, TX 78712, USA}
\email{adq243@utexas.edu}

\author[C. Walton]{Chelsea Walton}
\address{Walton: Department of Mathematics, Rice University,
P.O. Box 1892, Houston, TX 77005, USA}
\email{notlaw@rice.edu}

\usepackage{amsmath}
\usepackage{amsfonts}
\usepackage{amsthm}
\usepackage{amssymb,bbm}
\usepackage{dsfont}
\usepackage{stmaryrd}
\usepackage{mathabx}
\usepackage{appendix}
\usepackage[english]{babel}
\usepackage{url}
\usepackage{fancyhdr}
\usepackage{graphicx}
\usepackage{verbatim}
\usepackage[normalem]{ulem}
\usepackage[shortlabels]{enumitem}
\newcommand{\stkout}[1]{\ifmmode\text{\sout{\ensuremath{#1}}}\else\sout{#1}\fi}

\usepackage[draft]{hyperref}

\usepackage[all,cmtip]{xy}
\usepackage{geometry}
\usepackage[dvipsnames]{xcolor}
\usepackage{quiver}
\usepackage{lscape} 
\usepackage{multicol}

\usepackage[utf8]{inputenc}
\usepackage{tikz}
\usepackage{tikz-cd}

\allowdisplaybreaks

\usepackage{adjustbox}

\newcommand \red {\textcolor{red}}
\newcommand \blue {\textcolor{blue}}
\newcommand \green {\textcolor{green}}
\newcommand \gray {\textcolor{gray}}
\definecolor{forest}{rgb}{0.0, 0.55, 0.0}
\definecolor{BetterBlue}{rgb}{0.2, 0.4, 1.0}
\definecolor{BetterPurple}{rgb}{0.6, 0.3, 0.9}

\newcommand{\ac}{\textcolor{BetterPurple}}
\newcommand{\jk}{\textcolor{forest}}
\newcommand{\aq}{\textcolor{BetterBlue}}
\newcommand{\cw}{\textcolor{Cerulean}}

\renewcommand{\baselinestretch}{1.1}

\usepackage{enumitem}

\usepackage[notcite,notref,final]{showkeys}

\usepackage{calrsfs}
\DeclareMathAlphabet{\cal}{OMS}{zplm}{m}{n}

\DeclareMathAlphabet{\mathsf}{OT1}{cmss}{m}{n} 

\newcommand{\varep}{\varepsilon}

\newcommand{\op}[1]{\mathrm{#1}}
\newcommand{\oop}{\mathrm{op}}
\newcommand{\cop}{\mathrm{cop}}

\newcommand{\Id}{\textnormal{Id}}
\newcommand{\id}{\textnormal{id}}
\newcommand{\one}{\mathbbm{1}}
\newcommand{\Ra}{\Rightarrow}

\newcommand{\Hom}{\textnormal{Hom}}

\newcommand{\tr}{\mathsf{tr}}

\newcommand{\HDelta}{\underline{\Delta}}
\newcommand{\Hvarep}{\underline{\varep}}
\newcommand{\HU}{\underline{U}}
\newcommand{\Hbeta}{\underline{\beta}}
\newcommand{\HHvs}{\underline{H}}
\newcommand{\Hmult}{\underline{m}}
\newcommand{\Hunit}{\underline{u}}
\newcommand{\Hanti}{\underline{S}}
\newcommand{\HLam}{\underline{\Lambda}}
\newcommand{\Hlam}{\underline{\lambda}}
\newcommand{\HPsi}{\underline{\Psi}}

\usepackage{relsize}

\newcommand{\act}{\triangleright}
\newcommand\osq{{\mathbin{\text{\raisebox{-.25ex}{\scalebox{1.5}{$\square$}}}}}}

\newcommand\natisom{\stackrel{\hbox{$\sim$\hspace{.02in}}}{\smash{\Rightarrow}\rule{0pt}{0.4ex}}}

\newcommand\natisommon{\stackrel{\scalebox{.6}{$\otimes$}}{\stackrel{\hbox{$\sim$\hspace{.02in}}}{\smash{\Rightarrow}\rule{0pt}{0.4ex}}}}

\newcommand\equivto{\stackrel{\hbox{$\sim$\hspace{.02in}}}{\smash{\to}\rule{0pt}{0.3ex}}}

\newcommand{\pullbackcorner}[1][dr]{\save*!/#1-2pc/#1:(-1,1)@^{|-}\restore}

\newcommand{\pushoutcorner}[1][ul]{\save*!/#1-2pc/#1:(-1,1)@^{|-}\restore}

\newcommand{\Alg}{\mathsf{Alg}}
\newcommand{\Bialg}{\mathsf{Bialg}}
\newcommand{\Bimod}{\mathsf{Bimod}}
\newcommand{\Coalg}{\mathsf{Coalg}}
\newcommand{\ComAlg}{\mathsf{ComAlg}}
\newcommand{\CoMod}{\mathsf{Comod}}
\newcommand{\End}{\mathsf{End}}
\newcommand{\FrobAlg}{\mathsf{FrobAlg}}
\newcommand{\ExtFrobAlg}{\mathsf{ExtFrobAlg}}
\newcommand{\Fun}{\mathsf{Fun}_\Bbbk}
\newcommand{\HopfAlg}{\mathsf{HopfAlg}}
\newcommand{\Mod}{\mathsf{Mod}}
\newcommand{\MonCat}{\mathsf{MonCat}}
\newcommand{\Rep}{\mathsf{Rep}}
\newcommand{\Vect}{\mathsf{Vec}}
\newcommand{\Top}{\mathsf{Top}}
\newcommand{\Mat}{\textnormal{Mat}}

\newcommand{\mC}{\mathbb{C}}
\newcommand{\mE}{\mathbb{E}}
\newcommand{\mR}{\mathbb{R}}
\newcommand{\mQ}{\mathbb{Q}}
\newcommand{\mL}{\mathbb{L}}
\newcommand{\mZ}{\mathbb{Z}}
\newcommand{\mN}{\mathbb{N}}
\newcommand{\mH}{\mathbb{H}}
\newcommand{\mA}{\mathbb{A}}
\newcommand{\mF}{\mathbb{F}}

\newcommand{\cA}{\cal{A}}
\newcommand{\cC}{\cal{C}}
\newcommand{\cD}{\cal{D}}
\newcommand{\cB}{\cal{B}}
\newcommand{\cE}{\cal{E}}
\newcommand{\cF}{\cal{F}}
\newcommand{\cH}{\cal{H}}
\newcommand{\cO}{\cal{O}}
\newcommand{\cI}{\cal{I}}
\newcommand{\cL}{\cal{L}}
\newcommand{\cP}{\cal{P}}
\newcommand{\cT}{\cal{T}}
\newcommand{\cU}{\cal{U}}
\newcommand{\cV}{\cal{V}}
\newcommand{\cW}{\cal{W}}
\newcommand{\cM}{\cal{M}}
\newcommand{\cN}{\cal{N}}
\newcommand{\cS}{\cal{S}}
\newcommand{\cX}{\cal{X}}
\newcommand{\cZ}{\cal{Z}}

\newcommand{\rA}{\mathrm{A}}
\newcommand{\rB}{\mathrm{B}}
\newcommand{\rC}{\mathrm{C}}
\newcommand{\rD}{\mathrm{D}}
\newcommand{\rE}{\mathrm{E}}
\newcommand{\rF}{\mathrm{F}}
\newcommand{\rG}{\mathrm{G}}
\newcommand{\rH}{\mathrm{H}}
\newcommand{\rR}{\mathrm{R}}
\newcommand{\rT}{\mathrm{T}}
\newcommand{\rS}{\mathrm{S}}
\newcommand{\rU}{\mathrm{U}}
\newcommand{\bfD}{\mathbf{D}}

\newcommand{\textand}{\phantom{mm}\textnormal{and}\phantom{mm}}
\newcommand{\adj}[4]{#1\negmedspace: #2\rightleftarrows #3:\negmedspace #4}

\numberwithin{equation}{section}

\newtheorem{theorem}{Theorem}[section]

\newtheorem{proposition}[theorem]{Proposition}

\newtheorem{corollary}[theorem]{Corollary}
\newtheorem{lemma}[theorem]{Lemma}
\newtheorem{conjecture}[theorem]{Conjecture}

\newtheorem{theorem*}{Theorem}

\theoremstyle{definition}
\newtheorem{definition}[theorem]{Definition}

\newtheorem{example}[theorem]{Example}
\newtheorem{remark}[theorem]{Remark}

\makeatletter              
\let\c@equation\c@theorem  
\makeatother
\numberwithin{equation}{section}

\def\changemargin#1#2{\list{}{\rightmargin#2\leftmargin#1}\item[]}
\let\endchangemargin=\endlist

\newcommand{\skp}{\vspace{0.07in}}

\newcommand{\hs}{\hspace{.02in}} 

\makeatletter
\@namedef{subjclassname@2020}{%
  \textup{2020} Mathematics Subject Classification}
\makeatother

\subjclass[2020]{16K99, 18M15, 18M30, 57R56}

\keywords{extended Frobenius algebra,  extended Frobenius monoidal functor, extended Hopf algebra}

\tikzcdset{scale cd/.style={every label/.append style={scale=#1},
    cells={nodes={scale=#1}}}}

\begin{document}

\begin{abstract}
A classical result in quantum topology is that oriented 2-dimensional topological quantum field theories (2-TQFTs) are fully classified by commutative Frobenius algebras. In 2006, Turaev and Turner introduced additional structure on Frobenius algebras, forming what are called extended Frobenius algebras, to classify 2-TQFTs in the unoriented case. This work provides a systematic study of extended Frobenius algebras in various settings: over a field, in a monoidal category, and in the framework of monoidal functors. Numerous examples, classification results, and general constructions of extended Frobenius algebras are established.
\end{abstract}

\maketitle







\section{Introduction}\label{sec:intro}

The goal of this work is to study extended Frobenius algebras in various settings. Before providing further context, note that linear structures here are over an algebraically closed field $\Bbbk$ of characteristic zero, unless stated otherwise. Categories $\cC$ are assumed to be locally small, and will have further structure as specified below.  We will also read graphical diagrams from top to bottom.

\medskip

We are motivated by the vast program on producing {\it topological quantum field theories (TQFTs)}, which are categorical constructions that yield topological invariants. Loosely speaking,  a TQFT is a (higher) functor from a (higher) category of topological data to a (higher) target category with extra structure. In the 2-dimensional case, 2-TQFTs are symmetric monoidal functors from a symmetric monoidal category of 1-manifolds and 2-cobordisms  to a choice of a symmetric monoidal category $\cC$. Often, $\cC$ is taken to be the symmetric monoidal category $\mathsf{Vec}$ of $\Bbbk$-vector spaces. A classical result is that a 2-TQFT with values in $\cC$ is classified by where it sends the circle, which in the oriented setting, is a commutative Frobenius algebra in $\cC$; see, e.g., \cite{Kock}. Turaev and Turner expanded this correspondence in the unoriented setting, by tacking on extra structure to Frobenius algebras to form what are called extended Frobenius algebras \cite[Section~2]{TuraevTurner}. 

\medskip

\noindent \textbf{Turaev-Turner's 2-TQFT Result $(\star)$:} 
Isoclasses of unoriented 2-dimensional TQFTs in $\mathsf{Vec}$ are in bijection with isomorphism classes of commutative extended Frobenius
algebras over~$\Bbbk$. 

\medskip

Since then, extended Frobenius algebras have appeared in many works, such as  in  an adaptation of $(\star)$ to compute virtual link homologies \cite{Tubbenhauer},  for an analogue of $(\star)$ for homotopy quantum field theories  \cite{Tagami},  in  a modification of $(\star)$ to examine linearized TQFTs \cite{Czenky},  in  a categorical expansion of  $(\star)$ \cite{Ocal}, and in a study of topological invariants of ribbon graphs \cite{ChengLei}.

\medskip

We expect that extended Frobenius algebras will continue to play a crucial role in the TQFT program. Thus, we focus on the algebraic side of the program and study extended Frobenius algebras in detail-- producing numerous examples, classification results, and general constructions.

\medskip

We begin by taking $\cC = \mathsf{Vec}$, hence working over the field $\Bbbk$. Consider the  terminology below.

\begin{definition} \label{def:extFrob-k}
\begin{enumerate}[\upshape (a)]
\item A {\it Frobenius algebra} over $\Bbbk$ is a tuple $(A,m,u,\Delta, \varep)$, where $(A,m,u)$ is an associative unital $\Bbbk$-algebra, and $(A,\Delta,\varep)$ is a coassociative counital $\Bbbk$-coalgebra, satisfying the Frobenius law:
$(a \otimes 1_A)\Delta(b)= \Delta(ab) = \Delta(a)(1_A \otimes b)$,
for all $a,b \in A$.
A {\it morphism of Frobenius algebras over $\Bbbk$} is a morphism of the underlying $\Bbbk$-algebras and of $\Bbbk$-coalgebras.

\medskip

\item \cite[Definition 2.5]{TuraevTurner} A Frobenius algebra $(A,m,u,\Delta, \varep)$ is an  {\it extended Frobenius algebra} over $\Bbbk$ if it is equipped with a morphism $\phi: A \to A$ and an element $\theta \in A$ such that:
\smallskip
\begin{enumerate}[(i)]
\item $\phi:A \to A$ is an involution of Frobenius algebras,  
\smallskip
\item $\theta \in A$ satisfies $\phi(\theta a) = \theta a$, for all $a \in A$, 
\smallskip
\item $m (\phi \otimes \id_A) \Delta(1_A) = \theta^2$.
\end{enumerate}

\smallskip

\noindent A {\it morphism $f:(A,\phi_A,\theta_A) \to (B,\phi_B,\theta_B)$ of extended Frobenius algebras over $\Bbbk$} is a morphism $f: A \to B$ of  $\Bbbk$-Frobenius algebras such that $f \hs \phi_A = \phi_B \hs f$ and $f(\theta_A) = \theta_B$.

\medskip

\item We refer to $(\phi, \theta)$ in part (b) as the  {\it extended structure} of the underlying Frobenius algebra~$A$, and  say that  $A$ is {\it extendable} when $\phi$ and $\theta$ exist.
We also call an extended structure $(\phi, \theta)$ on $A$   {\it $\phi$-trivial} when $\phi = \id_A$, and  call it {\it $\theta$-trivial} when $\theta = 0$.
\end{enumerate}
\end{definition}

Note that we do not assume that algebras are commutative in our work. Our first main result is the classification of extended structures for  well-known examples of Frobenius algebras over $\Bbbk$. 

\begin{theorem}[Propositions~\ref{prop:k-class}--\ref{prop:xn-class},~\ref{prop:C2-class}--\ref{prop:C4-class},~\ref{prop:V4-class}--\ref{prop:T2-class}] \label{thm:Frobk-class}
Take $n \geq 2$, and $\omega_n \in \Bbbk$ an $n$-th root of unity. The extended structures for the Frobenius algebras below are classified, recapped as~follows.
\begin{enumerate}[\upshape (a)]
\item  $\Bbbk:$ all extensions are $\phi$-trivial.
\smallskip
\item $\mC$ over $\mR$: all extensions are $\phi$-trivial or $\theta$-trivial.
\smallskip
\item $\Bbbk[x]/(x^n)$: all extensions are $\phi$ trivial when $n$ is odd, and  is not extendable when $n$ is even.

\smallskip

\item $\Bbbk C_2$: all extensions are $\phi$-trivial or $\theta$-trivial.

\smallskip

\item $\Bbbk C_3$: all extensions are $\phi$-trivial or $\phi$ maps a generator $g$ of $C_3$ to $\omega_3g^2$.
\smallskip
\item $\Bbbk C_4$: all extensions are $\phi$-trivial, or $\theta$-trivial, or $\phi$ maps a generator $g$ of $C_4$ to $\omega_4g^3$.
\smallskip
\item $\Bbbk (C_2 \times C_2)$: here, $\phi$ maps $g$ to $\omega_2 g'$, where $g,g'$ are generators of $C_2 \times C_2$.
\smallskip
\item $T_2(-1):=\Bbbk\langle g,x \rangle/(g^2-1, x^2, gx+xg):$ all extensions are $\phi$-trivial. \qed
\end{enumerate}
\end{theorem}

Next, we move to the monoidal setting. See Section~\ref{sec:monbackground} for background material on monoidal categories $\cC:=(\cC, \otimes, \one)$ and on algebraic structures within $\cC$, especially (extended) Frobenius algebras in  $\cC$. This specializes to the setting above by working in $(\mathsf{Vec}, \otimes_\Bbbk, \Bbbk)$. Let $\mathsf{ExtFrobAlg}(\cC)$ denote the category of extended Frobenius algebras in $\cC$ [Definition~\ref{def:extFrobC}]. We first establish monoidal structures on $\mathsf{ExtFrobAlg}(\cC)$. Namely, if $\cC$ is also symmetric, then $\mathsf{ExtFrobAlg}(\cC)$ is monoidal with $\otimes = \otimes^\cC$ and $\one = \one^\cC$ [Proposition~\ref{prop:Ext-mon}]. Moreover, if $\cC$ is additive monoidal, then $\mathsf{ExtFrobAlg}(\cC)$ is monoidal with $\otimes$ being the biproduct of~$\cC$ and $\one$ being the zero object  of $\cC$ [Proposition~\ref{prop:eFMbiprod}].

\medskip

Now we focus on separability in a monoidal category $\cC$. A Frobenius algebra in $\cC$ is {\it separable} if its comultiplication map is a right inverse of its multiplication map [Definition~\ref{def:sep}]. Separability (or {\it specialness}) is a widely used condition in quantum theory (see, e.g., \cite{Muger, RFFS, HeunenVicary}). In particular, it is used to construct {\it state sum 2-TQFTs} \cite{NovakRunkel}. This brings us to the result below.

\begin{proposition}[Proposition~\ref{prop:sepExt}] \label{prop:sep-ext-intro}
A separable Frobenius algebra in a monoidal category is always extendable. \qed
\end{proposition}

Next, we turn our attention to Hopf algebras, which also play a role in quantum theory and TQFTs (see, e.g., \cite{KerlerLyubashenko, BBG, CCC}). It is well-known that finite-dimensional Hopf algebras over $\Bbbk$ (or more generally, Hopf algebras over $\Bbbk$ with a certain integral)  admit a Frobenius structure. A lesser known result is that in a {\it symmetric} monoidal category $\cC$,  {\it integral Hopf algebras} in $\cC$ [Definition~\ref{def:HopfAlg}] are Frobenius [Proposition~\ref{prop:hopftofrob}]. 
A graphical proof of this result is in Appendix~\ref{sec:HF}, which may be of independent interest to the reader. Building on this, we introduce {\it extended Hopf algebras} in symmetric monoidal categories [Definition~\ref{def:exthopf}], and obtain the result~below.

\begin{proposition} [Proposition~\ref{prop:eHopf}] \label{prop:ext-Hopf-integral} 
If an integral Hopf algebra in a symmetric monoidal category is extendable, then so is its corresponding Frobenius structure (via Proposition~\ref{prop:hopftofrob}). 
\qed
\end{proposition}

Finally, we examine functors that preserve extended Frobenius algebras in monoidal categories. To start, take monoidal categories $\cC$ and  $\cC'$, and note that a {\it Frobenius monoidal functor} $\cC \to \cC'$ [Definition~\ref{def:monfunc}] sends Frobenius algebras in $\cC$ to  those in $\cC'$. It is also known that the separability condition is preserved when such a functor is separable [Proposition~\ref{prop:pres}], and that such functors can be used to form higher categorical structures [Remark~\ref{rem:2cats}]. See also \cite{DayPastro} and  \cite[Chapter~6]{Bohm} for more details. Our last set of results extends the theory of  Frobenius monoidal functors by introducing the notion of an {\it extended Frobenius monoidal functor} [Definition~\ref{def:eFMfunc}]. We  establish that this construction satisfies many desirable conditions as discussed below.

\begin{theorem}[Propositions~\ref{prop:sepisEFM},~\ref{prop:eFMpres}, Theorem~\ref{thm:eFMcomp}, Remark~\ref{rem:2-cat}] \label{thm:eFrobMF-1} The following statements hold.
\begin{enumerate}[\upshape (a)]
\item A separable Frobenius monoidal functor is extended Frobenius monoidal.

\smallskip

\item An extended Frobenius monoidal functor preserves extended Frobenius algebras.

\smallskip

\item The composition of two extended Frobenius monoidal functors is extended Frobenius monoidal.

\smallskip

\item The collections of monoidal categories and extended Frobenius monoidal functors between them forms a (2-)category (with 2-cells being certain natural transformations). \qed
\end{enumerate}
\end{theorem}

Parts~(b,c) require intricate arguments (deferred to an appendix only appearing in the ArXiv preprint of this work). Various separable Frobenius monoidal functors appear in the literature; see, e.g.,  \cite{Szlachanyi, McCurdyStreet, Mori, BulacuTorrecillas, HLR, FHL, Yadav}. So, parts~(a,b) above imply that each of these constructions produce extended Frobenius algebras in monoidal categories. There are also extended Frobenius monoidal functors that are not necessarily separable [Examples~\ref{Ex:Ext-mon},~\ref{Ex:eFMbiprod}].

\bigskip

\noindent {\bf Organization of the article.} In Section~\ref{sec:pre-fld}, we study extended Frobenius algebras over a field, proving Theorem~\ref{thm:Frobk-class}. In Section~\ref{sec:pre-mon}, we focus on extended Frobenius algebras in a monoidal category $\cC$, and introduce graphical calculus diagrams for such structures. We also establish monoidal structures on the category of extended Frobenius algebras in $\cC$ in Section~\ref{sec:pre-mon}. In Section~\ref{sec:SepHopf}, we make connections to separable algebras in monoidal categories, and verify Proposition~\ref{prop:sep-ext-intro}. We also strengthen ties to Hopf algebras in monoidal categories in Section~\ref{sec:SepHopf}, obtaining Proposition~\ref{prop:ext-Hopf-integral}. The result that integral Hopf algebras are Frobenius is verified in Appendix~\ref{sec:HF} via graphical calculus arguments. In Section~\ref{sec:eFrobMF}, we introduce extended Frobenius monoidal functors, and establish Theorem~\ref{thm:eFrobMF-1}. Portions of the proof of Theorem~\ref{thm:eFrobMF-1} involve lengthy commutative diagram calculations; these are included in Appendix~B,  appearing only in the ArXiv preprint version of this work.


\section{Extended Frobenius algebras over a field}
\label{sec:pre-fld}

In this section, we study  extended Frobenius algebras over a field $\Bbbk$ as  introduced in Definition~\ref{def:extFrob-k}. We provide many examples of, and preliminary results for, such structures in Section~\ref{sec:Frobk-exam}. Then, in Section~\ref{sec:classfn}, we establish Theorem~\ref{thm:Frobk-class} on the classification of extended structures for several Frobenius algebras over $\Bbbk$.

 \smallskip

\begin{center}
{\it The roman numerals \textnormal{(i)}, \textnormal{(ii)}, \textnormal{(iii)} here will refer to the conditions in Definition~\ref{def:extFrob-k}(b).}
\end{center}

\subsection{Preliminary results and examples} 
\label{sec:Frobk-exam}

We begin with some useful preliminary results on extended Frobenius algebras $A$ over $\Bbbk$. First, the Frobenius law from Definition~\ref{def:extFrob-k}(a) implies that 
\begin{equation} \label{eq:DeltaA}
\Delta(a) = a (1_A)^1 \otimes (1_A)^2, \quad  \text{for $\Delta(1_A) := (1_A)^1 \otimes (1_A)^2$},
\end{equation}
for $a \in A$. So, $\Delta(1_A)$ determines the Frobenius structure of $A$.

\begin{lemma}\label{lemma:domain-extstruc}
If $A$ is a Frobenius algebra that is a domain, then an extended structure of $A$ (if it exists) must be either $\phi$-trivial or $\theta$-trivial.
\end{lemma}

\begin{proof}
Suppose that an extended structure $(A, \phi, \theta)$ exists. Then, $\theta \phi(a)=\phi(\theta) \phi(a)=\phi(\theta a)=\theta a$,
for all $a \in A$ by condition~(i).  Hence, $\theta(\phi(a)-a)=0$ for all $a\in A$, and the result follows from $A$ being a domain. 
\end{proof}

\begin{lemma}\label{lemma: no morph between struct} 
Let $A$ be a Frobenius algebra over $\Bbbk$, and let $(A, \phi, \theta)$ and $(A, \phi', \theta')$ be two extended structures of $A$. If $\theta \in \Bbbk 1_A$  and $\theta\ne \theta'$, then an extended Frobenius algebra morphism from $(A, \phi, \theta)$ to $(A, \phi', \theta')$ does not exist. 
\end{lemma}
	
\begin{proof}
Suppose by way of contrapositive that $\theta = \lambda 1_A$ for some $\lambda \in \Bbbk$ and there is a morphism $f:(A, \phi, \theta)\to (A, \phi', \theta')$ of extended Frobenius algebras. Since $f$ is unital and preserves the extended structure,  $\theta  = \lambda 1_A = \lambda f(1_A) =   f(\lambda 1_A)= f(\theta)=\theta'$, as desired.
\end{proof}

We will see in Proposition~\ref{prop:C2-class} that Lemma~\ref{lemma: no morph between struct} fails when $\theta \not \in \Bbbk 1_A$. We now include some examples of extended structures for well-known Frobenius algebras.

\begin{example}\label{ex: KG}
Let $G$ be a finite group. Its group algebra $\Bbbk G$ has a Frobenius algebra structure determined by 
$\Delta(e_G)=\sum_{h\in G} h\otimes h^{-1}$. Then, 
\[
\phi=\id_{\Bbbk G}, \qquad \theta=\pm\sqrt{|G|} \cdot e_G
\] 
yields extended structures of $\Bbbk G$. Now, conditions~(i) and~(ii) are trivially satisfied. Condition~(iii) holds as
$ m(\phi\otimes \id_{\Bbbk G})\Delta (e_{G})=m\left(\sum_{h\in G} h\otimes h^{-1}\right)= |G| \cdot e_G =\theta^2$	.
\end{example}

\begin{example}\label{example: ext on Cn}
Let $C_n$ denote the cyclic group of order $n \geq 2$, and let $g$ denote a generator of $C_n$.  Consider the Frobenius structure on  $\Bbbk C_n$  as defined in Example \ref{ex: KG}.
Then  
\[
\textstyle	\phi(g)=\omega_ng^{-1}, \qquad \textstyle \theta=\pm\frac{1}{\sqrt n}\sum_{j=0}^{n-1}\omega_n^j g^{-2j}
\]
is an extended structure of $\Bbbk C_n$ for any $n$-th root of unity $\omega_n\in \Bbbk.$  It is a quick check that condition~(i) holds. Towards condition~(ii), let $a:=\sum_{i=0}^{n-1} a_ig^i$ be an element in $\Bbbk C_n$. Then, 
\begin{align*}
\textstyle	\phi(a\theta)
\; = \;  \pm\frac{1}{\sqrt{n}}\sum_{i,j=0}^{n-1}a_i\omega_n^{j} \phi(g)^{i-2j}
\; = \; \pm\frac{1}{\sqrt{n}}\sum_{i,j=0}^{n-1}a_i\omega_n^{i-j} g^{-i+2j}
\; = \;  \pm\frac{1}{\sqrt{n}}\sum_{i,k=0}^{n-1}a_i\omega_n^{k} g^{i-2k}
\; = \;  a\theta.
\end{align*}
For  condition~(iii), we compute:
\begin{align*}
\textstyle	m(\phi\otimes \id_{\Bbbk G})\Delta (e_{C_n})
&= \textstyle m(\phi\otimes \id_{\Bbbk C_n})\left(\sum_{j=0}^{n-1} g^{j}\otimes g^{-j} \right)
\; = \; \textstyle \sum_{j=0}^{n-1} \omega_n^j g^{-2j}\\
&  =  \textstyle \frac{1}{n}\sum_{i=0}^{n-1}\sum_{k=0}^{n-1}\omega_n^{k} g^{-2k}
\; = \; \textstyle \frac{1}{n}\sum_{i,j=0}^{n-1}\omega_n^{i+j} g^{-2(i+j)}
\; = \;  \textstyle \frac{1}{n}\left(\sum_{j=0}^{n-1}\omega_n^j g^{-2j}\right)^2
 =  \theta^2.
\end{align*}
 \end{example}

\begin{example}  \label{example:Taft} Let $\omega:=\omega_n$ be a primitive $n$-th root of unity, for $n \geq 2$.
Consider the Taft algebra, 
$$T_n(\omega):=\Bbbk\langle g,x \rangle/(g^n-1, x^n, gx - \omega xg),$$ 
with Frobenius structure determined by
$ \Delta(1_{T_n(\omega)})=\sum_{j=0}^{n-1}\left(-\omega^j g^{j+1}\otimes g^{-(j+1)} x+ g^jx\otimes g^{-j}\right)$.
Then, this Frobenius structure on $T_n(\omega)$ can be extended via 
\[
\phi = \id_{T_n(\omega)},\qquad \textstyle \theta \; \in \bigoplus_{j=0,k=1}^{n-1}\Bbbk g^jx^k.
\]
To show this, we compute:
$m(\phi\otimes \id_{T_n(\omega)})\Delta(1)  
= 0 = \theta^2$, so condition~(iii) holds. Conditions~(i) and~(ii) are trivially satisfied.  
\end{example}

\begin{example}
Let $\Mat_{n}(\Bbbk)$ be the algebra of $n\times n$ matrices over $\Bbbk$, with basis $\{E_{i,j}\}_{i,j=1}^n$ of elementary matrices. Consider the Frobenius structure determined by 
$\textstyle \Delta(E_{i,j})=\sum_{\ell =1}^n E_{i,\ell} \otimes E_{\ell,j}$,
for all $1\leq i,j\leq n$.  Then, 
\[
\phi=\id_{\Mat_n(\Bbbk)}, \qquad
\theta=\pm \sqrt{n} \cdot I_n
\]
 give  extended structures of $\Mat_{n}(\Bbbk)$. Indeed,  
$m(\phi\otimes \id_{\Mat_n(\Bbbk)})\Delta(I_n)=\textstyle \sum_{i,\ell=1}^n E_{i,\ell}E_{\ell,i} =  n \cdot I_n =\theta^2$, so  condition~(iii) holds.
Moreover, conditions~(i) and~(ii) are trivially satisfied.   
\end{example}


\subsection{Classification results} 
\label{sec:classfn} 
Now we proceed to establish Theorem~\ref{thm:Frobk-class}, starting with the results for the Frobenius algebras: $\Bbbk$ over $\Bbbk$, $\mathbb{C}$ over $\mathbb{R}$, and the nilpotent algebra $\Bbbk[x]/(x^n)$ over $\Bbbk$.

\begin{proposition} \label{prop:k-class}
The only extended structures of the Frobenius algebra $\Bbbk$ where $\Delta_{\Bbbk}: \Bbbk \overset{\sim}{\to} \Bbbk \otimes \Bbbk$ are $\phi$-trivial, with $\theta=\pm 1_{\Bbbk}$. Moreover, these extended Frobenius algebra structures are non-isomorphic.
\end{proposition}
	
\begin{proof}
Suppose $\phi$ and $\theta$ give an extended structure of $\Bbbk$. Since $\phi:\Bbbk \to \Bbbk$ is a morphism of algebras, the only possible choice is $\phi=\id_{\Bbbk}$, which satisfies conditions~(i) and (ii) trivially.  Condition~(iii) implies that $\theta=\pm 1_{\Bbbk}$. Lastly, the structures are non-isomorphic by Lemma \ref{lemma: no morph between struct}. 
\end{proof}

\begin{proposition}  \label{prop:C-class}
Consider the Frobenius algebra $\mathbb C$ over $\mathbb R$ with 
$\Delta(1)=1\otimes 1-i\otimes i$. 
Then, 
 \begin{enumerate}[\upshape (a)]
     \item $\phi=\id_{\mathbb C}$ and $\theta=\pm \sqrt 2$, and

     \smallskip
     
     \item $\phi(z) = \overline{z}$ for all $z\in\mathbb C$, and $\theta=0$,
 \end{enumerate}
 are all of the extended structures of $\mathbb C$, and these extended Frobenius algebras are non-isomorphic.
\end{proposition}
	\begin{proof}
By Lemma \ref{lemma:domain-extstruc}, an extended structure of $\mathbb C$ should be $\phi$-trivial or $\theta$-trivial. If $\phi=\id_{\mathbb C},$ then 
$\theta^2
= m(\phi\otimes \id_{\mathbb C})\Delta (1)
= m(1\otimes 1- i\otimes i)
=2$,
and so $\theta=\pm \sqrt{2}$. On the other hand, if $\theta=0$, then 
$0= m(\phi\otimes \id_{\mathbb C})\Delta (1)
=1-\phi(i)i$.
Hence, $\phi(i)=-i$ and it follows that $\phi$ must be complex conjugation. Now condition~(iii) holds, and it is a quick check that conditions~(i) and~(ii) are satisfied for these choices. Lastly, it follows from Lemma~\ref{lemma: no morph between struct} that these structures are all non-isomorphic. 
\end{proof}

\begin{proposition}  \label{prop:xn-class}
Consider the algebra $\Bbbk[x]/(x^n)$, for $n \geq 2$, with Frobenius structure determined by
$\Delta(1)=\sum_{i=0}^{n-1} x^{i}\otimes x^{n-i-1}$. Then, the following statements hold.
\begin{enumerate}[\upshape (a)]
\item For $n$ even, $\Bbbk[x]/(x^n)$ is not extendable.

\smallskip

\item For $n$ odd,  all extended structures of $\Bbbk[x]/(x^n)$ are $\phi$-trivial, with $\theta=\pm \sqrt{n} x^{\frac{n-1}{2}}+\sum_{j=\frac{n+1}{2}}^{n-1} \theta_jx^j$ for some $\theta_{\frac{n+1}{2}},\dots, \theta_{n-1}\in \Bbbk.$
\end{enumerate}
\end{proposition}

\begin{proof}
Suppose that $\phi$ and $\theta$ give an extended structure of $\Bbbk[x]/(x^n)$. Then, a routine calculation with $\phi$ being multiplicative and $\phi^2 = \id$ (from condition~(i)) implies that $\phi(x) = \pm x$.
So, in the rest of the proof, we look at the cases $\phi=\id$ and $\phi(x)=-x$, and conclude the latter is never possible, while the former is only possible when $n$ is odd. 

Suppose first that $\phi=\id$. Then, conditions~(i) and~(ii) are satisfied trivially.
Let $\theta_0, \dots, \theta_{n-1}\in \Bbbk$ such that $\theta=\sum_{i=0}^{n-1}\theta_ix^i$. Then, condition~(iii) implies that 
\begin{align}\label{eq: torus condiction on k(x)/x^n}
\textstyle		nx^{n-1}=	\sum_{i=0}^{n-1}\theta_i^2x^{2i}+\sum_{i\ne j}\theta_i\theta_i x^{i+j}.
\end{align}
From the coefficient of $1$, it follows that $\theta_0=0.$	
We can argue by induction that $\theta_i=0$ for all $0\leq i \leq \frac{n-1}{2}-1$ if $n$ is odd, and  for all $0\leq i \leq \frac{n}{2}-1$ if $n$ is even. It follows that if $n$ is even, then the coefficient of $x^{n-1}$ in~\eqref{eq: torus condiction on k(x)/x^n} leads to the contradiction:
$\textstyle	n=2\sum_{i=0}^{\frac{n}{2}-1} \theta_i\theta_{n-1-i}=0$.
Thus, $\phi=\id$ is not possible when $n$ is even. On the other hand, if $n$ is odd, then the coefficient of $x^{n-1}$ in~\eqref{eq: torus condiction on k(x)/x^n}  yields
$\textstyle	n=(\theta_{\frac{n-1}{2}})^2+2\sum_{i=0}^{\frac{n-1}{2}-1} \theta_i\theta_{n-1-i}$,
which implies that $\theta_{\frac{n-1}{2}}=\pm \sqrt n \cdot 1_{\Bbbk}.$ 
So, $\phi=\id$ and $\theta = \pm \sqrt{n} x^{\frac{n-1}{2}}+\sum_{j=\frac{n+1}{2}}^{n-1} \theta_jx^j$ precisely satisfy conditions~(i),~(ii), and~(iii) yielding an extended structure on the Frobenius algebra $\Bbbk[x]/(x^n)$  when $n$ is odd.
	
It remains to look at the case $\phi(x)=-x$. It follows from $\phi$ being a morphism of coalgebras that this is not possible when $n$ is even, since we get the following contradiction:
\begin{align*}
\textstyle	\sum_{i=0}^{n-1} x^{i}\otimes x^{n-i-1}=\Delta(\phi(1))=(\phi\otimes \phi) \Delta(1)=\sum_{i=0}^{n-1} (-1)^{n-1} x^{i}\otimes x^{n-i-1}=-\sum_{i=0}^{n-1} x^{i}\otimes x^{n-i-1}.
	\end{align*}
When $n$ is odd, the equalities $\phi(\theta)=\theta$ and $\phi(x\theta)=x\theta$ from condition~(ii) yield the equations
\begin{align*}
\textstyle		\sum_{i=0}^{n-1}\theta_ix^i=\sum_{i=0}^{n-1}(-1)^i\theta_ix^i &&\text{ and } &&\textstyle	 \sum_{i=0}^{n-2}\theta_ix^{i+1}=\sum_{i=0}^{n-2}(-1)^{i+1}\theta_ix^{i+1} ,
\end{align*}
respectively. Hence $\theta_i=0$ for $1\leq i \leq n-2$, and we have that $\theta=\theta_{n-1}x^{n-1}$. But then this would imply
$0=\theta^2=m(\phi\otimes \id)\Delta(1)=x^{n-1}$.
Hence, $\phi(x)=-x$ is also not possible when $n$ is odd. 
\end{proof}

For a group $G$, consider the Frobenius algebra $\Bbbk G$ as in Example \ref{ex: KG}. We now provide classification results for the extended structures of $\Bbbk G$ when $G=C_2, C_3, C_4,$ and $C_2\times C_2.$

\begin{proposition} \label{prop:C2-class}
Let $g$ be a generator of $C_2$. The extended structures of $\Bbbk C_2$ are:
\begin{enumerate}[\upshape (a)]
\item $\phi=\id_{\Bbbk C_2}$ and $\theta \in \{\pm \sqrt{2}e_{C_2}, \; \pm \sqrt{2}g\}$, and

\smallskip

\item $\phi(g)=-g$ and $\theta=0$. 
\end{enumerate}
Moreover, $(\Bbbk C_2, \id_{\Bbbk C_2}, \sqrt{2}g)\cong (\Bbbk C_2, \id_{\Bbbk C_2}, -\sqrt{2}g)$ as extended Frobenius algebras, and all other structures are non-isomorphic. That is, there are four isomorphism classes of extended Frobenius structures on $\Bbbk C_2. $
\end{proposition}

\begin{proof}
Suppose that $\phi$ and $\theta$ define an extended structure on $\Bbbk C_2$, with  $\phi(g)=\phi_0 e_{C_2}+\phi_1 g$ and $\theta = \theta_0 e_{C_2} + \theta_1 g$ for $\phi_0, \phi_1, \theta_0, \theta_1 \in \Bbbk$. By the counitality of $\phi$, we have that $\phi_0 =\varep(\phi(g)) = \varep(g) = 0$, and
$\phi_1^2 = \varep(\phi_1^2g^2) = \varep(\phi(g^2))= \varep(g^2)= 1$.
So, $\phi_1=\pm 1$. Both choices are involutions and it is a quick check that they satisfy condition~(i).  We look now at the conditions~(ii) and~(iii).

When $\phi=\id$, we have that $\theta_0^2+\theta_1^2=2 e_{C_2}$ and $2\theta_0\theta_1=0$,
and so either  $\theta=\pm \sqrt{2}  e_{C_2}$ or $\theta=\pm \sqrt{2}g$. Both of these satisfy conditions~(ii) and~(iii).  
When $\phi(g)=-g$, condition~(iii) yields
$\theta_0^2+\theta_1^2=0$ and $2\theta_0\theta_1=0$.
Hence,  $\theta=0$, and condition~(ii) is satisfied in this case. 

 Lastly, it follows from Lemma \ref{lemma: no morph between struct} that an isomorphism can only exist between $(\Bbbk C_2, \id_{\Bbbk C_2}, \sqrt{2}g)$ and $(\Bbbk C_2, \id_{\Bbbk C_2}, -\sqrt{2}g)$, which are in fact isomorphic via the morphism of extended Frobenius algebras $f:\Bbbk C_2\to \Bbbk C_2$ defined by $g\mapsto -g.$
\end{proof}

\begin{proposition}  \label{prop:C3-class}
Let $g$ be a generator of $C_3$. The extended structures of $\Bbbk C_3$ are:
\begin{enumerate}[\upshape (a)]
\item $\phi=\id_{\Bbbk C_3}$ and $\theta \in \{\pm \sqrt{3}e_{C_3}, \; \pm \frac{1}{\sqrt{3}}(e_{C_3}-2\omega_3g-2\omega_3^2g^2)\}$,
\item $\phi(g)=\omega_3 g^2$ and $\theta=\pm \frac{1}{\sqrt{3}}(e_{C_3}+\omega_3 g+\omega_3^2 g^2),$ 
\end{enumerate}
where $\omega_3\in \Bbbk$ is some $3$-rd root of unity. Moreover, these structures are all non-isomorphic.  
\end{proposition}

\begin{proof}
Suppose $\phi$ and $\theta$ define an extended structure of $\Bbbk C_3$, where $\phi(g)= \phi_0 e_{C_3}+ \phi_1 g+ \phi_2 g^2$ and $\theta = \theta_0 e_{C_3}+ \theta_1 g+ \theta_2 g^2$, for $\phi_i, \theta_i \in \Bbbk$. 
By condition~(i), we get that $\phi=\id$ or $\phi(g)=\omega_3 g^2$.  We now examine the conditions: $m(\phi\otimes \id_{\Bbbk C_3})\Delta (e_{C_3})=\theta^2$, and $\phi(\theta a) = \theta a$ for $a \in \Bbbk C_3$. 

When $\phi=\id$, this gives the equation
$\theta^2=3 e_{C_3}$.
Hence, $\theta_0\ne 0,$ and if $\theta_1=0$ or $\theta_2=0$, these imply $\theta=\pm \sqrt{3}e_{C_3}$. Else, if $\theta_1,\theta_2\ne 0,$ it follows that  $\theta=\pm \frac{1}{\sqrt{3}}(e_{C_3}-2\omega_3g-2\omega_3^2g^2)$ for some $3$-rd root of unity $\omega_3$. Condition (ii) is trivially satisfied for these cases. 
When $\phi(g)=\omega_3 g^2$, then condition~(iii) implies that 
$\theta^2 = e_{C_3}+\omega_3 g+\omega_3^2g^2$.
We also require  $\theta=\phi(\theta)=\theta_0 e_{C_3} +\theta_1\omega_3g^2+\theta_2\omega_3^2g$, and thus $\theta_2= \omega_3\theta_1$. Therefore, we get that $\theta=\pm \frac{1}{\sqrt{3}}(e_{C_3}+\omega_3 g+\omega_3^2 g^2)$. One can check that these choices satisfy condition (ii); see Example \ref{example: ext on Cn}. 

 Lastly, any morphism $f$ of extended Frobenius algebras between these possible structures is counital, so  $f(g) = cg$ or $f(g) = cg^2$ for some $c\in \Bbbk$ such that $c^3=1$. From this and  Lemma \ref{lemma: no morph between struct}, we conclude there are no such morphisms between the different extended structures. 
\end{proof}

\begin{proposition}  \label{prop:C4-class}
Let $g$ be a generator of $C_4$. The extended structures of $\Bbbk C_4$ are:
\begin{enumerate}[\upshape (a)]
\item $\phi=\id_{\Bbbk C_4}$ and $\theta\in\{\pm 2e_{C_4}, ~\pm 2g^2,  ~\pm (1-i)(g+ig^3), ~\pm (1+i) (g-ig^3)\}$;
\smallskip
\item $\phi(g)=-g$ and $\theta=0$;

\smallskip

\item $\phi(g)=\omega_4 g^3$ and $\theta\in \left\{\pm \frac{1+\omega_4^2}{2}(e_{C_4}-g^2), \; \pm i \frac{1+\omega_4^2}{2}(g-g^3)\right\}$, 
\end{enumerate}
for any $4$-th root of unity $\omega_4\in \Bbbk$. These form eight isomorphism classes of extended structures. 
\end{proposition}

\begin{proof}
Suppose that $\phi$ and $\theta$ define an extended structure on $\Bbbk C_4$, where for $\phi_i, \theta_i \in \Bbbk$, we have
$\phi(g)= \phi _0 e_{C_3} + \phi_1 g + \phi_2 g^2 + \phi_3g^3$ and $\theta= \theta_0 e_{C_3}+ \theta_1 g+ \theta_2 g^2+ \theta_3 g^3$. 
By condition~(i), we get that $\phi_2=0$ with $\phi(g)=\phi_1 g$ or $\phi(g)=\phi_3g^3$; else,  	 $\phi_2\ne 0$ with $\phi_1^2+\phi_3^2=0$. 
But a routine computation using $\phi^2(g) = g$ and condition~(iii) shows that the $\phi_2 \neq 0$ case is not possible. 
So, either $\phi(g)=\phi_1g$ or $\phi(g)=\phi_3g^3$. Since $\phi^2(g)=g$,  we obtain $\phi(g)=\pm g$ or $\phi(g)=\omega_4 g^3$ for some $\omega_4\in \Bbbk$.

Suppose that $\phi=\id_{\Bbbk C_4}$. Then, condition~(ii) is trivially satisfied. Condition~(iii) implies that $4 e_{C_4}=\theta^2$, and we get the choices for $\theta$ in part (a).  
When $\phi(g)=-g$, condition~(ii) implies that $\theta_1=\theta_3=0$. So, by condition~(iii), we obtain that
$\theta_0^2+2\theta_0\theta_2g^2+\theta_2^2=0$,
and it follows that $\theta=0$. This yields the choice in part~(b).
Lastly, if $\phi(g)=\omega_4g^3$, then from condition~(ii), we know that $\theta_1=\omega_4^3\theta_3$. Also from condition~(iii), we get that 
$\theta^2=(1+\omega_4^2)e_{C_4}+ (\omega_4+\omega_4^3)g^2$.
Solving for $\theta^2$ in $\Bbbk C_4$, we get the two choices for $\theta$ in part (c). The former coincides with the choice of structure given in Example \ref{example: ext on Cn}. For the latter, it is easy to check that condition~(ii) still holds.
		
 We prove now that there are exactly eight isomorphism classes of extended structures. It follows from Lemma \ref{lemma: no morph between struct} that three such classes are given by
\[
\{(\Bbbk C_4, \id_{\Bbbk C_4}, 2e_{C_4})\}, \;\;
\{(\Bbbk C_4, \id_{\Bbbk C_4}, -2e_{C_4})\}, \;\;
\{(\Bbbk C_4, \phi(g)=-g, 0)\}.
\]
Next, there can be no isomorphisms  $f:(\Bbbk C_4, \id_{\Bbbk C_4}, \theta)\to (\Bbbk C_4, \phi(g)=\omega_4g^3, \theta'),$ as this would imply $f(g)=f(\omega_4g^3)$. Also, the algebra isomorphisms  $f,f':\Bbbk C_4\to \Bbbk C_4$  defined by $f(g)=-g$ and $f'(g)=ig$ imply that
\[
\{(\Bbbk C_4,\id_{\Bbbk C_4}, \pm(1-i)(g+ig^3)), \; \; 
\{(\Bbbk C_4,\id_{\Bbbk C_4}, \pm(1+i)(g-ig^3))\},  \; \; 
\{(\Bbbk C_4,\id_{\Bbbk C_4}, \pm 2g^2)\}
\]
are isomorphism classes of extended structures. 
The remaining isomorphism classes  are then 
\[
\{(\Bbbk C_4, \; \phi(g)=\omega_4g^3, \; \textstyle \pm \frac{1+\omega_4^2}{2}(e_{C_4}-g^2))\}, \;\; \{(\Bbbk C_4, \;\phi(g)=\omega_4g^3, \; \textstyle  \pm i\frac{1+\omega_4^2}{2}(g-g^3))\}
\]
by a routine calculation.
\end{proof}

Given the results in Proposition~\ref{prop:C2-class},~\ref{prop:C3-class},~\ref{prop:C4-class}, we propose the following result.

\begin{conjecture}
Let $g$ be a generator of $C_n$. The following are the only  possibilities for the Frobenius automorphism $\phi$ for an extended structure on  $\Bbbk C_n$:
\begin{enumerate}[\upshape (a)]
\item $\phi(g)=\pm g$  or $\phi(g)=\omega_n g^{-1}$ when $n$ is even, 

\smallskip

\item $\phi(g)=g$ or $\phi(g)=\omega_n g^{-1}$ when $n$ is odd, 
\end{enumerate}
where $\omega_n \in \Bbbk$ is any $n$-th root of unity.
\end{conjecture}

The remainder of Theorem~\ref{thm:Frobk-class} is established in the next two results.

\begin{proposition}   \label{prop:V4-class}
The extended structures of $\Bbbk (C_2\times  C_2)$ are:
\begin{enumerate}[\upshape (a)]
\item $\phi=\id_{\Bbbk (C_2 \times C_2)}$ and $\theta \in \{\pm 2e, ~ \pm 2g_i, ~ \pm(e+g_\ell)\pm(g_i-g_j), ~\pm (e-g_\ell) \pm (g_i+g_j)\}$;

\smallskip

\item $\phi(g_i)=-g_i,$ ~$\phi(g_j)=-g_j$, ~$\phi(g_\ell)=g_\ell$, ~ and $\theta=0$; 

\smallskip

\item $\phi(g_i)=g_j$, ~$\phi(g_j)=g_i$, ~$\phi(g_\ell)=g_\ell$,~ and $\theta \in \{\pm(e+g_\ell), ~\pm(g_i+g_j)\}$;

\smallskip

\item $\phi(g_i)=-g_j$, ~$\phi(g_j)=-g_i$, $\phi(g_\ell)=g_\ell$, ~ and $\theta \in \{\pm(e-g_\ell), ~\pm(g_i-g_j)\}$;
\end{enumerate}
where $C_2\times C_2=\{e,g_1, g_2, g_3\}$ and $\{i,j,\ell\}=\{1,2,3\}$. 
\end{proposition}

\begin{proof}
It follows from $\phi$ being counital that $\phi(g_i)=a_{i,1}g_1+a_{i,2}g_2+a_{i,3}g_3$ for $a_{i,p}\in \Bbbk$, for all $1\leq i,p \leq 3$. Since $\phi$ is multiplicative, we then get that
\begin{align*}
e=\phi(g_i^2)=\phi(g_i)^2=(a_{i,1}^2+a_{i,2}^2+a_{i,3}^2)e +2a_{i,1}a_{i,2}g_3+2a_{i,1}a_{i,3}g_2+2a_{i,2}a_{i,3}g_1. 
\end{align*}
Hence, $\phi(g_i)=\pm g_j$ for some $1\leq j \leq 3$. But $\phi^2=\id_{\Bbbk (C_2\times C_2)}$, and thus the remaining possibilities for $\phi$ are the ones listed in the statement.
It remains to find suitable $\theta$ for each possible $\phi$.  Let $\theta_0, \theta_1, \theta_2, \theta_3 \in \Bbbk$ such that $\theta=\theta_0 e + \theta_1g_2+\theta_2g_2+ \theta_3 g_3$. 

We compute 
$\textstyle \theta^2	
= \phi(e)e+\sum_{i=1}^3 \phi(g_i)g_i.$
When $\phi=\id_{\Bbbk (C_2\times C_2)}$, one can check that we get the choices of $\theta$ in part (a) by condition~(iii).
When $\phi(g_i)=-g_i,$ $\phi(g_j)=-g_j$ and $\phi(g_\ell)=g_\ell$ for $\{i,j,\ell\}=\{1,2,3\}$, condition~(iii)  implies $\theta^2=0$, so $\theta=0$; this implies part (b). 
When $\phi(g_i)=g_j,$ $\phi(g_j)=g_i$ and $\phi(g_\ell)=g_\ell$ for $\{i,j,\ell\}=\{1,2,3\}$, 
conditions~(ii) and (iii) yield the choices of $\theta$ in part (c). 
The case $\phi(g_i)=-g_j,$ $\phi(g_j)=-g_i$ and $\phi(g_\ell)=g_\ell$ for $\{i,j,\ell\}=\{1,2,3\}$ is analogous.
\end{proof}   

\begin{proposition}  \label{prop:T2-class}
	Consider the Taft algebra $T_2(-1):=\Bbbk\langle g,x \rangle/(g^2-1, x^2, gx+xg)$ as defined in Example \ref{example:Taft}. All extensions of $T_2(-1)$ are $\phi$-trivial, with $\theta\in \Bbbk x\oplus \Bbbk gx$. 
\end{proposition}

\begin{proof}
First, note that $\Delta(1) = -g \otimes gx + x \otimes 1 + 1 \otimes x + gx \otimes g$. So, by \eqref{eq:DeltaA}, we get that 
$\Delta(g) = -1 \otimes gx + gx \otimes 1 + g \otimes x + x \otimes g$,
$\Delta(x)  = gx \otimes gx  + x \otimes x$, and $\Delta(gx) = x \otimes gx  + gx \otimes x$.
Hence, $\varep(1) = \varep(g) = \varep(gx) =0$ and $\varep(x) = 1$. Now 
suppose that $\phi: T_2(-1) \to T_2(-1)$ and $\theta \in T_2(-1)$ define an extended structure on $T_2(-1)$. Let $a_i, b_i \in \Bbbk$ such that $\phi(g)=a_1+a_2g+a_3x+a_4gx$ and  $\phi(x)=b_1+b_2g+b_3x+b_4gx$. Since $\phi$ is an algebra morphism, we have that 
\begin{align*}	1&=\phi(g)^2=a_1^2+a_2^2+2a_1a_2g+2a_1a_3x+2a_1a_4gx, \\	0&=\phi(x)^2=b_1^2+b_2^2+2b_1b_2g+2b_1b_3x+2b_1b_4gx.
\end{align*}
It follows that $\phi(g) =\pm g +a_3 x+a_4gx$ and $\phi(x)=b_3x+b_4gx$. On the other hand, since $\phi$ is counital, we  get
$0=\varep(\phi(g))=a_3$ and $1=\varep(\phi(x))=b_3$.
So, $\phi(g) =\pm g +a_4gx$ and $\phi(x)=x+b_4gx$. Also, $\phi$ is an involution, hence
$g=\phi(\pm g +a_4gx)= \pm ( g +a_4gx) \pm a_4(gx+b_4x)$.
It follows that $\phi=\id_{T_2(-1)}$. Lastly, 
$\theta^2=	m(\phi\otimes \id_{T_2(-1)})\Delta(1)=0$,
and thus $\theta\in \Bbbk x\oplus \Bbbk gx$.
\end{proof}

\begin{conjecture}  \label{conj:Tn-class}
Consider the Taft algebra, 
$T_n(\omega):=\Bbbk\langle g,x \rangle/(g^n-1, x^n, gx - \omega xg)$ from Example~\ref{example:Taft}. Then, all extensions of $T_n(\omega)$ are $\phi$-trivial, with $\theta\in \Bbbk x \oplus \Bbbk gx \oplus \cdots \oplus \Bbbk g^{n-1} x$. 
\end{conjecture}




\medskip


\section{Extended Frobenius algebras in a monoidal category}
\label{sec:pre-mon}

In this section, we first discuss monoidal categories and algebraic structures in monoidal categories in Section~\ref{sec:monbackground}. There, we generalize Definition~$\ref{def:extFrob-k}$ to the monoidal setting, following \cite[Section 2.2]{TuraevTurner}; see Definition~\ref{def:extFrobC}. Finally, we put monoidal structures on the category of extended Frobenius algebras in Section~\ref{sec:oper}. 

\subsection{Background material}\label{sec:monbackground}
For details on algebras in monoidal categories, see, for example, \cite[Chapter 3]{Kock}, \cite[Parts I and II]{TuraevVirelizier} or \cite[Chapters 3 and~4]{Walton}. The first reference also includes an introduction to Frobenius algebras in monoidal categories. Extended Frobenius algebras in monoidal categories can be found in \cite[Section 2.2]{TuraevTurner}, \cite{Czenky}, and \cite{Ocal}.

\subsubsection{Monoidal categories} A {\it monoidal category} consists of a category $\cC$ equipped with a bifunctor $\otimes:  \cC \times \cC \to \cC$,  a natural isomorphism $a:=\{a_{X,Y,Z}:  (X \otimes Y) \otimes Z \hs \equivto X \otimes (Y \otimes Z)\}_{X,Y,Z \in \cC}$, an object $\one \in \cC$, and  natural isomorphisms $\ell:=\{\ell_X:  \one \otimes X \hs \equivto X\}_{X \in \cC}$ and $r:=\{r_X : X  \otimes \one \hs \equivto X\}_{X \in \cC}$, 
such that the pentagon and triangle axioms hold.

\smallskip

Unless stated otherwise, by MacLane's strictness theorem, we will assume that all monoidal categories are {\it strict} in the sense that 
\[
X \otimes Y \otimes Z := (X \otimes Y) \otimes Z = X \otimes (Y \otimes Z), \quad \quad
X := \one \otimes X = X \otimes \one,
\]
 for all $X, Y, Z \in \cC$; that is, $a_{X,Y,Z},\; \ell_X,\; r_X$ are identity maps.

\smallskip

A monoidal category $\cC$ is {\it symmetric} if it is equipped with $c:=\{c_{X,Y}:  X \otimes Y \overset{\sim}{\to} Y \otimes X\}_{X,Y \in \cC}$,  a natural isomorphism  with $c_{Y,X} \circ c_{X,Y} = \id_{X \otimes Y}$ for $X,Y \in \cC$,  such that the  hexagon axioms hold. The component $c_{X,Y}$ of $c$, the $c^2 = \id$ property, the naturality of $c$ at a morphism $f \in \cC$, and unit coherence of $c$ are all depicted in Figure~\ref{fig:symmon}.

\medskip


\begin{figure}[h!]
\scalebox{0.6}{
\tikzset{every picture/.style={line width=0.75pt}} 
\begin{tikzpicture}[x=0.75pt,y=0.75pt,yscale=-1,xscale=1]

\draw [line width=2.25]    (31.25,30.38) .. controls (31.14,70.16) and (78.56,60.54) .. (78.34,100.68) ;
\draw [line width=2.25]    (78.11,30.35) .. controls (78.56,70.16) and (31.37,59.87) .. (31.82,100.35) ;
\draw [line width=2.25]    (156.59,31.7) .. controls (156.47,51.49) and (203.9,46.7) .. (203.67,66.68) ;
\draw [line width=2.25]    (203.45,31.68) .. controls (203.9,51.49) and (156.7,46.37) .. (157.15,66.51) ;
\draw [line width=2.25]    (157.25,66.37) .. controls (157.14,86.16) and (204.56,81.37) .. (204.34,101.35) ;
\draw [line width=2.25]    (203.67,66.68) .. controls (204.12,86.49) and (156.93,81.37) .. (157.38,101.51) ;
\draw [line width=2.25]    (249.33,30.92) -- (249.67,100.68) ;
\draw [line width=2.25]    (269.33,30.92) -- (269.67,100.68) ;
\draw [line width=2.25]    (346.25,30.05) .. controls (346.14,69.83) and (393.56,60.2) .. (393.34,100.35) ;
\draw [line width=2.25]    (393.11,30.02) .. controls (393.56,69.83) and (346.37,59.54) .. (346.82,100.01) ;
\draw [line width=2.25]    (426.25,30.05) .. controls (426.14,69.83) and (473.56,60.2) .. (473.34,100.35) ;
\draw [line width=2.25]    (473.11,30.02) .. controls (473.56,69.83) and (426.37,59.54) .. (426.82,100.01) ;
\draw [line width=2.25]    (544.25,29.72) .. controls (544.14,69.49) and (591.56,59.87) .. (591.34,100.01) ;
\draw [line width=2.25]    (591.11,29.68) .. controls (591.56,69.49) and (544.37,59.21) .. (544.82,99.68) ;
\draw [line width=2.25]    (624.25,29.72) .. controls (624.14,69.49) and (671.56,59.87) .. (671.34,100.01) ;
\draw [line width=2.25]    (671.11,29.68) .. controls (671.56,69.49) and (624.37,59.21) .. (624.82,99.68) ;
\draw   (9.67,10.5) -- (1050.6,10.5) -- (1050.6,151.01) -- (9.67,151.01) -- cycle ;
\draw [line width=2.25]  [dash pattern={on 6.75pt off 4.5pt}]  (742.05,29.98) .. controls (741.94,69.76) and (789.36,60.14) .. (789.14,100.28) ;
\draw [line width=2.25]    (788.91,29.95) .. controls (789.36,69.76) and (742.17,59.47) .. (742.62,99.95) ;
\draw [line width=2.25]    (826.33,29.92) -- (826.67,99.68) ;
\draw [line width=2.25]  [dash pattern={on 6.75pt off 4.5pt}]  (846.33,29.92) -- (846.67,99.68) ;
\draw [line width=2.25]    (921.25,29.98) .. controls (921.14,69.76) and (968.56,60.14) .. (968.34,100.28) ;
\draw [line width=2.25]  [dash pattern={on 6.75pt off 4.5pt}]  (968.11,29.95) .. controls (968.56,69.76) and (921.37,59.47) .. (921.82,99.95) ;
\draw [line width=2.25]  [dash pattern={on 6.75pt off 4.5pt}]  (1005.53,29.92) -- (1005.87,99.68) ;
\draw [line width=2.25]    (1025.53,29.92) -- (1025.87,99.68) ;

\draw (219.5,60) node [anchor=north west][inner sep=0.75pt]  [font=\large] [align=left] {=};
\draw (403.5,60) node [anchor=north west][inner sep=0.75pt]  [font=\large] [align=left] {=};
\draw (600.5,60) node [anchor=north west][inner sep=0.75pt]  [font=\large] [align=left] {=};
\draw (39,113.89) node [anchor=north west][inner sep=0.75pt] [font=\large]  [align=left] {$\displaystyle c_{X,Y}$ \ \ \ \ \ \ \ \ \ \ \ \ \ \ \ \ \ \ \ \ \ \ \ \ \ \   (S1) \ \ \ \ \ \ \ \ \ \ \ \ \ \ \ \ \ \ \ \ \  \ \ \ \ \ \ \ \ \ (S2) \ \ \ \ \ \ \ \ \ \ \ \ \ \ \ \ \ \ \ \ \ \ \ \ \ \ \ \ \ \ \  (S3) \  \ \ \ \ \ \ \ \ \ \ \ \ \ \ \ \ \ \ \ \ \ \ \ \  \ \ \ \ \ \ (S4) \ \ \ \ \ \ \ \ \ \ \ \ \ \   \ \ \ \ \ \ \ \ \ \ \ \ \ \ (S5)};
\draw  [fill={rgb, 255:red, 255; green, 255; blue, 255 }  ,fill opacity=1 ][line width=1.5]   (344.96,38.4) -- (360.96,38.4) -- (360.96,56.4) -- (344.96,56.4) -- cycle  ;
\draw (358,40.5) node [anchor=north east] [inner sep=0.75pt]  [font=\small] [align=left] {$\displaystyle f$};
\draw  [fill={rgb, 255:red, 255; green, 255; blue, 255 }  ,fill opacity=1 ][line width=1.5]   (458.56,74.4) -- (474.56,74.4) -- (474.56,92.4) -- (458.56,92.4) -- cycle  ;
\draw (472,76.5) node [anchor=north east] [inner sep=0.75pt]  [font=\small] [align=left] {$\displaystyle f$};
\draw  [fill={rgb, 255:red, 255; green, 255; blue, 255 }  ,fill opacity=1 ][line width=1.5]   (576.96,36.8) -- (592.96,36.8) -- (592.96,54.8) -- (576.96,54.8) -- cycle  ;
\draw (590,38.5) node [anchor=north east] [inner sep=0.75pt]  [font=\small] [align=left] {$\displaystyle f$};
\draw  [fill={rgb, 255:red, 255; green, 255; blue, 255 }  ,fill opacity=1 ][line width=1.5]   (622.96,74.4) -- (638.96,74.4) -- (638.96,92.4) -- (622.96,92.4) -- cycle  ;
\draw (636,76.5) node [anchor=north east] [inner sep=0.75pt]  [font=\small] [align=left] {$\displaystyle f$};
\draw (798,60) node [anchor=north west][inner sep=0.75pt]  [font=\large] [align=left] {=};
\draw (979,60) node [anchor=north west][inner sep=0.75pt]  [font=\large] [align=left] {=};
\end{tikzpicture}
}
\vspace{-.3in}
\caption{Some axioms for a symmetric monoidal category.}
\label{fig:symmon}
\end{figure}
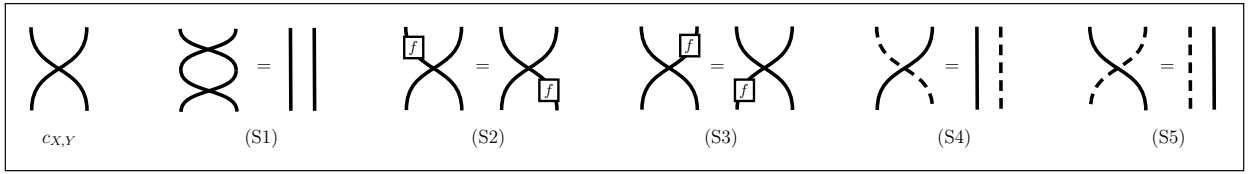




\subsubsection{Algebraic structures in monoidal categories} \label{sec:algstr} Take a monoidal category $\cC:=(\cC, \otimes, \one)$. 

\smallskip

An {\it algebra in $\cC$} is an object $A \in \cC$, equipped with morphisms $m: A \otimes A \to A$ and $u: \one \to A$ in~$\cC$, subject to associativity and unitality axioms:
\[
m(m \otimes \id_A) = m(\id_A \otimes m), \qquad m(u \otimes \id_A) = \id_A = m(\id_A \otimes u).
\]
These structures form a category, $\mathsf{Alg}(\cC)$, where a morphism $(A,m_A,u_A) \to (B,m_B,u_B)$ is a morphism $f:A \to B$ in $\cC$ such that $f \hs m_A = m_B(f \otimes f)$ and $f \hs u_A = u_B$. 

\smallskip

A {\it coalgebra in $\cC$} is an object $A \in \cC$, equipped with morphisms $\Delta: A \to A \otimes A$ and $\varep: A \to \one$ in~$\cC$, subject to coassociativity and counitality axioms:
\[
(\Delta \otimes \id_A)\Delta = (\id_A \otimes \Delta)\Delta, \qquad (\varep \otimes \id_A)\Delta = \id_A = \varep(\id_A \otimes u)\Delta.
\]
These structures form a category, $\mathsf{Coalg}(\cC)$, where a morphism $(A,\Delta_A,\varep_A) \to (B,\Delta_B,\varep_B)$ is a morphism $f:A \to B$ in $\cC$ such that $\Delta_B \hs f = (f \otimes f) \Delta_A$ and $\varep_B \hs f = \varep_A$.

\medskip

Our main algebraic structures of interest in this article are given as follows.

\begin{definition}\label{def:extFrobC} Consider the following constructions in a monoidal category $\cC:=(\cC, \otimes, \one)$.
\begin{enumerate}[\upshape (a)]
\item A {\it Frobenius algebra in $\cC$} is a tuple $(A,m,u,\Delta, \varep)$, where $(A,m,u)$ is an algebra in $\cC$, and $(A,\Delta,\varep)$ is a coalgebra in $\cC$, subject to the Frobenius law:
\[
(m \otimes \id_A)(\id_A \otimes \Delta) = \Delta m = (\id_A \otimes m)(\Delta \otimes \id_A).
\]
A {\it morphism of Frobenius algebras in $\cC$} is a morphism of the underlying algebras and coalgebras in $\cC$. 
The above objects and morphisms form a category, $\mathsf{FrobAlg}(\cC)$.
\smallskip
\item An {\it extended Frobenius algebra in $\cC$} is a tuple $(A,m,u,\Delta, \varep, \phi, \theta)$, where $(A,m,u,\Delta,\varep)$ is a Frobenius algebra in $\cC$, and $\phi:A \to A$ and $\theta: \one \to A$ are morphisms in $\cC$ such that
\smallskip
\begin{itemize}
    \item[(i)] $\phi$ is a morphism of Frobenius algebras in $\cC$, with $\phi^2 = \id_A$;
    \smallskip
    \item[(ii)] $\phi \hs m ( \theta \otimes \id_A) = m ( \theta \otimes \id_A )$;
    \smallskip
    \item[(iii)] $m (\phi \otimes \id_A) \Delta u = m (\theta \otimes \theta)$. 
\end{itemize}
\smallskip
A {\it morphism $f:(A,\phi_A,\theta_A) \to (B,\phi_B,\theta_B)$ of extended Frobenius algebras in $\cC$} is a morphism $f: A \to B$ of Frobenius algebras  in $\cC$, such that $f \hs \phi_A = \phi_B \hs f$ and $f \hs \theta_A = \theta_B$.
The above objects and morphisms form a category, $\mathsf{ExtFrobAlg}(\cC)$.
\smallskip
\item The morphisms $\phi$ and $\theta$ in part (b) are the  {\it extended structure} of the underlying Frobenius algebra. When $\phi$ and $\theta$ exist, we say that the underlying Frobenius algebra is {\it extendable}.
\smallskip
\item An extended structure $(\phi, \theta)$ on a Frobenius algebra $A$ is said to be  {\it $\phi$-trivial} if $\phi$ is the identity morphism, and is {\it $\theta$-trivial} if $\theta$ is the zero morphism (when these exist in $\cC$).
\end{enumerate}
\end{definition}

The structure morphisms for an extended Frobenius algebra in $\cC$ are depicted in Figure~\ref{fig:eFrob}, and the axioms that they satisfy are depicted in Figure~\ref{fig:eFrobax}. Here, we read diagrams from top down.

\smallskip


\begin{figure}[h!]
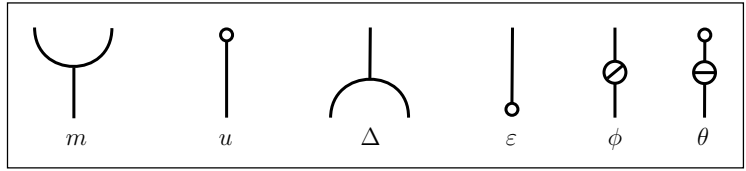


\scalebox{0.65}{

\tikzset{every picture/.style={line width=0.75pt}} 


}
\vspace{-.1in}
\caption{Structure morphisms for an extended Frobenius algebra in $\cC$.}
\label{fig:eFrob}
\vspace{0.05in}
\end{figure}



\begin{figure}[h!]
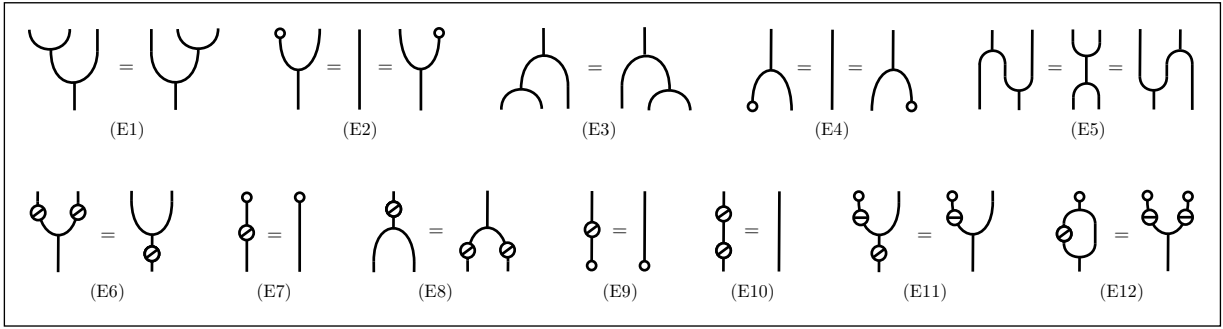


\scalebox{0.68}{

\tikzset{every picture/.style={line width=0.75pt}} 

%
}
\vspace{-.3in}
\caption{Axioms for an extended Frobenius algebra in $\cC$.}
\label{fig:eFrobax}
\vspace{0.05in}
\end{figure}

One useful lemma is the following,  adapted from \cite[Lemma~2.8]{TuraevTurner} for the monoidal setting.
\begin{lemma} \label{lem:EFrobid}
 If $(A, m, u, \Delta, \varep, \phi, \theta)$ is an extended Frobenius algebra in $\cC$, then \[
  m (\phi \otimes \id_A) \Delta \;  = \; m\bigl(m(\theta \otimes \theta) \otimes \id_A\bigr).\] 
\end{lemma}
\begin{proof}
This is proved in Figure~\ref{fig:eFrobid} with references to Figures~\ref{fig:eFrob}~and~\ref{fig:eFrobax}.
\end{proof}

\pagebreak


\begin{figure}[h!]

\scalebox{0.85}{
\tikzset{every picture/.style={line width=0.75pt}} 
\begin{tikzpicture}[x=0.75pt,y=0.75pt,yscale=-1,xscale=1]

\draw   (11.5,10.5) -- (480,10.5) -- (480,120.52) -- (11.5,120.52) -- cycle ;
\draw  [line width=1.5]  (442.25,34.29) .. controls (442.24,36.14) and (440.73,37.63) .. (438.88,37.62) .. controls (437.03,37.61) and (435.54,36.11) .. (435.55,34.26) .. controls (435.56,32.41) and (437.06,30.92) .. (438.91,30.92) .. controls (440.76,30.93) and (442.25,32.44) .. (442.25,34.29) -- cycle ;
\draw [line width=1.5]    (414.77,37.01) .. controls (414.17,68.76) and (439.17,68.26) .. (438.88,37.62) ;
\draw [line width=1.5]    (459.67,29.76) -- (459.67,65.26) ;
\draw  [fill={rgb, 255:red, 255; green, 255; blue, 255 }  ,fill opacity=1 ][line width=1.5]  (409.72,47.18) .. controls (409.72,44.19) and (412.09,41.76) .. (415.03,41.76) .. controls (417.96,41.76) and (420.33,44.19) .. (420.33,47.18) .. controls (420.33,50.18) and (417.96,52.61) .. (415.03,52.61) .. controls (412.09,52.61) and (409.72,50.18) .. (409.72,47.18) -- cycle ;
\draw [line width=1.5]    (409.72,47.18) -- (420.33,47.18) ;
\draw  [line width=1.5]  (418.13,33.68) .. controls (418.13,35.53) and (416.62,37.02) .. (414.77,37.01) .. controls (412.92,37) and (411.43,35.5) .. (411.44,33.65) .. controls (411.45,31.8) and (412.95,30.31) .. (414.8,30.31) .. controls (416.65,30.32) and (418.14,31.83) .. (418.13,33.68) -- cycle ;
\draw [line width=1.5]    (36.52,50.18) .. controls (36.59,37.18) and (59.72,36.58) .. (59.72,50.58) ;
\draw [line width=1.5]    (36.52,79.48) .. controls (36.49,93.38) and (60.12,93.08) .. (59.72,79.48) ;
\draw [line width=1.5]    (36.52,79.48) -- (36.52,50.18) ;
\draw [line width=1.5]    (59.72,79.48) -- (59.72,50.58) ;
\draw [line width=1.5]    (48.05,100.54) -- (48.12,89.94) ;
\draw  [fill={rgb, 255:red, 255; green, 255; blue, 255 }  ,fill opacity=1 ][line width=1.5]  (31.42,64.52) .. controls (31.42,61.52) and (33.79,59.09) .. (36.73,59.09) .. controls (39.66,59.09) and (42.03,61.52) .. (42.03,64.52) .. controls (42.03,67.51) and (39.66,69.94) .. (36.73,69.94) .. controls (33.79,69.94) and (31.42,67.51) .. (31.42,64.52) -- cycle ;
\draw [line width=1.5]    (32.45,67.64) -- (41.21,61.14) ;
\draw  [fill={rgb, 255:red, 255; green, 255; blue, 255 }  ,fill opacity=1 ][line width=1.5]  (31.42,64.52) .. controls (31.42,61.52) and (33.79,59.09) .. (36.73,59.09) .. controls (39.66,59.09) and (42.03,61.52) .. (42.03,64.52) .. controls (42.03,67.51) and (39.66,69.94) .. (36.73,69.94) .. controls (33.79,69.94) and (31.42,67.51) .. (31.42,64.52) -- cycle ;
\draw [line width=1.5]    (32.45,67.64) -- (41.21,61.14) ;
\draw  [fill={rgb, 255:red, 255; green, 255; blue, 255 }  ,fill opacity=1 ][line width=1.5]  (433.22,47.68) .. controls (433.22,44.69) and (435.59,42.26) .. (438.53,42.26) .. controls (441.46,42.26) and (443.83,44.69) .. (443.83,47.68) .. controls (443.83,50.68) and (441.46,53.11) .. (438.53,53.11) .. controls (435.59,53.11) and (433.22,50.68) .. (433.22,47.68) -- cycle ;
\draw [line width=1.5]    (433.22,47.68) -- (443.83,47.68) ;
\draw [line width=1.5]    (48.05,40.71) -- (48.12,30.11) ;
\draw [line width=1.5]    (426.67,64.76) .. controls (426.67,83.76) and (459.67,84.76) .. (459.67,65.26) ;
\draw [line width=1.5]    (443.67,79.76) -- (443.67,100.26) ;
\draw [line width=1.5]    (426.67,64.76) -- (426.73,60.79) ;
\draw [line width=1.5]    (126.45,64.79) .. controls (126.51,51.79) and (149.65,51.19) .. (149.65,65.19) ;
\draw [line width=1.5]    (126.19,79.81) .. controls (126.15,93.71) and (149.79,93.41) .. (149.39,79.81) ;
\draw [line width=1.5]    (126.19,79.81) -- (126.45,64.79) ;
\draw [line width=1.5]    (149.39,79.81) -- (149.65,65.19) ;
\draw [line width=1.5]    (137.72,100.88) -- (137.79,90.28) ;
\draw  [fill={rgb, 255:red, 255; green, 255; blue, 255 }  ,fill opacity=1 ][line width=1.5]  (120.75,71.85) .. controls (120.75,68.85) and (123.13,66.42) .. (126.06,66.42) .. controls (128.99,66.42) and (131.37,68.85) .. (131.37,71.85) .. controls (131.37,74.84) and (128.99,77.27) .. (126.06,77.27) .. controls (123.13,77.27) and (120.75,74.84) .. (120.75,71.85) -- cycle ;
\draw [line width=1.5]    (121.79,74.98) -- (130.54,68.47) ;
\draw  [fill={rgb, 255:red, 255; green, 255; blue, 255 }  ,fill opacity=1 ][line width=1.5]  (120.75,71.85) .. controls (120.75,68.85) and (123.13,66.42) .. (126.06,66.42) .. controls (128.99,66.42) and (131.37,68.85) .. (131.37,71.85) .. controls (131.37,74.84) and (128.99,77.27) .. (126.06,77.27) .. controls (123.13,77.27) and (120.75,74.84) .. (120.75,71.85) -- cycle ;
\draw [line width=1.5]    (121.79,74.98) -- (130.54,68.47) ;
\draw [line width=1.5]    (138.05,55.38) -- (138.11,48.45) ;
\draw [line width=1.5]    (126.19,37.48) .. controls (126.15,51.38) and (149.79,51.08) .. (149.39,37.48) ;
\draw [line width=1.5]    (149.39,37.48) -- (149.45,30.55) ;
\draw  [line width=1.5]  (129.55,34.14) .. controls (129.54,35.99) and (128.04,37.49) .. (126.19,37.48) .. controls (124.34,37.47) and (122.85,35.96) .. (122.85,34.11) .. controls (122.86,32.26) and (124.37,30.77) .. (126.22,30.78) .. controls (128.07,30.79) and (129.56,32.3) .. (129.55,34.14) -- cycle ;
\draw [line width=1.5]    (217.17,54.79) .. controls (217.24,41.79) and (232.9,41.2) .. (232.9,55.2) ;
\draw [line width=1.5]    (217.19,79.48) .. controls (217.15,93.38) and (240.84,93.12) .. (240.39,79.48) ;
\draw [line width=1.5]    (217.19,79.48) -- (217.17,54.79) ;
\draw [line width=1.5]    (240.39,79.48) -- (240.5,71.45) ;
\draw [line width=1.5]    (228.72,100.54) -- (228.79,89.94) ;
\draw  [fill={rgb, 255:red, 255; green, 255; blue, 255 }  ,fill opacity=1 ][line width=1.5]  (211.75,66.85) .. controls (211.75,63.85) and (214.13,61.42) .. (217.06,61.42) .. controls (219.99,61.42) and (222.37,63.85) .. (222.37,66.85) .. controls (222.37,69.84) and (219.99,72.27) .. (217.06,72.27) .. controls (214.13,72.27) and (211.75,69.84) .. (211.75,66.85) -- cycle ;
\draw [line width=1.5]    (212.79,69.98) -- (221.54,63.47) ;
\draw  [fill={rgb, 255:red, 255; green, 255; blue, 255 }  ,fill opacity=1 ][line width=1.5]  (211.75,66.85) .. controls (211.75,63.85) and (214.13,61.42) .. (217.06,61.42) .. controls (219.99,61.42) and (222.37,63.85) .. (222.37,66.85) .. controls (222.37,69.84) and (219.99,72.27) .. (217.06,72.27) .. controls (214.13,72.27) and (211.75,69.84) .. (211.75,66.85) -- cycle ;
\draw [line width=1.5]    (212.79,69.98) -- (221.54,63.47) ;
\draw [line width=1.5]    (232.84,62.12) -- (232.9,55.2) ;
\draw [line width=1.5]    (232.84,62.12) .. controls (232.8,76.02) and (248.57,75.39) .. (248.17,61.79) ;
\draw [line width=1.5]    (225.39,44.48) -- (225.5,38.45) ;
\draw  [line width=1.5]  (228.87,35.12) .. controls (228.86,36.97) and (227.35,38.46) .. (225.5,38.45) .. controls (223.65,38.45) and (222.16,36.94) .. (222.17,35.09) .. controls (222.18,33.24) and (223.68,31.75) .. (225.53,31.76) .. controls (227.38,31.77) and (228.88,33.27) .. (228.87,35.12) -- cycle ;
\draw [line width=1.5]    (248.17,61.79) -- (248.17,30.79) ;
\draw [line width=1.5]    (317.17,53.45) .. controls (317.24,40.45) and (332.9,39.87) .. (332.9,53.87) ;
\draw [line width=1.5]    (326,83.12) .. controls (326.34,94.79) and (348.34,94.12) .. (348,82.45) ;
\draw [line width=1.5]    (317,69.45) -- (317.17,53.45) ;
\draw [line width=1.5]    (326,83.12) -- (325.84,78.45) ;
\draw [line width=1.5]    (336.67,99.79) -- (336.67,91.12) ;
\draw  [fill={rgb, 255:red, 255; green, 255; blue, 255 }  ,fill opacity=1 ][line width=1.5]  (311.75,59.98) .. controls (311.75,56.99) and (314.13,54.56) .. (317.06,54.56) .. controls (319.99,54.56) and (322.37,56.99) .. (322.37,59.98) .. controls (322.37,62.98) and (319.99,65.41) .. (317.06,65.41) .. controls (314.13,65.41) and (311.75,62.98) .. (311.75,59.98) -- cycle ;
\draw [line width=1.5]    (312.79,63.11) -- (321.54,56.61) ;
\draw  [fill={rgb, 255:red, 255; green, 255; blue, 255 }  ,fill opacity=1 ][line width=1.5]  (311.75,59.98) .. controls (311.75,56.99) and (314.13,54.56) .. (317.06,54.56) .. controls (319.99,54.56) and (322.37,56.99) .. (322.37,59.98) .. controls (322.37,62.98) and (319.99,65.41) .. (317.06,65.41) .. controls (314.13,65.41) and (311.75,62.98) .. (311.75,59.98) -- cycle ;
\draw [line width=1.5]    (312.79,63.11) -- (321.54,56.61) ;
\draw [line width=1.5]    (333,69.45) -- (332.9,53.87) ;
\draw [line width=1.5]    (317,69.45) .. controls (317,81.12) and (333,82.45) .. (333,69.45) ;
\draw [line width=1.5]    (325.39,43.14) -- (325.5,37.12) ;
\draw  [line width=1.5]  (328.87,33.79) .. controls (328.86,35.64) and (327.35,37.13) .. (325.5,37.12) .. controls (323.65,37.11) and (322.16,35.61) .. (322.17,33.76) .. controls (322.18,31.91) and (323.68,30.42) .. (325.53,30.42) .. controls (327.38,30.43) and (328.88,31.94) .. (328.87,33.79) -- cycle ;
\draw [line width=1.5]    (348,82.45) -- (348.17,29.45) ;

\draw (75.93,46.97) node [anchor=north west][inner sep=0.75pt]   [align=left] {{\scriptsize (E2) \ \ \ \ \ \ \ \ \ \  \hs \ \ \ \ \ \ \ (E5) \ \ \ \ \ \ \ \ \ \  \ \ \ \ \ \ \ \ \ \ \ (E1) \ \ \ \ \ \ \ \  \ \ \ \ \hs \ \ \ \ \ (E12)}};
\draw (375.05,62) node [anchor=north west][inner sep=0.75pt]   [align=left] {=};
\draw (80.17,62) node [anchor=north west][inner sep=0.75pt]   [align=left] {=};
\draw (173,62) node [anchor=north west][inner sep=0.75pt]   [align=left] {=};
\draw (279,62) node [anchor=north west][inner sep=0.75pt]   [align=left] {=};
\end{tikzpicture}
}
\vspace{-.1in}
\caption{Proof of Lemma~\ref{lem:EFrobid}.}
\label{fig:eFrobid}
\end{figure}
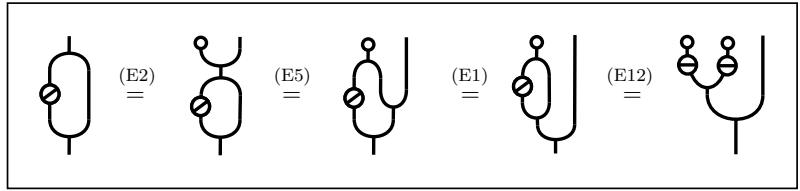

\begin{proposition} \label{prop:FrobHom}
  A morphism of extended Frobenius algebras in  $\cC$ must be an isomorphism.
\end{proposition}

\begin{proof}
 This follows from the well-known fact that a morphism of Frobenius algebras  in $\cC$ must be an isomorphism. We repeat the proof here for the reader's convenience. Take a morphism of (extended) Frobenius algebras $f: A \to B$ in $\cC$, that is, $f$ is a morphism of the underlying algebras and coalgebras in $\cC$. In graphical calculus, we will denote the (extended) Frobenius structure morphisms on $A$ by those given in Figure \ref{fig:eFrob}, and the (extended) Frobenius structure morphisms on $B$ will be denoted according to Figure \ref{fig:StructMorphAB}.  We then define a morphism $g:B\to A$ in Figure \ref{fig:defg}, and show that $gf = \id_A$ and $fg = \id_B$ using graphical calculus in Figure \ref{fig:FrobHomProof}. 
\begin{figure}[h!]
\begin{minipage}{.61\textwidth}
\centering
\scalebox{.87}{
\tikzset{every picture/.style={line width=0.75pt}} 

\begin{tikzpicture}[x=0.75pt,y=0.75pt,yscale=-1,xscale=1]

\draw [line width=1.25]    (30,53) .. controls (30,82) and (61.5,81.5) .. (61,52.5) ;
\draw [line width=1.25]    (46,74) -- (46,92.5) ;
\draw [line width=1.25]    (99.44,43.5) -- (99.44,83) ;
\draw [line width=1.25]    (170,96) .. controls (170.02,67) and (139.02,66.98) .. (139.5,95.98) ;
\draw [line width=1.25]    (154.01,73.5) -- (154,54.5) ;
\draw [line width=1.25]    (210.57,105.5) -- (210.55,66) ;
\draw [line width=1.25]    (261,52.38) -- (261,93.63) ;
\draw [line width=1.25]    (310.5,43.63) -- (310.5,83.13) ;
\draw   (10,20) -- (340.5,20) -- (340.5,150) -- (10,150) -- cycle ;
\draw  [line width=1.25]  (203.13,74.75) .. controls (203.13,70.64) and (206.46,67.31) .. (210.56,67.31) .. controls (214.67,67.31) and (218,70.64) .. (218,74.75) .. controls (218,78.85) and (214.67,82.18) .. (210.56,82.18) .. controls (206.46,82.18) and (203.13,78.85) .. (203.13,74.75) -- cycle ;
\draw  [line width=1.25]  (146.57,74.5) .. controls (146.57,70.39) and (149.9,67.07) .. (154.01,67.07) .. controls (158.11,67.07) and (161.44,70.39) .. (161.44,74.5) .. controls (161.44,78.61) and (158.11,81.94) .. (154.01,81.94) .. controls (149.9,81.94) and (146.57,78.61) .. (146.57,74.5) -- cycle ;
\draw  [line width=1.25]  (92,74.56) .. controls (92,70.46) and (95.33,67.13) .. (99.44,67.13) .. controls (103.54,67.13) and (106.87,70.46) .. (106.87,74.56) .. controls (106.87,78.67) and (103.54,82) .. (99.44,82) .. controls (95.33,82) and (92,78.67) .. (92,74.56) -- cycle ;
\draw  [line width=1.25]  (38.56,75.5) .. controls (38.56,71.4) and (41.89,68.07) .. (46,68.07) .. controls (50.11,68.07) and (53.44,71.4) .. (53.44,75.5) .. controls (53.44,79.61) and (50.11,82.94) .. (46,82.94) .. controls (41.89,82.94) and (38.56,79.61) .. (38.56,75.5) -- cycle ;
\draw  [fill={rgb, 255:red, 255; green, 255; blue, 255 }  ,fill opacity=1 ][line width=1.25]  (95.93,43.5) .. controls (95.93,41.57) and (97.5,40) .. (99.44,40) .. controls (101.37,40) and (102.94,41.57) .. (102.94,43.5) .. controls (102.94,45.44) and (101.37,47.01) .. (99.44,47.01) .. controls (97.5,47.01) and (95.93,45.44) .. (95.93,43.5) -- cycle ;
\draw  [fill={rgb, 255:red, 255; green, 255; blue, 255 }  ,fill opacity=1 ][line width=1.25]  (307,43.63) .. controls (307,41.69) and (308.57,40.13) .. (310.5,40.13) .. controls (312.43,40.13) and (314,41.69) .. (314,43.63) .. controls (314,45.56) and (312.43,47.13) .. (310.5,47.13) .. controls (308.57,47.13) and (307,45.56) .. (307,43.63) -- cycle ;
\draw  [fill={rgb, 255:red, 255; green, 255; blue, 255 }  ,fill opacity=1 ][line width=1.25]  (207.07,106.5) .. controls (207.07,104.56) and (208.64,102.99) .. (210.57,102.99) .. controls (212.51,102.99) and (214.08,104.56) .. (214.08,106.5) .. controls (214.08,108.43) and (212.51,110) .. (210.57,110) .. controls (208.64,110) and (207.07,108.43) .. (207.07,106.5) -- cycle ;
\draw  [fill={rgb, 255:red, 255; green, 255; blue, 255 }  ,fill opacity=1 ][line width=1.25]  (254.13,74.56) .. controls (254.13,70.46) and (257.46,67.13) .. (261.56,67.13) .. controls (265.67,67.13) and (269,70.46) .. (269,74.56) .. controls (269,78.67) and (265.67,82) .. (261.56,82) .. controls (257.46,82) and (254.13,78.67) .. (254.13,74.56) -- cycle ;
\draw [fill={rgb, 255:red, 255; green, 255; blue, 255 }  ,fill opacity=1 ][line width=1.25]    (256.62,78.68) -- (267,70) ;

\draw  [line width=1.25]  (258.06,74.56) .. controls (258.06,72.63) and (259.63,71.06) .. (261.56,71.06) .. controls (263.5,71.06) and (265.07,72.63) .. (265.07,74.56) .. controls (265.07,76.5) and (263.5,78.07) .. (261.56,78.07) .. controls (259.63,78.07) and (258.06,76.5) .. (258.06,74.56) -- cycle ;

\draw [line width=1.25]    (30,53) -- (30,40) ;
\draw [line width=1.25]    (61,52.5) -- (61,40) ;
\draw [line width=1.25]    (46,110) -- (46,92.5) ;
\draw [line width=1.25]    (100,110) -- (99.44,73) ;
\draw [line width=1.25]    (139.5,109.99) -- (139.5,94.98) ;
\draw [line width=1.25]    (170,110.51) -- (170,92.51) ;
\draw [line width=1.25]    (154,54.5) -- (154,39.5) ;
\draw [line width=1.25]    (210.55,77) -- (210,40) ;
\draw [line width=1.25]    (261,110) -- (261,93.63) ;
\draw [line width=1.25]    (261,52.38) -- (260.99,40) ;
\draw [line width=1.25]    (310,110.5) -- (310.5,73.13) ;
\draw  [fill={rgb, 255:red, 255; green, 255; blue, 255 }  ,fill opacity=1 ][line width=1.25]  (304.66,70.76) .. controls (307.24,67.56) and (311.93,67.07) .. (315.12,69.66) .. controls (318.31,72.24) and (318.8,76.93) .. (316.22,80.12) .. controls (313.63,83.31) and (308.95,83.8) .. (305.76,81.22) .. controls (302.56,78.63) and (302.07,73.95) .. (304.66,70.76) -- cycle ;
\draw [fill={rgb, 255:red, 255; green, 255; blue, 255 }  ,fill opacity=1 ][line width=1.25]    (304,75.53) -- (317.53,75.31) ;

\draw  [line width=1.25]  (306.93,75.44) .. controls (306.93,73.5) and (308.5,71.93) .. (310.44,71.93) .. controls (312.37,71.93) and (313.94,73.5) .. (313.94,75.44) .. controls (313.94,77.37) and (312.37,78.94) .. (310.44,78.94) .. controls (308.5,78.94) and (306.93,77.37) .. (306.93,75.44) -- cycle ;

\draw (38,122) node [anchor=north west][inner sep=0.75pt]  [font=\small]  {$m_{B}$};
\draw (93,122) node [anchor=north west][inner sep=0.75pt]  [font=\small]  {$u_{B}$};
\draw (145,118) node [anchor=north west][inner sep=0.75pt]  [font=\small]  {$\Delta _{B}$};
\draw (252,118) node [anchor=north west][inner sep=0.75pt]  [font=\small]  {$\phi_{B}$};
\draw (204,122) node [anchor=north west][inner sep=0.75pt]  [font=\small]  {$\varepsilon_{B}$};
\draw (303,118) node [anchor=north west][inner sep=0.75pt]  [font=\small]  {$\theta_{B}$};

\end{tikzpicture}
}
\vspace{-.1in}
\caption{Extended Frobenius structure on $B$.}
\label{fig:StructMorphAB}
\end{minipage}         
\begin{minipage}{.38\textwidth}
\centering

\scalebox{.78}{\tikzset{every picture/.style={line width=0.75pt}} 

\begin{tikzpicture}[x=0.75pt,y=0.75pt,yscale=-1,xscale=1]

\draw [line width=1.5]    (174.5,43.5) -- (174.5,51.5) ;
\draw [line width=1.5]    (190.5,80.51) .. controls (190.51,51.51) and (159.51,51.49) .. (160,80.49) ;
\draw [line width=1.5]    (174.51,58.5) -- (174.5,43.5) ;
\draw [line width=1.5]    (39.5,40) -- (40,150) ;
\draw [line width=1.5]    (129.5,40) -- (129,94.5) ;
\draw [line width=1.5]    (160,76.49) -- (160,93) ;
\draw [line width=1.5]    (190.5,77.51) -- (190,150.5) ;
\draw [line width=1.5]    (129,90.5) .. controls (129,119.5) and (160.5,119) .. (160,90) ;
\draw [line width=1.5]    (145,122) -- (145,139) ;
\draw [line width=1.5]    (145.01,151) -- (144.99,111.5) ;
\draw   (10.5,20.5) -- (210.5,20.5) -- (210.5,170) -- (10.5,170) -- cycle ;
\draw  [line width=1.5]  (137.55,111.5) .. controls (137.55,107.39) and (140.88,104.06) .. (144.99,104.06) .. controls (149.1,104.06) and (152.43,107.39) .. (152.43,111.5) .. controls (152.43,115.61) and (149.1,118.94) .. (144.99,118.94) .. controls (140.88,118.94) and (137.55,115.61) .. (137.55,111.5) -- cycle ;
\draw  [line width=1.5]  (137.56,131.56) .. controls (137.56,127.46) and (140.89,124.13) .. (145,124.13) .. controls (149.11,124.13) and (152.44,127.46) .. (152.44,131.56) .. controls (152.44,135.67) and (149.11,139) .. (145,139) .. controls (140.89,139) and (137.56,135.67) .. (137.56,131.56) -- cycle ;
\draw  [fill={rgb, 255:red, 255; green, 255; blue, 255 }  ,fill opacity=1 ][line width=1.5]  (141.51,147.5) .. controls (141.51,145.56) and (143.08,143.99) .. (145.01,143.99) .. controls (146.94,143.99) and (148.51,145.56) .. (148.51,147.5) .. controls (148.51,149.43) and (146.94,151) .. (145.01,151) .. controls (143.08,151) and (141.51,149.43) .. (141.51,147.5) -- cycle ;
\draw  [fill={rgb, 255:red, 255; green, 255; blue, 255 }  ,fill opacity=1 ][line width=1.5]  (171,44) .. controls (171,42.06) and (172.57,40.49) .. (174.5,40.49) .. controls (176.43,40.49) and (178,42.06) .. (178,44) .. controls (178,45.93) and (176.43,47.5) .. (174.5,47.5) .. controls (172.57,47.5) and (171,45.93) .. (171,44) -- cycle ;

\draw  [fill={rgb, 255:red, 255; green, 255; blue, 255 }  ,fill opacity=1 ][line width=1.5]   (32.26,85.5) -- (48.26,85.5) -- (48.26,103.5) -- (32.26,103.5) -- cycle  ;
\draw (45,90) node [anchor=north east] [inner sep=0.75pt]  [font=\footnotesize] [align=left] {$\displaystyle g$};
\draw (79,87.5) node [anchor=north west][inner sep=0.75pt]   [align=left] {:=};
\draw  [fill={rgb, 255:red, 255; green, 255; blue, 255 }  ,fill opacity=1 ][line width=1.5]   (152.26,76.5) -- (168.26,76.5) -- (168.26,94.5) -- (152.26,94.5) -- cycle  ;
\draw (165.26,79.25) node [anchor=north east] [inner sep=0.75pt]  [font=\footnotesize] [align=left] {$\displaystyle f$};
\draw (24,38) node [anchor=north west][inner sep=0.75pt]   [align=left] {$\displaystyle B$};
\draw (24,136.5) node [anchor=north west][inner sep=0.75pt]   [align=left] {$\displaystyle A$};
\end{tikzpicture}
}
\vspace{-0.1in}
\caption{Defining $g$.}
\label{fig:defg}
\end{minipage}
\end{figure}

                    \begin{figure}[h!]
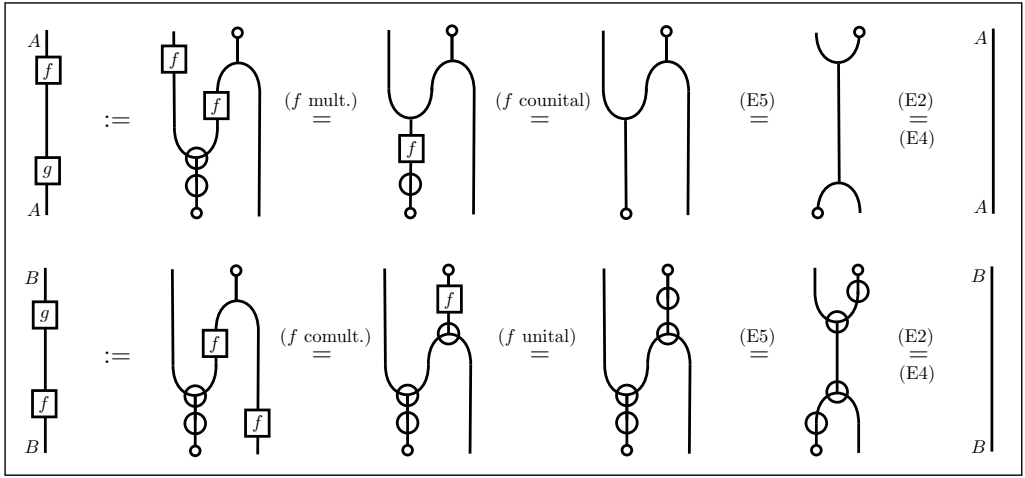


\tikzset{every picture/.style={line width=0.75pt}} 
\scalebox{.7}{     

}

\vspace{-0.1 in}
\caption{Proof that $gf= \id_A$ and $fg = \id_B$.}
\label{fig:FrobHomProof}
\end{figure}

\end{proof}


\subsection{Structure results} 
\label{sec:oper}
Recall the category $\mathsf{ExtFrobAlg}(\cC)$ defined in Definition~\ref{def:extFrobC}. We put monoidal structures on this category, using two distinct monoidal products, in the following results.

\begin{proposition} \label{prop:Ext-mon}
 Let $(\cC,\otimes,\one,c)$ be a symmetric monoidal category. Then,   $\ExtFrobAlg(\cC)$ is monoidal with $\otimes:=\otimes^{\cC}$ and $\one:=\one^{\cC}$. 
\end{proposition}

\begin{proof}
    We first note that $\one^\cC = (\one^\cC, \ell_\one, \id_\one, \ell^{-1}_\one, \id_\one, \id_\one,\id_\one)$ is an extended Frobenius algebra in $\cC$. 

    Next, we show that the monoidal product of two extended Frobenius algebras is extended Frobenius. Namely, we verify that given extended Frobenius algebras $(A,m_A,u_A,\Delta_A,\varep_A,\phi_A,\theta_A)$ and $(B,m_B,u_B,\Delta_B,\varep_B,\phi_B,\theta_B)$, then $(A\otimes B,\tilde m,\tilde u,\tilde\Delta,\tilde\varep,\tilde\phi,\tilde\theta)$ is an extended Frobenius algebra, where 
    \[\tilde m := (m_A\otimes m_B)(\id_A\otimes c_{B,A}\otimes \id_B),\qquad \tilde\Delta := (\id_A\otimes c_{A,B}\otimes\id_B)(\Delta_A\otimes\Delta_B)\]
    \[\tilde u:= u_A\otimes u_B,\qquad \tilde\varep := \varep_A\otimes\varep_B,\qquad \tilde \phi := \phi_A\otimes\phi_B,\qquad \tilde\theta := \theta_A\otimes\theta_B.\]
Figure~\ref{fig:ExtAxB} shows what these morphisms look like in graphical calculus, using the symbols from Figure~\ref{fig:eFrob} for $A$ and the symbols from Figure~\ref{fig:StructMorphAB} for $B$, as in Proposition \ref{prop:FrobHom}. Recall also the axioms for a symmetric monoidal category from Figure~\ref{fig:symmon}.

    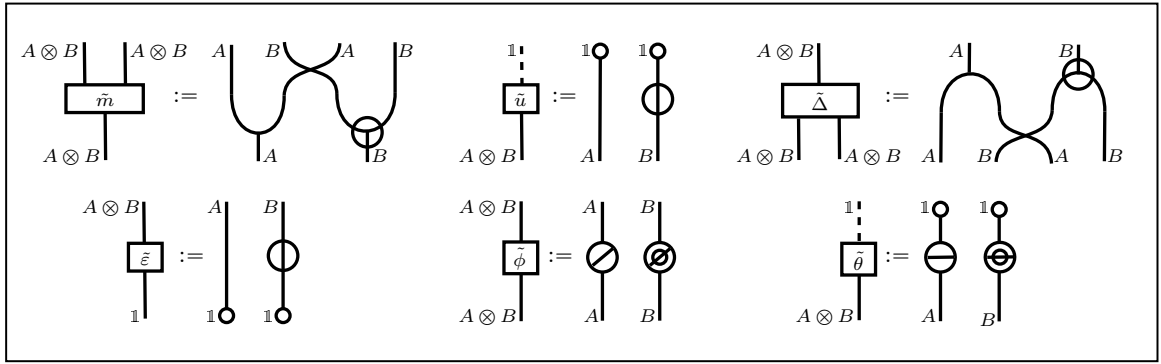
\begin{figure}[h!]

        \scalebox{1}{

\tikzset{every picture/.style={line width=0.75pt}} 

\begin{tikzpicture}[x=0.75pt,y=0.75pt,yscale=-1,xscale=1]

\draw [line width=1.25]    (177.13,55.61) .. controls (176.76,39.27) and (152,48.5) .. (150.66,29.42) ;
\draw [line width=1.25]    (150.53,56.37) .. controls (151.51,39.59) and (178,46.5) .. (178.61,30) ;
\draw [line width=1.25]    (124,56.01) .. controls (124,81.62) and (150.98,81.98) .. (150.53,56.37) ;
\draw [line width=1.25]    (137.52,75.54) -- (137.52,89.56) ;
\draw [line width=1.25]    (177.13,55.61) .. controls (177.13,81.23) and (206.95,80.31) .. (206.49,54.69) ;
\draw [line width=1.25]    (192.52,73.98) -- (192.52,90) ;
\draw [line width=1.25]    (124,30.82) -- (124,56.01) ;
\draw [line width=1.25]    (206.49,29.51) -- (206.49,54.69) ;
\draw [line width=1.25]    (310,37) -- (309.75,89.5) ;
\draw [line width=1.25]    (339,37) -- (339,89.5) ;
\draw [line width=1.25]    (511.3,64.34) .. controls (511.53,80.69) and (536.37,71.66) .. (537.55,90.75) ;
\draw [line width=1.25]    (537.91,63.81) .. controls (536.78,80.58) and (510.36,73.44) .. (509.61,89.94) ;
\draw [line width=1.25]    (564.43,64.39) .. controls (564.65,38.78) and (537.67,38.19) .. (537.91,63.81) ;
\draw [line width=1.25]    (551.07,45.74) -- (551.19,30.73) ;
\draw [line width=1.25]    (511.3,64.34) .. controls (511.52,38.73) and (481.69,39.39) .. (481.93,65.01) ;
\draw [line width=1.25]    (496.07,44.84) -- (496.2,29.82) ;
\draw [line width=1.25]    (564.22,89.58) -- (564.43,64.39) ;
\draw [line width=1.25]    (481.72,90.2) -- (481.93,65.01) ;
\draw [line width=1.25]    (121.5,167.5) -- (121.5,110) ;
\draw [line width=1.25]    (150,167.5) -- (150,110) ;
\draw [line width=1.25] [dashed]    (440.56,109.81) -- (440.56,140) ;
\draw [line width=1.25]     (440.56,140) -- (440.56,169.81) ;
\draw [line width=1.25]    (481.56,117.81) -- (481.56,170.31) ;
\draw [line width=1.25]    (511.16,117.91) -- (511.16,170.41) ;
\draw [line width=1.25]    (270.06,109.81) -- (270.06,170.31) ;
\draw [line width=1.25]    (311.06,110.06) -- (311.06,170.06) ;
\draw [line width=1.25]    (339.96,110.04) -- (339.96,170.04) ;
\draw   (10.33,10) -- (591,10) -- (591,190.5) -- (10.33,190.5) -- cycle ;
\draw  [line width=1.25]  (142.56,137.25) .. controls (142.56,133.14) and (145.89,129.81) .. (150,129.81) .. controls (154.11,129.81) and (157.44,133.14) .. (157.44,137.25) .. controls (157.44,141.36) and (154.11,144.69) .. (150,144.69) .. controls (145.89,144.69) and (142.56,141.36) .. (142.56,137.25) -- cycle ;
\draw  [line width=1.25]  (185.08,74.98) .. controls (185.08,70.88) and (188.41,67.55) .. (192.52,67.55) .. controls (196.63,67.55) and (199.96,70.88) .. (199.96,74.98) .. controls (199.96,79.09) and (196.63,82.42) .. (192.52,82.42) .. controls (188.41,82.42) and (185.08,79.09) .. (185.08,74.98) -- cycle ;
\draw  [line width=1.25]  (331.56,58.25) .. controls (331.56,54.14) and (334.89,50.81) .. (339,50.81) .. controls (343.11,50.81) and (346.44,54.14) .. (346.44,58.25) .. controls (346.44,62.36) and (343.11,65.69) .. (339,65.69) .. controls (334.89,65.69) and (331.56,62.36) .. (331.56,58.25) -- cycle ;
\draw  [line width=1.25]  (543.63,45.74) .. controls (543.63,41.64) and (546.96,38.31) .. (551.07,38.31) .. controls (555.17,38.31) and (558.5,41.64) .. (558.5,45.74) .. controls (558.5,49.85) and (555.17,53.18) .. (551.07,53.18) .. controls (546.96,53.18) and (543.63,49.85) .. (543.63,45.74) -- cycle ;
\draw  [fill={rgb, 255:red, 255; green, 255; blue, 255 }  ,fill opacity=1 ][line width=1.25]  (303.63,138.06) .. controls (303.63,133.96) and (306.96,130.63) .. (311.06,130.63) .. controls (315.17,130.63) and (318.5,133.96) .. (318.5,138.06) .. controls (318.5,142.17) and (315.17,145.5) .. (311.06,145.5) .. controls (306.96,145.5) and (303.63,142.17) .. (303.63,138.06) -- cycle ;
\draw [fill={rgb, 255:red, 255; green, 255; blue, 255 }  ,fill opacity=1 ][line width=1.25]    (306.12,142.18) -- (316.5,133.5) ;

\draw  [fill={rgb, 255:red, 255; green, 255; blue, 255 }  ,fill opacity=1 ][line width=1.25]  (332.53,138.04) .. controls (332.53,133.93) and (335.86,130.6) .. (339.96,130.6) .. controls (344.07,130.6) and (347.4,133.93) .. (347.4,138.04) .. controls (347.4,142.14) and (344.07,145.47) .. (339.96,145.47) .. controls (335.86,145.47) and (332.53,142.14) .. (332.53,138.04) -- cycle ;
\draw [fill={rgb, 255:red, 255; green, 255; blue, 255 }  ,fill opacity=1 ][line width=1.25]    (335.02,142.16) -- (345.4,133.47) ;

\draw  [line width=1.25]  (336.46,138.04) .. controls (336.46,136.1) and (338.03,134.53) .. (339.96,134.53) .. controls (341.9,134.53) and (343.47,136.1) .. (343.47,138.04) .. controls (343.47,139.97) and (341.9,141.54) .. (339.96,141.54) .. controls (338.03,141.54) and (336.46,139.97) .. (336.46,138.04) -- cycle ;

\draw  [fill={rgb, 255:red, 255; green, 255; blue, 255 }  ,fill opacity=1 ][line width=1.25]  (475.78,133.38) .. controls (478.37,130.19) and (483.05,129.7) .. (486.24,132.28) .. controls (489.44,134.87) and (489.93,139.55) .. (487.34,142.74) .. controls (484.76,145.94) and (480.07,146.43) .. (476.88,143.84) .. controls (473.69,141.26) and (473.2,136.57) .. (475.78,133.38) -- cycle ;
\draw [fill={rgb, 255:red, 255; green, 255; blue, 255 }  ,fill opacity=1 ][line width=1.25]    (475.13,138.15) -- (488.66,137.94) ;

\draw  [fill={rgb, 255:red, 255; green, 255; blue, 255 }  ,fill opacity=1 ][line width=1.25]  (505.78,133.28) .. controls (508.37,130.09) and (513.05,129.6) .. (516.24,132.18) .. controls (519.44,134.77) and (519.93,139.45) .. (517.34,142.64) .. controls (514.76,145.84) and (510.07,146.33) .. (506.88,143.74) .. controls (503.69,141.16) and (503.2,136.47) .. (505.78,133.28) -- cycle ;
\draw [fill={rgb, 255:red, 255; green, 255; blue, 255 }  ,fill opacity=1 ][line width=1.25]    (505.13,138.05) -- (518.66,137.84) ;

\draw  [line width=1.25]  (508.06,137.96) .. controls (508.06,136.03) and (509.63,134.46) .. (511.56,134.46) .. controls (513.5,134.46) and (515.07,136.03) .. (515.07,137.96) .. controls (515.07,139.9) and (513.5,141.47) .. (511.56,141.47) .. controls (509.63,141.47) and (508.06,139.9) .. (508.06,137.96) -- cycle ;

\draw  [fill={rgb, 255:red, 255; green, 255; blue, 255 }  ,fill opacity=1 ][line width=1.25]  (306.5,34) .. controls (306.5,32.07) and (308.07,30.5) .. (310,30.5) .. controls (311.93,30.5) and (313.5,32.07) .. (313.5,34) .. controls (313.5,35.93) and (311.93,37.5) .. (310,37.5) .. controls (308.07,37.5) and (306.5,35.93) .. (306.5,34) -- cycle ;
\draw  [fill={rgb, 255:red, 255; green, 255; blue, 255 }  ,fill opacity=1 ][line width=1.25]  (335.5,34) .. controls (335.5,32.07) and (337.07,30.5) .. (339,30.5) .. controls (340.93,30.5) and (342.5,32.07) .. (342.5,34) .. controls (342.5,35.93) and (340.93,37.5) .. (339,37.5) .. controls (337.07,37.5) and (335.5,35.93) .. (335.5,34) -- cycle ;
\draw  [fill={rgb, 255:red, 255; green, 255; blue, 255 }  ,fill opacity=1 ][line width=1.25]  (118,167.5) .. controls (118,165.57) and (119.57,164) .. (121.5,164) .. controls (123.44,164) and (125.01,165.57) .. (125.01,167.5) .. controls (125.01,169.43) and (123.44,171) .. (121.5,171) .. controls (119.57,171) and (118,169.43) .. (118,167.5) -- cycle ;
\draw  [fill={rgb, 255:red, 255; green, 255; blue, 255 }  ,fill opacity=1 ][line width=1.25]  (146.5,167.5) .. controls (146.5,165.57) and (148.07,164) .. (150,164) .. controls (151.94,164) and (153.51,165.57) .. (153.51,167.5) .. controls (153.51,169.43) and (151.94,171) .. (150,171) .. controls (148.07,171) and (146.5,169.43) .. (146.5,167.5) -- cycle ;
\draw  [fill={rgb, 255:red, 255; green, 255; blue, 255 }  ,fill opacity=1 ][line width=1.25]  (478.06,113.81) .. controls (478.06,111.88) and (479.63,110.31) .. (481.56,110.31) .. controls (483.5,110.31) and (485.07,111.88) .. (485.07,113.81) .. controls (485.07,115.75) and (483.5,117.32) .. (481.56,117.32) .. controls (479.63,117.32) and (478.06,115.75) .. (478.06,113.81) -- cycle ;
\draw  [fill={rgb, 255:red, 255; green, 255; blue, 255 }  ,fill opacity=1 ][line width=1.25]  (507.66,113.91) .. controls (507.66,111.98) and (509.23,110.41) .. (511.16,110.41) .. controls (513.1,110.41) and (514.67,111.98) .. (514.67,113.91) .. controls (514.67,115.85) and (513.1,117.42) .. (511.16,117.42) .. controls (509.23,117.42) and (507.66,115.85) .. (507.66,113.91) -- cycle ;
\draw [line width=1.25]    (50,50.5) -- (49.75,29.5) ;
\draw [line width=1.25]    (70.5,50.5) -- (70.25,29.5) ;
\draw [line width=1.25]    (60.5,89) -- (60.25,65) ;
\draw [line width=1.25]  [dashed]  (270.5,51.5) -- (270.25,30.5) ;
\draw [line width=1.25]    (270.5,89) -- (270.25,65) ;
\draw [line width=1.25]    (430.47,67.53) -- (430.65,88.53) ;
\draw [line width=1.25]    (409.97,67.96) -- (410.15,88.96) ;
\draw [line width=1.25]    (420.1,30) -- (420.27,51) ;
\draw [line width=1.25]    (80,131) -- (79.75,110) ;
\draw [line width=1.25]    (80.5,169) -- (80.25,145) ;

\draw (16,29) node [anchor=north west][inner sep=0.75pt]  [font=\tiny]  {$A\otimes B$};
\draw (71,29) node [anchor=north west][inner sep=0.75pt]  [font=\tiny]  {$A\otimes B$};
\draw (27,82) node [anchor=north west][inner sep=0.75pt]  [font=\tiny]  {$A\otimes B$};
\draw  [line width=1.25]   (40.5,51) -- (79.5,51) -- (79.5,65) -- (40.5,65) -- cycle  ;
\draw (54,54) node [anchor=north west][inner sep=0.75pt]  [font=\tiny]  {$\tilde{m}$};
\draw (112,29) node [anchor=north west][inner sep=0.75pt]  [font=\tiny]  {$A$};
\draw (177,29) node [anchor=north west][inner sep=0.75pt]  [font=\tiny]  {$A$};
\draw (206,29) node [anchor=north west][inner sep=0.75pt]  [font=\tiny]  {$B$};
\draw (139,29) node [anchor=north west][inner sep=0.75pt]  [font=\tiny]  {$B$};
\draw (138,82) node [anchor=north west][inner sep=0.75pt]  [font=\tiny]  {$A$};
\draw (193,82) node [anchor=north west][inner sep=0.75pt]  [font=\tiny]  {$B$};
\draw  [line width=1.25]   (261.5,51) -- (278.5,51) -- (278.5,65) -- (261.5,65) -- cycle  ;
\draw (265,54) node [anchor=north west][inner sep=0.75pt]  [font=\tiny]  {$\tilde{u}$};
\draw (261,29) node [anchor=north west][inner sep=0.75pt]  [font=\tiny]  {$\one$};
\draw (237,82) node [anchor=north west][inner sep=0.75pt]  [font=\tiny]  {$A\otimes B$};
\draw (298,29) node [anchor=north west][inner sep=0.75pt]  [font=\tiny]  {$\one$};
\draw (327,29) node [anchor=north west][inner sep=0.75pt]  [font=\tiny]  {$\one$};
\draw (298,82) node [anchor=north west][inner sep=0.75pt]  [font=\tiny]  {$A$};
\draw (327,82) node [anchor=north west][inner sep=0.75pt]  [font=\tiny]  {$B$};
\draw (377,82) node [anchor=north west][inner sep=0.75pt]  [font=\tiny]  {$A\otimes B$};
\draw (387,29) node [anchor=north west][inner sep=0.75pt]  [font=\tiny]  {$A\otimes B$};
\draw (431,82) node [anchor=north west][inner sep=0.75pt]  [font=\tiny]  {$A\otimes B$};
\draw (470,82) node [anchor=north west][inner sep=0.75pt]  [font=\tiny]  {$A$};
\draw (537,82) node [anchor=north west][inner sep=0.75pt]  [font=\tiny]  {$A$};
\draw (564,82) node [anchor=north west][inner sep=0.75pt]  [font=\tiny]  {$B$};
\draw (498,82) node [anchor=north west][inner sep=0.75pt]  [font=\tiny]  {$B$};
\draw (484,29) node [anchor=north west][inner sep=0.75pt]  [font=\tiny]  {$A$};
\draw (539,29) node [anchor=north west][inner sep=0.75pt]  [font=\tiny]  {$B$};
\draw (47,108) node [anchor=north west][inner sep=0.75pt]  [font=\tiny]  {$A\otimes B$};
\draw (71,163) node [anchor=north west][inner sep=0.75pt]  [font=\tiny]  {$\one$};
\draw  [line width=1.25]   (72,131) -- (89,131) -- (89,145) -- (72,145) -- cycle  ;
\draw (76.5,134) node [anchor=north west][inner sep=0.75pt]  [font=\tiny]  {$\tilde{\varepsilon}$};
\draw (109,163) node [anchor=north west][inner sep=0.75pt]  [font=\tiny]  {$\one$};
\draw (138,163) node [anchor=north west][inner sep=0.75pt]  [font=\tiny]  {$\one$};
\draw (110,108) node [anchor=north west][inner sep=0.75pt]  [font=\tiny]  {$A$};
\draw (138,108) node [anchor=north west][inner sep=0.75pt]  [font=\tiny]  {$B$};
\draw  [fill={rgb, 255:red, 255; green, 255; blue, 255 }  ,fill opacity=1 ][line width=1.25]   (431.8,131) -- (448.8,131) -- (448.8,147) -- (431.8,147) -- cycle  ;
\draw (436,134) node [anchor=north west][inner sep=0.75pt]  [font=\tiny]  {$\tilde{\theta}$};
\draw (431,108) node [anchor=north west][inner sep=0.75pt]  [font=\tiny]  {$\one$};
\draw (407,163) node [anchor=north west][inner sep=0.75pt]  [font=\tiny]  {$A\otimes B$};
\draw (469,108) node [anchor=north west][inner sep=0.75pt]  [font=\tiny]  {$\one$};
\draw (499,108) node [anchor=north west][inner sep=0.75pt]  [font=\tiny]  {$\one$};
\draw (470,163) node [anchor=north west][inner sep=0.75pt]  [font=\tiny]  {$A$};
\draw (493,163) node [anchor=north west][inner sep=0.75pt]  [font=\tiny]  {$ \begin{array}{l}
B\\
\end{array}$};
\draw (237,163) node [anchor=north west][inner sep=0.75pt]  [font=\tiny]  {$A\otimes B$};
\draw (300,163) node [anchor=north west][inner sep=0.75pt]  [font=\tiny]  {$A$};
\draw (328,163) node [anchor=north west][inner sep=0.75pt]  [font=\tiny]  {$B$};
\draw (237,108) node [anchor=north west][inner sep=0.75pt]  [font=\tiny]  {$A\otimes B$};
\draw (299,108) node [anchor=north west][inner sep=0.75pt]  [font=\tiny]  {$A$};
\draw (328,108) node [anchor=north west][inner sep=0.75pt]  [font=\tiny]  {$B$};
\draw  [fill={rgb, 255:red, 255; green, 255; blue, 255 }  ,fill opacity=1 ][line width=1.25]   (261.4,130.33) -- (278.4,130.33) -- (278.4,146.33) -- (261.4,146.33) -- cycle  ;
\draw (264.5,132) node [anchor=north west][inner sep=0.75pt]  [font=\tiny]  {$\tilde{\phi}$};
\draw (93,53) node [anchor=north west][inner sep=0.75pt]  [font=\footnotesize]  {$:=$};
\draw (284,53) node [anchor=north west][inner sep=0.75pt]  [font=\footnotesize]  {$:=$};
\draw (452,53) node [anchor=north west][inner sep=0.75pt]  [font=\footnotesize]  {$:=$};
\draw (452,134) node [anchor=north west][inner sep=0.75pt]  [font=\footnotesize]  {$:=$};
\draw (282,134) node [anchor=north west][inner sep=0.75pt]  [font=\footnotesize]  {$:=$};
\draw (96,134) node [anchor=north west][inner sep=0.75pt]  [font=\footnotesize]  {$:=$};
\draw  [line width=1.25]   (401.93,50.8) -- (440.07,50.8) -- (440.07,67.2) -- (401.93,67.2) -- cycle  ;
\draw (415,53.5) node [anchor=north west][inner sep=0.75pt]  [font=\tiny,rotate=-0.68]  {$\tilde{\Delta}$};

\end{tikzpicture}
}

    \vspace{-.1in}
    \caption{Extended Frobenius structure morphisms for $A\otimes B$.}
    \label{fig:ExtAxB}
    \end{figure}

\smallskip

We then have that $(A\otimes B,\tilde m,\tilde u,\tilde\Delta,\tilde\varep)\in\mathsf{FrobAlg}(\cC)$ by \cite[Section 2.4]{Kock}. To see that this Frobenius algebra is extended via $\tilde\phi$ and $\tilde\theta$, we verify the three required conditions in Definition~\ref{def:extFrobC}(b).
    
\begin{enumerate}\renewcommand{\labelenumi}{(\roman{enumi})}
\item It is easy to see that $\tilde\phi$ is an involution, since both $\phi_A$ and $\phi_B$ are involutions. Moreover, since both $\phi_A$, $\phi_B$ are Frobenius morphisms, so is their monoidal product in $\cC$.

\smallskip

\item Figure~\ref{fig:E11AxB} gives that $\tilde\phi\,\tilde m(\tilde\theta\otimes \id_{A\otimes B}) = \tilde m(\tilde\theta\otimes\id_{A\otimes B})$.

\smallskip

\item Finally, Figure~\ref{fig:E12AxB} gives that $\tilde m(\tilde\phi\otimes\id_{A\otimes B})\tilde\Delta\tilde u = \tilde m(\tilde\theta\otimes\tilde\theta)$.

\end{enumerate}
Thus, we have that $(A\otimes B, \tilde\phi,\tilde\theta)\in\mathsf{ExtFrobAlg}(\cC)$, as desired.

Lastly, we note that by taking $\one^\cC$ as the unit and $\otimes^\cC$ as the monoidal product in $\mathsf{ExtFrobAlg}(\cC)$, with extended structures behaving as described above, we obtain that the required pentagon and triangle axioms in $(\mathsf{ExtFrobAlg}(\cC),\otimes^\cC,\one^\cC)$ are all inherited from the same axioms in $(\cC,\otimes^\cC,\one^\cC)$. From this, we can conclude that $(\mathsf{ExtFrobAlg}(\cC),\otimes^\cC,\one^\cC)$ is a monoidal category. 
\end{proof}

\vspace{0.1in} 

    \begin{figure}[h!]

        \scalebox{.75}{


}

        \vspace{-.1in}
        \caption{Proof that $A\otimes B$ satisfies Definition~\ref{def:extFrobC}(b)(ii).}
        \label{fig:E11AxB}
    \end{figure}

    \begin{figure}[h!]

        \scalebox{.75}{

\tikzset{every picture/.style={line width=0.75pt}} 


}

        \vspace{-.1in}
        \caption{Proof that $A\otimes B$ satisfies Definition~\ref{def:extFrobC}(b)(iii).}
        \label{fig:E12AxB}
    \end{figure}

Now we turn our attention to extended Frobenius algebras in additive monoidal categories. See \cite[Section~3.1.3]{Walton} for background material on such categories.

\begin{proposition} \label{prop:eFMbiprod} Let $(\cC,\otimes,\one)$ be an additive  monoidal category. Then, the category  $\ExtFrobAlg(\cC)$ is monoidal with $\otimes$ being the biproduct $\osq$, and $\one$ being the zero object~$\mathsf{0}$. 
\end{proposition}

\begin{proof}
We first note that $\mathsf{0}$ is an extended Frobenius algebra in $\cC$, with structure morphisms $m,u,\Delta,\varep,$ and $\theta$ all being zero morphisms, and $\phi = \id_\mathsf{0}$. We next note that similar to the previous proposition, the pentagon and triangle axioms in $(\mathsf{\ExtFrobAlg}(\cC),\osq,\mathsf 0)$ will be inherited from these same axioms on the strict monoidal category $(\cC,\osq,\mathsf 0)$. Hence, to finish the proof, it suffices to show that the biproduct of two extended Frobenius algebras is again extended Frobenius. To do so, let $(A,m_A,u_A,\Delta_A,\varep_A,\phi_A,\theta_A)$ and $(B,m_B,u_B,\Delta_B,\varep_B,\phi_B,\theta_B)$ be two extended Frobenius algebras in $\cC$. We will show that $(A\osq B,\tilde m,\tilde u,\tilde\Delta,\tilde\varep,\tilde\phi,\tilde\theta)$ is an extended Frobenius algebra, where $\tilde m$, $\tilde u$, $\tilde\Delta$, $\tilde \varep$, $\tilde\phi$, and $\tilde\theta$ are defined by the following universal property diagrams.

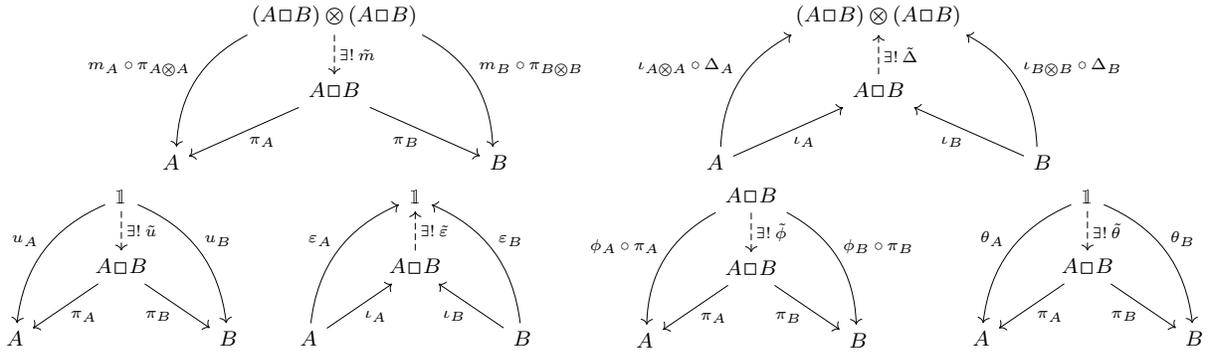
\vspace{0.05in}

\begin{scriptsize}\begin{center}
    \begin{tikzcd}
  & (A\osq B)\otimes (A\osq B) \arrow[ldd, "m_A\,\circ\, \pi_{A\otimes A}"', bend right] \arrow[rdd, "{m_B\,\circ\, \pi_{B\otimes B}}", bend left] \arrow[d, "\exists!\, \tilde m", dashed] &    \\
  & A\osq B \arrow[ld, "\pi_A"] \arrow[rd, "\pi_{B}"']                                                                                                                       &    \\
A &                                                                                                                                                                               & B
\end{tikzcd}\qquad\begin{tikzcd}
                                                                                  & (A\osq B)\otimes (A\osq B)                        &                                                                                           \\
                                                                                  & A\osq B \arrow[u, "\exists!\,\tilde\Delta"', dashed] &                                                                                           \\
A \arrow[ruu, "\iota_{A\otimes A}\,\circ \,\Delta_A", bend left] \arrow[ru, "\iota_A"'] &                                                     & B \arrow[luu, "\iota_{B\otimes B}\,\circ\, \Delta_B"', bend right] \arrow[lu, "\iota_{B}"]
\end{tikzcd}

\begin{tikzcd}
  & \one \arrow[ldd, "u_A"', bend right] \arrow[rdd, "u_B", bend left] \arrow[d, "\exists!\, \tilde u", dashed] &    \\
  & A\osq B \arrow[ld, "\pi_A"] \arrow[rd, "\pi_{B}"']                                                &    \\
A &                                                                                                        & B
\end{tikzcd}\qquad\begin{tikzcd}
                                                          & \one                                                &                                                                 \\
                                                          & A\osq B \arrow[u, "\exists!\,\tilde\varep"', dashed] &                                                                 \\
A \arrow[ruu, "\varep_A", bend left] \arrow[ru, "\iota_A"'] &                                                     & B \arrow[luu, "\varep_B"', bend right] \arrow[lu, "\iota_{B}"]
\end{tikzcd}\qquad\begin{tikzcd}
  & A\osq B \arrow[ldd, "\phi_A\,\circ\, \pi_A"', bend right] \arrow[rdd, "\phi_B\,\circ\,\pi_{B}", bend left] \arrow[d, "\exists!\,\tilde\phi", dashed] &    \\
  & A\osq B \arrow[ld, "\pi_A"] \arrow[rd, "\pi_{B}"']                                                                                      &    \\
A &                                                                                                                                           & B
\end{tikzcd}\qquad\begin{tikzcd}
  & \one \arrow[ldd, "\theta_A"', bend right] \arrow[rdd, "\theta_B", bend left] \arrow[d, "\exists!\,\tilde\theta", dashed] &    \\
  & A\osq B \arrow[ld, "\pi_A"] \arrow[rd, "\pi_{B}"']                                                                 &    \\
A &                                                                                                                     & B
\end{tikzcd}
\end{center}
\end{scriptsize}

It is well known that with the above constructions, $(A\osq B,\tilde m,\tilde u,\tilde \Delta,\tilde\varep)$ is a Frobenius algebra. See \cite[Exercises 2.2.7 and 2.2.8]{Kock} for the case where $\cC = \mathsf{Vec}$. Thus, we only need to verify that $\tilde\phi$ and $\tilde\theta$ extend this Frobenius algebra. The three required properties from Definition \ref{def:extFrobC}(b) can be verified by respectively considering each of the universal property diagrams below. 

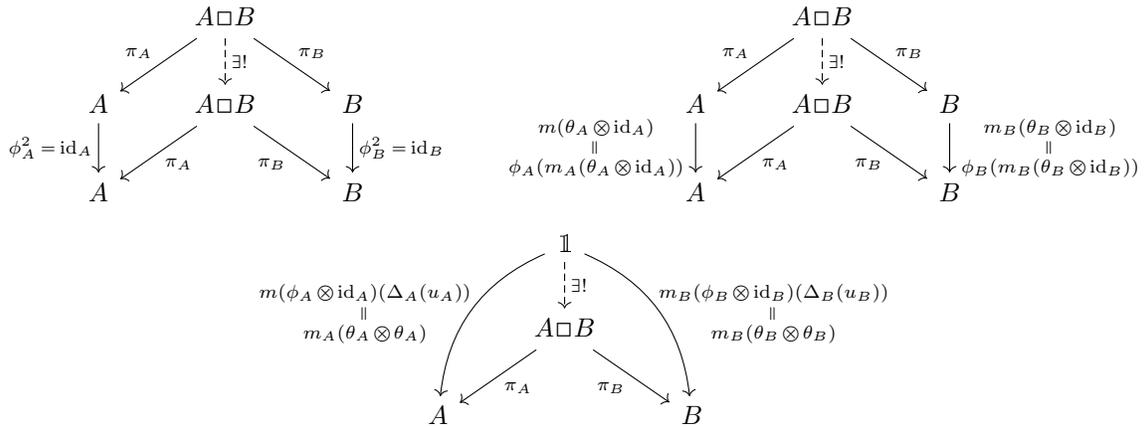
\vspace{0.03in}

\begin{center}\begin{small}
    \begin{tikzcd}
                               & A\osq B \arrow[ld, "\pi_A"'] \arrow[rd, "\pi_{B}"] \arrow[d, "\exists!", dashed] &                                      \\
A \arrow[d, "\phi_A^2 \,=\,\id_A"'] & A\osq B \arrow[ld, "\pi_A"] \arrow[rd, "\pi_{B}"']                               & B \arrow[d, "\phi_B^2 \,=\, \id_{B}"] \\
A                              &                                                                                    & B                                 
\end{tikzcd}\qquad \begin{tikzcd}
                                                                    & A\osq B \arrow[ld, "\pi_A"'] \arrow[rd, "\pi_{B}"] \arrow[d, "\exists!", dashed] &                                                                                \\
A \arrow[d, "\substack{m(\theta_A\,\otimes\,\id_A) \\ \rotatebox[origin=c]{90}{$=$} \\ \phi_A(m_A(\theta_A\,\otimes\,\id_A))}"'{xshift = -2pt}] & A\osq B \arrow[ld, "\pi_A"] \arrow[rd, "\pi_{B}"']                               & B \arrow[d, "\substack{m_B(\theta_B\,\otimes\,\id_{B}) \\ \rotatebox[origin=c]{90}{$=$} \\ \phi_B(m_B(\theta_B\,\otimes\,\id_{B}))}"{xshift = 2pt}] \\
A                                                                   &                                                                                    & B                                                                            
\end{tikzcd}

\medskip

\begin{tikzcd}
  & \one \arrow[ldd, near end, "\substack{m(\phi_A\,\otimes\,\id_A)(\Delta_A(u_A)) \\ \rotatebox[origin=c]{90}{$=$} \\ m_A(\theta_A\,\otimes\,\theta_A)}"'{xshift = 10pt}, bend right] \arrow[rdd, near end, "\substack{m_B(\phi_B\,\otimes\,\id_{B})(\Delta_B(u_B)) \\ \rotatebox[origin=c]{90}{$=$} \\ m_B(\theta_B\,\otimes\,\theta_B)}"{xshift = -10pt}, bend left] \arrow[d, "\exists!", dashed] &    \\
  & A\osq B \arrow[ld, "\pi_A"] \arrow[rd, "\pi_{B}"']                                                                                                                                                               &    \\
A &                                                                                                                                                                                                                    & B
\end{tikzcd}
\end{small}
\end{center}

Using uniqueness of the completing map in each of the diagrams, it follows that (i) $(\tilde\phi)^2 = \id_{A\osq B}$, (ii) $\tilde m(\tilde\theta\otimes\id_{A\osq B}) = \tilde \phi(\tilde m(\tilde\theta\otimes\id_{A\osq B}))$, and (iii) $\tilde m(\tilde \phi\otimes\id_{A\osq B})(\tilde\Delta(\tilde u)) = \tilde m(\tilde\theta\otimes\tilde\theta)$.

This completes the proof that $(A\osq B,\tilde\phi,\tilde\theta)$ is an extended Frobenius algebras in $\cC$, which in turn gives that $(\mathsf{ExtFrobAlg}(\cC),\osq,0)$ is a monoidal category.
\end{proof}


\section{Ties to separable algebras and Hopf algebras}
\label{sec:SepHopf}

In this section, we study extended Frobenius algebras in (symmetric) monoidal categories $\cC$, in the context of separable algebras and Hopf algebras in $\cC$; see Sections~\ref{sec:Sep} and~\ref{sec:Hopf}, respectively. We also introduce the notion of an extended Hopf algebra in $\cC$, and make connections to extended Frobenius algebras in $\cC$, in Section~\ref{sec:eHopf}.


\subsection{Tie to separable algebras} \label{sec:Sep} 
Take $\cC:=(\cC, \otimes, \one)$ to be a monoidal category, and consider the terminology below. See  \cite[Chapter~6]{Bohm} and references within for the case when $\cC = \mathsf{Vec}$.

\begin{definition} \label{def:sep} 
\begin{enumerate}[(a)]
    \item  We say that an algebra $A:=(A,m,u)$ in $\cC$ is {\it separable} if there exists a morphism $t:A \to A \otimes A$ such that $mt = \id_A$, and 
\[
(m \otimes \id_A)(\id_A \otimes t) \; = \;  tm  \; = \; 
(\id_A \otimes m)(t \otimes \id_A).
\]
\item A Frobenius algebra  $A:= (A,m,u,\Delta,\varep)$ is {\it separable Frobenius} if  $m \Delta =  \id_A$.
\end{enumerate}
\end{definition}

These structures form full subcategories as indicated below:
\[
\mathsf{SepAlg}(\cC) \subset \mathsf{Alg}(\cC), \qquad
\mathsf{SepFrobAlg}(\cC) \subset \mathsf{FrobAlg}(\cC).
\]

\begin{proposition} \label{prop:sepExt} If $A$ is a separable Frobenius algebra in $\cC$, then $A$ is extendable.
\end{proposition}

\begin{proof}
 Suppose that $A:= (A,m,u,\Delta,\varep)$ is a separable Frobenius algebra, and take $\phi:=\id_A$ and $\theta:= u$. Then, conditions (i) and~(ii) of Definition~\ref{def:extFrobC}(b) clearly hold. Condition (iii) of Definition~\ref{def:extFrobC}(b) holds by the computation below:
\[
m (\phi \otimes \id_A) \Delta u 
\; = \; m \Delta u 
\; = \;  u
\; = \; m ( u \otimes  u)
\; = \; m (\theta \otimes \theta),
\]
where the third equality follows from a unitality axiom of $A$.
\end{proof}

\begin{example}
The monoidal unit $\one \in \cC$ is a separable Frobenius algebra, with $m$ and $\Delta$ identified as $\id_\one$, and with $u = \varep = \id_\one$. The Frobenius structure is then extended with $\phi = \theta = \id_\one$.
\end{example}


\subsection{Tie to Hopf algebras}  \label{sec:Hopf}
Take $\cC:=(\cC, \otimes, \one, c)$ to be a symmetric monoidal category. See  \cite[Chapter~10]{Radford} and references within for the case when $\cC = \mathsf{Vec}$ for the material below.

\begin{definition} \label{def:HopfAlg}
Consider the following constructions in $\cC:= (\cC,\otimes,\one,c)$.
\begin{enumerate}\renewcommand{\labelenumi}{(\alph{enumi})}
\item A \textit{Hopf algebra in $\cC$} is a tuple $(H, m, u, \HDelta, \Hvarep, S)$, where $(H,m,u)$ in an algebra in $\cC$ and $(H,\HDelta,\Hvarep)$ is a coalgebra in $\cC$, subject to the bialgebra laws,
\[
\HDelta, \; \Hvarep \; \in \mathsf{Alg}(\cC) 
\; \; \;
(\Leftrightarrow
\; \; 
m,  u \; \in \mathsf{Coalg}(\cC)),
\]
and where $S: H \to H$ ({\it antipode}) is a morphism in $\cC$ satisfying the antipode axiom,
\[
m(S\otimes\id_H)\HDelta \; =  \; u \hspace{0.02in} \Hvarep \; = \; m(\id_H\otimes S)\HDelta.
\]
 If the antipode $S$ is invertible with inverse $S^{-1}:H\to H$ in $\cC$, then we call the tuple $(H,m,u,\HDelta,\Hvarep,S,S^{-1})$ a \textit{Hopf algebra with invertible antipode}.

  \medskip
  
\item  A \textit{left integral} for a Hopf algebra $(H,m,u,\HDelta,\Hvarep,S)$ is a morphism $\Lambda: \one\to H$ which satisfies $m(\id_H\otimes\Lambda) = \Lambda \hspace{0.02in} \Hvarep$. A \textit{right cointegral} for the Hopf algebra $(H,m,u,\HDelta,\Hvarep,S)$ is a morphism $\lambda:H\to\one$ satisfying $(\lambda\otimes\id_H)\HDelta = u \hspace{0.02in} \lambda$.  If $\Lambda$ and $\lambda$ further satisfy $\lambda \hspace{0.02in} \Lambda = \id_{\one}$, then $\Lambda$ and $\lambda$ are said to be \textit{normalized}. A Hopf algebra equipped with a normalized (co)integral pair is called an \textit{integral Hopf algebra}.
  
  \medskip
  
\item A \textit{morphism of integral Hopf algebras} $f: H \to K$ is a morphism, which is both an algebra and coalgebra morphism, and which satisfies $f\Lambda_H = \Lambda_K$ and $\lambda_K f = \lambda_H$. 

  \medskip

\item We organize the above into a category, $\mathsf{IntHopfAlg}(\cC)$, whose objects are integral Hopf algebras and whose morphisms are morphisms of integral Hopf algebras as defined above.
\end{enumerate} 
\end{definition}

See Figures~\ref{fig:Hopfmap}-\ref{fig:Hopfint} in Appendix \ref{sec:HF} for a graphical representation of this definition. 

\begin{remark}
If a Hopf algebra is equipped with a normalized integral and cointegral, then the antipode is invertible; see, e.g.,  \cite[Lemma~3.5]{CollinsDuncan}. 
\end{remark}

Now we show that an integral Hopf algebra in $\cC$ admits the structure of a Frobenius algebra in~$\cC$. A similar argument can also be found in \cite[Appendix~A.2]{FuchsSchweigert}.

\begin{proposition}\label{prop:hopftofrob}
We have that 
\begin{align*}
\Psi:\mathsf{IntHopfAlg}(\cC) &\to \mathsf{FrobAlg}(\cC)\\
(H,\hspace{0.01in} m, \hspace{0.01in} u,  \hspace{0.01in} \HDelta, \hspace{0.01in} \Hvarep, \hspace{0.01in} S, \hspace{0.01in} S^{-1}, \hspace{0.01in} \Lambda, \hspace{0.01in}\lambda) 
&\mapsto (H, \hspace{0.01in} m, \hspace{0.01in} u, \; \Delta:= (m  \otimes S)(\id_H \otimes \HDelta\hs \Lambda), \;  \varep:=\lambda)
\end{align*}
is a well-defined functor, which acts as the identity on morphisms.
\end{proposition}

\begin{proof}
This is established in Appendix~\ref{sec:HF} via graphical calculus arguments.   
\end{proof}

\begin{example} \label{ex: Psi(kG)}
Let $G$ be any finite group. The group algebra $\Bbbk G$ is a finite-dimensional Hopf algebra with  $\HDelta(g) = g\otimes g$, $\Hvarep(g) = 1$, and $S(g) = g^{-1}$, for all $g\in G$. This Hopf algebra admits a normalized (co)integral pair given by $\Lambda := \sum_{h \in G} h$ and $\lambda(g):= \delta_{e,g}1_{\Bbbk}$. Applying $\Psi$ to this integral Hopf algebra, we obtain the Frobenius structure on $\Bbbk G$ described in Example \ref{ex: KG} and~\eqref{eq:DeltaA}, where $\Delta(g):= \sum_{h \in G}g h \otimes h^{-1}$ and $\varep(g):= \lambda(g) = \delta_{e,g}1_\Bbbk$, for all $g\in G$.
\end{example}

\begin{proposition} \label{prop:intHopfext}
If $H \in \mathsf{IntHopfAlg}(\cC)$ is equipped with $\theta: \one \to H \in \cC$ such that $m(\theta\otimes \theta)= u \hs\Hvarep\hs \Lambda,$
then the Frobenius algebra $\Psi(H)$ from Proposition~\ref{prop:hopftofrob} is extendable.
In particular, when $\cC= \mathsf{Vec}$, the Frobenius algebra $\Psi(H)$ over $\Bbbk$ is extendable with $\phi=\id_{\Psi(H)}$ and $\theta=\pm \sqrt{\Hvarep(\Lambda(1_{\Bbbk}))} \hs u.$
\end{proposition}

\begin{proof}
Suppose that the morphism $\theta:\one \to H$ as in the statement exists. Then, taking $\phi = \id_{\Psi(H)}$, and using this $\theta$, we extend the Frobenius structure. To verify the axioms of Definition~\ref{def:extFrobC}(b), notice that
conditions~(i) and~(ii) hold trivially. Condition (iii) is verified in Figure \ref{fig:prop4.7}; using notation and axioms from Appendix~\ref{sec:HF}.
The last statement on the case when $\cC=\mathsf{Vec}$ is clear.
 \end{proof}
\begin{figure}[h!]
\scalebox{.75}{

\tikzset{every picture/.style={line width=0.75pt}} 

\begin{tikzpicture}[x=0.75pt,y=0.75pt,yscale=-1,xscale=1]

\draw [line width=1.5]    (266.61,48.74) -- (266.61,32.62) ;
\draw [shift={(266.61,28.62)}, rotate = 90] [fill={rgb, 255:red, 0; green, 0; blue, 0 }  ][line width=0.08]  [draw opacity=0] (11.61,-5.58) -- (0,0) -- (11.61,5.58) -- cycle    ;
\draw [line width=1.5]    (252.11,62.74) .. controls (252.61,44.24) and (281.61,43.74) .. (281.61,63.24) ;
\draw [line width=1.5]    (221.61,62.74) .. controls (222.61,83.74) and (251.61,85.74) .. (252.11,62.74) ;
\draw [line width=1.5]    (221.61,62.74) -- (221.55,29.95) ;
\draw [line width=1.5]    (237.11,98.74) -- (237.11,79.74) ;
\draw [line width=1.5]    (281.61,98.74) -- (281.61,63.24) ;
\draw  [fill={rgb, 255:red, 255; green, 255; blue, 255 }  ,fill opacity=1 ][line width=1.5]  (272.93,76.3) .. controls (272.93,71.6) and (276.74,67.8) .. (281.43,67.8) .. controls (286.13,67.8) and (289.94,71.6) .. (289.94,76.3) .. controls (289.94,81) and (286.13,84.8) .. (281.43,84.8) .. controls (276.74,84.8) and (272.93,81) .. (272.93,76.3) -- cycle ;
\draw [line width=1.5]    (277.83,75.9) -- (284.97,76.02) ;
\draw [line width=1.5]    (281.45,72.33) -- (281.45,79.93) ;
\draw  [fill={rgb, 255:red, 0; green, 0; blue, 0 }  ,fill opacity=1 ][line width=1.5]  (263.31,48.74) .. controls (263.31,46.92) and (264.78,45.44) .. (266.61,45.44) .. controls (268.43,45.44) and (269.9,46.92) .. (269.9,48.74) .. controls (269.9,50.56) and (268.43,52.03) .. (266.61,52.03) .. controls (264.78,52.03) and (263.31,50.56) .. (263.31,48.74) -- cycle ;
\draw   (10.05,10.14) -- (570.96,10.14) -- (570.96,160.6) -- (10.05,160.6) -- cycle ;
\draw [line width=1.5]    (40.51,80.14) .. controls (41.28,50.35) and (81.28,50.35) .. (80.62,79.68) ;
\draw [line width=1.5]    (61.41,57.58) -- (61.41,38.08) ;
\draw [line width=1.5]    (80.61,93.18) -- (80.62,79.68) ;
\draw  [fill={rgb, 255:red, 255; green, 255; blue, 255 }  ,fill opacity=1 ][line width=1.5]  (33.32,87.32) .. controls (33.32,83.35) and (36.54,80.14) .. (40.51,80.14) .. controls (44.47,80.14) and (47.69,83.35) .. (47.69,87.32) .. controls (47.69,91.29) and (44.47,94.5) .. (40.51,94.5) .. controls (36.54,94.5) and (33.32,91.29) .. (33.32,87.32) -- cycle ;
\draw [line width=1.5]    (35.01,92.02) -- (45.41,82.78) ;
\draw [line width=1.5]    (40.51,94.5) .. controls (40.34,123.04) and (80.61,123.18) .. (80.61,93.18) ;
\draw [line width=1.5]    (61.41,139.98) -- (61.54,116.2) ;
\draw  [fill={rgb, 255:red, 255; green, 255; blue, 255 }  ,fill opacity=1 ][line width=1.5]  (57.68,34.35) .. controls (57.68,32.29) and (59.35,30.62) .. (61.41,30.62) .. controls (63.47,30.62) and (65.14,32.29) .. (65.14,34.35) .. controls (65.14,36.41) and (63.47,38.08) .. (61.41,38.08) .. controls (59.35,38.08) and (57.68,36.41) .. (57.68,34.35) -- cycle ;
\draw [line width=1.5]    (130.51,80.14) .. controls (131.28,50.35) and (171.28,50.35) .. (170.62,79.68) ;
\draw [line width=1.5]    (151.41,57.58) -- (151.41,38.08) ;
\draw [line width=1.5]    (170.61,93.18) -- (170.62,79.68) ;
\draw [line width=1.5]    (130.51,94.5) .. controls (130.34,123.04) and (170.61,123.18) .. (170.61,93.18) ;
\draw [line width=1.5]    (151.41,139.98) -- (151.54,116.2) ;
\draw  [fill={rgb, 255:red, 255; green, 255; blue, 255 }  ,fill opacity=1 ][line width=1.5]  (147.68,34.35) .. controls (147.68,32.29) and (149.35,30.62) .. (151.41,30.62) .. controls (153.47,30.62) and (155.14,32.29) .. (155.14,34.35) .. controls (155.14,36.41) and (153.47,38.08) .. (151.41,38.08) .. controls (149.35,38.08) and (147.68,36.41) .. (147.68,34.35) -- cycle ;
\draw [line width=1.5]    (130.51,94.5) -- (130.51,80.14) ;
\draw [line width=1.5]    (237.11,98.74) .. controls (236.94,127.27) and (281.61,128.74) .. (281.61,98.74) ;
\draw [line width=1.5]    (260.14,139.98) -- (259.94,121.4) ;
\draw [line width=1.5]    (340.11,80.54) .. controls (340.88,50.75) and (380.88,50.75) .. (380.22,80.08) ;
\draw [line width=1.5]    (361.01,57.98) -- (361.01,38.48) ;
\draw [line width=1.5]    (340.11,94.9) .. controls (339.94,123.44) and (380.21,123.58) .. (380.21,93.58) ;
\draw [line width=1.5]    (361.01,140.38) -- (361.14,116.6) ;
\draw [line width=1.5]    (340.11,94.9) -- (340.11,80.54) ;
\draw  [fill={rgb, 255:red, 255; green, 255; blue, 255 }  ,fill opacity=1 ][line width=1.5]  (371.33,89.1) .. controls (371.33,84.4) and (375.14,80.6) .. (379.83,80.6) .. controls (384.53,80.6) and (388.34,84.4) .. (388.34,89.1) .. controls (388.34,93.8) and (384.53,97.6) .. (379.83,97.6) .. controls (375.14,97.6) and (371.33,93.8) .. (371.33,89.1) -- cycle ;
\draw [line width=1.5]    (376.23,88.7) -- (383.37,88.82) ;
\draw [line width=1.5]    (379.85,85.13) -- (379.85,92.73) ;
\draw  [fill={rgb, 255:red, 0; green, 0; blue, 0 }  ,fill opacity=1 ][line width=1.5]  (357.71,57.98) .. controls (357.71,56.16) and (359.19,54.68) .. (361.01,54.68) .. controls (362.83,54.68) and (364.3,56.16) .. (364.3,57.98) .. controls (364.3,59.8) and (362.83,61.27) .. (361.01,61.27) .. controls (359.19,61.27) and (357.71,59.8) .. (357.71,57.98) -- cycle ;
\draw [line width=1.5]    (361.01,48.54) -- (361.01,32.42) ;
\draw [shift={(361.01,28.42)}, rotate = 90] [fill={rgb, 255:red, 0; green, 0; blue, 0 }  ][line width=0.08]  [draw opacity=0] (11.61,-5.58) -- (0,0) -- (11.61,5.58) -- cycle    ;
\draw [line width=1.5]    (444.57,74.44) -- (444.6,34.42) ;
\draw [shift={(444.61,30.42)}, rotate = 90.05] [fill={rgb, 255:red, 0; green, 0; blue, 0 }  ][line width=0.08]  [draw opacity=0] (11.61,-5.58) -- (0,0) -- (11.61,5.58) -- cycle    ;
\draw  [fill={rgb, 255:red, 0; green, 0; blue, 0 }  ,fill opacity=1 ][line width=1.5]  (441.31,62.78) .. controls (441.31,60.96) and (442.79,59.48) .. (444.61,59.48) .. controls (446.43,59.48) and (447.9,60.96) .. (447.9,62.78) .. controls (447.9,64.6) and (446.43,66.07) .. (444.61,66.07) .. controls (442.79,66.07) and (441.31,64.6) .. (441.31,62.78) -- cycle ;
\draw  [fill={rgb, 255:red, 255; green, 255; blue, 255 }  ,fill opacity=1 ][line width=1.5]  (440.84,78.18) .. controls (440.84,76.11) and (442.51,74.44) .. (444.57,74.44) .. controls (446.63,74.44) and (448.3,76.11) .. (448.3,78.18) .. controls (448.3,80.24) and (446.63,81.91) .. (444.57,81.91) .. controls (442.51,81.91) and (440.84,80.24) .. (440.84,78.18) -- cycle ;
\draw  [fill={rgb, 255:red, 255; green, 255; blue, 255 }  ,fill opacity=1 ][line width=1.5]  (440.84,97.78) .. controls (440.84,95.71) and (442.51,94.04) .. (444.57,94.04) .. controls (446.63,94.04) and (448.3,95.71) .. (448.3,97.78) .. controls (448.3,99.84) and (446.63,101.51) .. (444.57,101.51) .. controls (442.51,101.51) and (440.84,99.84) .. (440.84,97.78) -- cycle ;
\draw [line width=1.5]    (444.57,139.64) -- (444.57,101.51) ;
\draw [line width=1.5]    (506.63,74.39) -- (506.71,41.89) ;
\draw [line width=1.5]    (507.51,95) .. controls (507.34,123.54) and (546.51,124.39) .. (546.51,94.39) ;
\draw [line width=1.5]    (528.41,140.48) -- (528.54,116.7) ;
\draw  [fill={rgb, 255:red, 255; green, 255; blue, 255 }  ,fill opacity=1 ][line width=1.5]  (502.98,38.16) .. controls (502.98,36.09) and (504.65,34.42) .. (506.71,34.42) .. controls (508.78,34.42) and (510.45,36.09) .. (510.45,38.16) .. controls (510.45,40.22) and (508.78,41.89) .. (506.71,41.89) .. controls (504.65,41.89) and (502.98,40.22) .. (502.98,38.16) -- cycle ;
\draw  [fill={rgb, 255:red, 255; green, 255; blue, 255 }  ,fill opacity=1 ][line width=1.5]  (499.82,81.32) .. controls (499.82,77.35) and (503.04,74.14) .. (507.01,74.14) .. controls (510.97,74.14) and (514.19,77.35) .. (514.19,81.32) .. controls (514.19,85.29) and (510.97,88.5) .. (507.01,88.5) .. controls (503.04,88.5) and (499.82,85.29) .. (499.82,81.32) -- cycle ;
\draw [line width=1.5]    (499.82,81.32) -- (514.19,81.32) ;
\draw  [fill={rgb, 255:red, 255; green, 255; blue, 255 }  ,fill opacity=1 ][line width=1.5]  (539.32,81.82) .. controls (539.32,77.85) and (542.54,74.64) .. (546.51,74.64) .. controls (550.47,74.64) and (553.69,77.85) .. (553.69,81.82) .. controls (553.69,85.79) and (550.47,89) .. (546.51,89) .. controls (542.54,89) and (539.32,85.79) .. (539.32,81.82) -- cycle ;
\draw [line width=1.5]    (539.32,81.82) -- (553.69,81.82) ;
\draw [line width=1.5]    (546.63,74.39) -- (546.71,42.39) ;
\draw  [fill={rgb, 255:red, 255; green, 255; blue, 255 }  ,fill opacity=1 ][line width=1.5]  (542.98,38.66) .. controls (542.98,36.59) and (544.65,34.92) .. (546.71,34.92) .. controls (548.78,34.92) and (550.45,36.59) .. (550.45,38.66) .. controls (550.45,40.72) and (548.78,42.39) .. (546.71,42.39) .. controls (544.65,42.39) and (542.98,40.72) .. (542.98,38.66) -- cycle ;
\draw [line width=1.5]    (507.51,95) -- (507.01,88.5) ;
\draw [line width=1.5]    (546.51,94.39) -- (546.51,89) ;

\draw (98,82.5) node [anchor=north west][inner sep=0.75pt]   [align=left] {:= \ \ \ \ \ \ \ \ \ \ \ \ \ \ \ = \  \ \ \ \ \ \ \ \ \ \ \ \ \ \ \ \ \ \ \ = \ \  \ \ \ \ \ \ \ \ \ \ \ \ \ \ \ = \ \ \ \ \ \ \ \ \ \ =};
\draw (90.5,65.6) node [anchor=north west][inner sep=0.75pt] [font=\small]  [align=left] {(def) \ \ \ \ \ \ \ \ \ \ \  Fig.\hs\ref{fig:DelEp} \ \ \ \ \ \ \ \ \ \ \ \ \ \ \ \ \ (H2) \ \ \ \ \ \ \ \ \ \ \ \ \ \ \ (H9) \ \ \ \ \ \ \ (hyp)};

\end{tikzpicture}

} \vspace{-.1in}
        \caption{Proof of Definition~\ref{def:extFrobC}(b)(iii) for Proposition~\ref{prop:intHopfext}.}
        \label{fig:prop4.7}
    \end{figure}
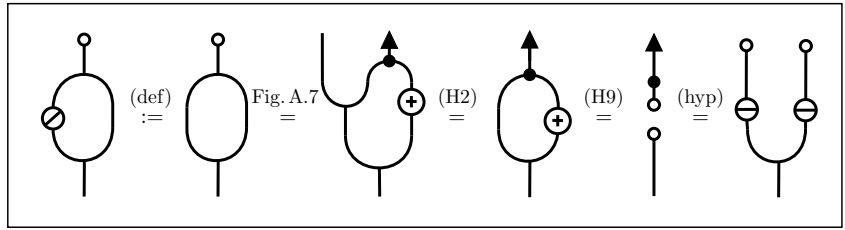

\begin{example}
Let $G$ be a finite group, and recall that the group algebra $\Bbbk G$ has a Hopf algebra structure, which induces a Frobenius algebra structure, as described in Example \ref{ex: Psi(kG)}. In this case, we have that $u\Hvarep(\Lambda) = u(\Hvarep(\sum_{h\in G} h) = u(\sum_{h\in G}1_{\Bbbk}) = |G|\cdot u(1_{\Bbbk}) = |G|\cdot e_G$. The above proposition then tells us that the choice $\phi =\id_{\Bbbk G}$ and $\theta = \pm\sqrt{|G|}\cdot e_G$ extends the induced Frobenius algebra structure on $\Bbbk G$. Note that this is the same extended Frobenius structure as introduced in Example~\ref{ex: KG}.
\end{example}


\subsection{Extended Hopf algebras}  \label{sec:eHopf}
Continue to let $\cC$ be a symmetric monoidal category. Here, we introduce extended Hopf algebras in $\cC$.

\begin{definition}\label{def:exthopf}
An integral Hopf algebra $(H,m,u,\HDelta,\Hvarep,S, S^{-1}, \Lambda,\lambda)$ is called {\it extended} if it is equipped with two morphisms $\phi:H\to H$ and $\theta:\one\to H$ in $\cC$ satisfying the following axioms:
\begin{enumerate}\renewcommand{\labelenumi}{(\roman{enumi})}
\item $\phi$ is a morphism of integral Hopf algebras such that $\phi^2 = \id_H$;
\smallskip
\item $\phi\hs m(\theta\otimes\id_H) = m(\theta\otimes\id_H)$;
\smallskip
\item $m(\phi\otimes S)\HDelta\Lambda = m(\theta\otimes\theta)$.
\end{enumerate}
A \textit{morphism of extended Hopf algebras} $f:(H,\phi,\theta)\to (H',\phi',\theta')$ is a morphism of integral Hopf algebras in $\cC$ which also satisfies $f\phi = \phi' f$ and $f\theta = \theta'$.
\end{definition}

We use the above to define a category $\mathsf{ExtHopfAlg}(\cC)$. Also, consider the forgetful functor,
\begin{align*}
U: \mathsf{ExtHopfAlg}(\cC) &\to \mathsf{IntHopfAlg}(\cC)\\
(H,\hspace{0.01in} m, \hspace{0.01in} u,  \hspace{0.01in} \HDelta, \hspace{0.01in} \Hvarep, \hspace{0.01in} S, \hspace{0.01in} S^{-1}, \hspace{0.01in} \Lambda, \hspace{0.01in}\lambda, \hspace{0.01in} \phi, \hspace{0.01in} \theta) 
&\mapsto (H,\hspace{0.01in} m, \hspace{0.01in} u,  \hspace{0.01in} \HDelta, \hspace{0.01in} \Hvarep, \hspace{0.01in} S, \hspace{0.01in} S^{-1}, \hspace{0.01in} \Lambda, \hspace{0.01in}\lambda).
\end{align*}
We have the following result.

\begin{proposition} \label{prop:eHopf}
Take $H \in \mathsf{ExtHopfAlg}(\cC)$. Then, the Frobenius algebra $\Psi U(H)$ in $\cC$ from Proposition \ref{prop:hopftofrob}  is extendable via the morphisms $\phi$ and~$\theta$.
\end{proposition}

\begin{proof}
We will verify that $\phi$ and $\theta$ extend the Frobenius algebra $\Psi U(H)$ by checking the axioms of Definition~\ref{def:extFrobC}(b).
Since $\phi:(H,m,u,\HDelta,\Hvarep, S, S^{-1}, \Lambda, \lambda)\to (H,m,u,\HDelta,\Hvarep, S, S^{-1}, \Lambda, \lambda)$ is a morphism of integral Hopf algebras, the functoriality of $\Psi$ and $U$ gives that $\phi: (H,m,u,\Delta,\varep)\to (H,m,u,\Delta,\varep)$ is a Frobenius algebra morphism. Moreover, we have that $\phi^2 = \id_H$ by Definition \ref{def:exthopf}(i). So, condition~(i) of Definition~\ref{def:extFrobC}(b) holds. Condition~(ii) of Definition~\ref{def:extFrobC}(b) also holds by Definition~\ref{def:exthopf}(ii) since the multiplication morphism is the same for both the Hopf and Frobenius structures on $H$. Towards condition~(iii) of Definition~\ref{def:extFrobC}(b), we compute:
\[
m(\phi\otimes\id_H)\Delta u \; = \; m(\phi\otimes S)(m\otimes\id_H)(u\otimes\HDelta)\Lambda \; = \; m(\theta\otimes\theta),
\]

\noindent where the first equality is the definition of $\Delta$ and a level exchange, and the second equality is by the unitality of $m$ and $u$ and Definition \ref{def:exthopf}(iii).
\end{proof}

The consequence below is straight-forward.

\begin{corollary}
There is a functor $\HPsi:\mathsf{ExtHopfAlg}(\cC)\to \mathsf{ExtFrobAlg}(\cC)$ which sends an extended Hopf algebra $(H,m,u,\HDelta,\Hvarep,S,S^{-1},\Lambda,\lambda,\phi,\theta)$ to the extended Frobenius algebra $(H,m,u,\Delta,\varep,\phi,\theta)$, with $\Delta$ and $\varep$ defined in Proposition~\ref{prop:hopftofrob}, and which acts as the identity on morphisms. \qed 
\end{corollary}

\begin{remark}
While the above result tells us that every extended Hopf algebra gives rise to an extended Frobenius algebra via the same $\phi$ and $\theta$, the converse is not true. In particular, given $H \in \mathsf{IntHopfAlg}(\cC)$, we get that $\Psi(H) \in \mathsf{FrobAlg}(\cC)$. If $\Psi(H)$ is extendable via $\phi_{\Psi(H)}$ and $\theta_{\Psi(H)}$, it is not necessarily true that $(H, \phi_{\Psi(H)}, \theta_{\Psi(H)})$ is an extended Hopf algebra in $\cC$. 

For instance,  consider the Frobenius algebra structure on $\Bbbk C_2$, induced by the Hopf structure, as described in Example \ref{ex: Psi(kG)}. This Frobenius structure can be extended by taking $\phi(g) = -g$ (where $g$ is a generator of $C_2$) and $\theta = 0$, as in Proposition \ref{prop:C2-class}(b). However, this choice of $\phi$ and $\theta$ does not extend the integral Hopf structure on $\Bbbk C_2$, since $\phi$ is not comultiplicative with respect to $\HDelta$. 
\end{remark}


\section{Extended Frobenius monoidal functors}
\label{sec:eFrobMF}
In this section, we introduce the construction of an extended Frobenius monoidal functor, which preserves extended Frobenius algebras [Proposition~\ref{prop:eFMpres}]. Background material is covered in Section~\ref{sec:back-MF}, and the main construction is covered in Section~\ref{sec:result-EMF}. Examples are presented in Section~\ref{sec:examples-EMF}.


\subsection{Background on monoidal functors} \label{sec:back-MF} We can move between monoidal categories in several ways. Consider the terminology below, along with the references, \cite[Chapter~6]{Bohm}, \cite{DayPastro}, \cite[Sections~1.4 and~7.5]{TuraevVirelizier}, \cite[Section~3.2]{Walton}, for  details about the material in this part. 

\begin{definition} \label{def:monfunc} Take a functor $F:\cC\to\cC'$ between monoidal categories $(\cC, \otimes, \one)$ and $(\cC', \otimes', \one').$
\begin{enumerate}[\upshape (a)]
\item We say that $F$ is a {\it monoidal functor} if it is equipped with a natural transformation \linebreak $F^{(2)}:=\{F^{(2)}_{X,Y}:  F(X) \otimes' F(Y) \to F(X \otimes Y)\}_{X,Y \in \cC}$, and  a morphism $F^{(0)}: \one' \to F(\one)$ in $\cC'$, 
that satisfy associativity and unitality constraints.
\smallskip
\item A monoidal functor $(F,F^{(2)}, F^{(0)})$ is said to be {\it strong} if $F^{(2)}$ is a natural isomorphism and $F^{(0)}$ is an isomorphism. In this case, denote $F^{(-2)}_{X,Y}:= (F^{(2)}_{X,Y})^{-1}$ and $F^{(-0)}:=(F^{(0)})^{-1}$.
\smallskip
\item We say that $F$ is a {\it comonoidal functor} if it is equipped with a natural transformation \linebreak $F_{(2)}:=\{F_{(2)}^{X,Y}:  F(X \otimes Y) \to F(X) \otimes' F(Y)\}_{X,Y \in \cC}$, and  a morphism $F_{(0)}: F(\one) \to \one'$ in $\cC'$, 
that satisfy coassociativity and counitality constraints.
\smallskip
\item We say that $F$ is a {\it Frobenius monoidal functor} if it is part of a tuple $(F, F^{(2)}, F^{(0)}, F_{(2)}, F_{(0)})$, where $(F, F^{(2)}, F^{(0)})$ is a monoidal functor, and $(F, F_{(2)}, F_{(0)})$ is a comonoidal functor, subject to the Frobenius conditions, for all $X,Y,Z \in \cC$:
\[
\begin{array}{rl}
\bigl(F^{(2)}_{X,Y} \otimes' \id_{F(Z)}\bigr)\bigl(\id_{F(X)} \otimes' F_{(2)}^{Y,Z}\bigr) 
&= \;  F_{(2)}^{X \otimes Y, Z} \circ F^{(2)}_{X, Y \otimes Z} \;,\medskip\\
\bigl(\id_{F(X)} \otimes' F^{(2)}_{Y,Z}\bigr)\bigl(F_{(2)}^{X,Y} \otimes' \id_{F(Z)}\bigr) 
&= \; F_{(2)}^{X, Y \otimes Z} \circ F^{(2)}_{X \otimes Y, Z}.
\end{array}
\]

\item A Frobenius monoidal functor $(F, F^{(2)}, F^{(0)}, F_{(2)}, F_{(0)})$ is {\it separable} if for each $X,Y \in \cC$:
\[
F^{(2)}_{X,Y} \circ F_{(2)}^{X,Y} = \id_{F(X \otimes Y)}.
\] 
\end{enumerate}
\end{definition}

Moreover,  consider the transformations of (co)monoidal functors below.

\begin{definition} \label{def:monNT} Take monoidal categories $\cC:=(\cC,\otimes, \one)$ and $\cC':=(\cC',\otimes', \one')$.
\begin{enumerate}[\upshape (a)]
\item A {\it monoidal natural transformation} from a monoidal functor $(F,F^{(2)},F^{(0)}): \cC \to \cC'$ to a monoidal functor $(G,G^{(2)},G^{(0)}): \cC \to \cC'$ is a natural transformation $\phi: F \Rightarrow G$ such that 
\[
\phi_{X \otimes Y} \circ F^{(2)}_{X,Y} = G^{(2)}_{X,Y} \circ (\phi_X \otimes' \phi_Y)\; \text{ for all $X,Y \in \cC$}, \quad \qquad \phi_{\one} \circ  F^{(0)}= G^{(0)}.
\]
\item A {\it comonoidal natural transformation} from a comonoidal functor $(F,F_{(2)},F_{(0)}): \cC \to \cC'$ to a comonoidal functor $(G,G_{(2)},G_{(0)}): \cC \to \cC'$ is a natural transformation $\phi: F \Rightarrow G$ such~that 
\[
(\phi_X \otimes' \phi_Y) \circ F_{(2)}^{X,Y} = G_{(2)}^{X,Y} \circ \phi_{X \otimes Y}  \; \text{ for all $X,Y \in \cC$}, \quad \qquad   F_{(0)} = G_{(0)} \circ \phi_{\one}.
\]
\item A {\it Frobenius monoidal natural transformation} is a natural transformation $\phi: F \Rightarrow G$ between Frobenius monoidal functors $(F,F^{(2)},F^{(0)},F_{(2)},F_{(0)})$ and $(G,G^{(2)},G^{(0)},G_{(2)},G_{(0)})$ from $\cC$ to $\cC'$ that is monoidal for the underlying monoidal functor structure and comonoidal for the underlying comonoidal functor structure. 
\end{enumerate}
\end{definition}

Next, we see in the result below that the various types of functors in Definition~\ref{def:monfunc} preserve the corresponding algebraic structures introduced in Section~\ref{sec:algstr} and Definition~\ref{def:sep}.

\begin{proposition} \label{prop:pres} \cite[Proposition~4.3]{Walton} \cite[Corollary~5]{DayPastro} \cite[Lemma~6.10]{Bohm} Take monoidal categories $\cC$ and $\cC'$.
\begin{enumerate}[\upshape (a)]
\item A monoidal functor $(F, F^{(2)}, F^{(0)}): \cC \to \cC'$ yields  $\mathsf{Alg}(F): \mathsf{Alg}(\cC) \to \mathsf{Alg}(\cC')$, a functor where $\mathsf{Alg}(F)(A,m_A,u_A)$ is defined as  
\[\bigl(F(A), \;  \; m_{F(A)}:= F(m_A) \hs F^{(2)}_{A,A}, \; \; u_{F(A)}:= F(u_A) \hs F^{(0)}\bigr).
\]

\smallskip

\item  A comonoidal functor $(F, F_{(2)}, F_{(0)}): \cC \to \cC'$ yields $\mathsf{Coalg}(F): \mathsf{Coalg}(\cC)\to\mathsf{Coalg}(\cC')$,  a functor where $\mathsf{Coalg}(F)(A, \Delta_A, \varep_A)$ is defined as 
\[
\bigl(F(A), \;  \; \Delta_{F(A)}:= F_{(2)}^{A,A} \hs F(\Delta_A), \; \; \varep_{F(A)}:= F_{(0)} \hs F(\varep_A)\bigr).
\]
\item Moreover, a Frobenius monoidal functor $(F,F^{(2)},F^{(0)}, F_{(2)}, F_{(0)}): \cC \to \cC'$ yields a functor $\mathsf{FrobAlg}(F): \mathsf{FrobAlg}(\cC) \to \mathsf{FrobAlg}(\cC')$  by using the formulas from parts (a) and (b).
\medskip
\item Likewise, a separable Frobenius monoidal functor $(F, F^{(2)},F^{(0)},F_{(2)}, F_{(0)}): \cC \to \cC'$  yields a functor $\mathsf{SepFrobAlg}(\cC) \to \mathsf{SepFrobAlg}(\cC')$ by using the formulas from parts (a) and (b). \qed
\end{enumerate}
\end{proposition}

One nice feature of the functors here is that they are closed under composition.

\begin{proposition} \label{prop:comp-bkg} \cite[Exercise~3.4]{Walton} \cite[Proposition~4]{DayPastro} \cite[Exercises~3.10 and~6.4]{Bohm} 
Take monoidal categories $\cC$, $\cC'$, and $\cC''$.
\begin{enumerate}[\upshape (a)]
\item Let $(F,F^{(2)},F^{(0)}):\cC \to \cC'$ and $(G,G^{(2)},G^{(0)}):\cC' \to \cC''$ be monoidal functors. Then, the composition $GF: \cC \to \cC''$ is monoidal, with $(GF)^{(2)}$ and $(GF)^{(0)}$ defined by:
\[
(GF)^{(2)}_{X,Y}:=G(F^{(2)}_{X,Y})\circ  G^{(2)}_{F(X),F(Y)} \;\;  \forall X,Y \in \cC, \qquad (GF)^{(0)}:=G(F^{(0)})\circ  G^{(0)}.
\]
\item Let $(F,F_{(2)},F_{(0)}):\cC \to \cC'$ and $(G,G_{(2)},G_{(0)}):\cC' \to \cC''$ be comonoidal functors. Then, the composition $GF: \cC \to \cC''$ is comonoidal, with $(GF)_{(2)}$ and $(GF)_{(0)}$ defined by:
\[
(GF)_{(2)}^{X,Y}:= G_{(2)}^{F(X),F(Y)}\circ G(F_{(2)}^{X,Y}) \;\;  \forall X,Y \in \cC, \qquad (GF)_{(0)}:=G_{(0)}\circ  G(F_{(0)}).
\]
\item Let $(F, F^{(2)},F^{(0)},F_{(2)},F_{(0)}):\cC \to \cC'$ and $(G,G^{(2)},G^{(0)}, G_{(2)},G_{(0)}):\cC' \to \cC''$ be Frobenius monoidal functors. Then, the composition $GF: \cC \to \cC''$ is Frobenius monoidal by using the formulas from parts (a) and (b).
\medskip
\item The composition of two separable Frobenius monoidal functors is also separable Frobenius monoidal by using the formulas from parts (a) and (b). \qed
\end{enumerate}
\end{proposition}

\begin{remark} \label{rem:2cats} 
It is now straightforward to build the 2-category, 
\textsc{Mon} (resp., \textsc{Comon},  \textsc{FrobMon},  \textsc{SepFrobMon}),
via the data below.
\begin{enumerate}[\upshape (a)]
\item 0-cells are monoidal categories.

\smallskip

\item 1-cells are (resp., co-, Frobenius, separable Frobenius) monoidal  functors.

\smallskip

\item 2-cells are (resp., co-, Frobenius, Frobenius) monoidal natural transformations.

\smallskip

\item The identity 1-cell/2-cell is the identity (resp., co-, Frobenius, Frobenius) monoidal functor/natural transformation. 

\smallskip

\item Horizontal composition of 1-cells is given in Proposition~\ref{prop:comp-bkg}.

\smallskip

\item Vertical/horizontal composition of 2-cells is given by the standard vertical/horizontal composition of monoidal and comonoidal natural transformations. 
\end{enumerate}
See \cite[Section~4.10.3]{Walton} and references within, and also see \cite[Exercise~2.7.11]{JohnsonYau}. 
\end{remark}


\subsection{Main construction and results} \label{sec:result-EMF}

Here, we extend the results in Propositions \ref{prop:pres} and \ref{prop:comp-bkg} to the category $\mathsf{ExtFrobAlg}(\cC)$. In particular, we will define a type of functor that preserves extended Frobenius algebras, and then show that this type of functor is closed under composition.

\begin{definition} \label{def:eFMfunc}
A Frobenius monoidal functor $(F,F^{(2)},F^{(0)},F_{(2)},F_{(0)}):(\cC,\otimes,\one)\to(\cC',\otimes',\one')$ is called  an {\it extended Frobenius monoidal functor} (or is {\it extendable}) if there exist a natural transformation $\widehat{F}:F\Rightarrow F$ and a morphism $\widecheck{F}:\one'\to F(\one) \in \cC'$ such that the conditions below hold.
\begin{enumerate}[\upshape (a)]
\item $\widehat{F}$ is a Frobenius monoidal natural transformation.

\smallskip

\item $F^{(2)}_{\one,\one}\circ(\widehat{F}_{\one}\otimes'\id_{F(\one)})\circ F_{(2)}^{\one,\one}\circ F^{(0)}=F^{(2)}_{\one,\one}\circ(\widecheck{F}\otimes'\widecheck{F})$.

\smallskip

\item The following are true for each $X,Y\in\cC$:
\begin{enumerate}[\upshape (i)]

\smallskip

\item $\widehat{F}_X\circ\widehat{F}_X=\id_{F(X)}$;

\smallskip

\item $\widehat{F}_{\one\otimes X}\circ F^{(2)}_{\one,X}\circ(\widecheck{F}\otimes' \id_{F(X)})=F^{(2)}_{\one,X}\circ(\widecheck{F}\otimes' \id_{F(X)})$;

\smallskip

\item $F^{(2)}_{X,Y}\circ(\widehat{F}_{X}\otimes'\id_{F(Y)})\circ F_{(2)}^{X,Y}=F^{(2)}_{X\otimes Y,\one}\circ(\widehat{F}_{X\otimes Y}\otimes'\id_{F(\one)})\circ F_{(2)}^{X\otimes Y,\one}$.
\end{enumerate}
\end{enumerate}

\smallskip

Part (b) is represented by the following commutative diagram.
{\small
\[
\begin{tikzcd}[row sep=0.5cm]
	{\one'} && {F(\one)} && {F(\one)\otimes' F(\one)} \\
	&&&& {F(\one)\otimes' F(\one)} \\
	{F(\one)\otimes'F(\one)} &&&& {F(\one)}
	\arrow["{F^{(0)}}", from=1-1, to=1-3]
	\arrow["{\widecheck{F}\otimes'\widecheck{F}}"', from=1-1, to=3-1]
	\arrow["{F^{\one,\one}_{(2)}}", from=1-3, to=1-5]
	\arrow["{\widehat{F}_{\one}\otimes'\id_{F(\one)}}", from=1-5, to=2-5]
	\arrow["{F_{\one,\one}^{(2)}}", from=2-5, to=3-5]
	\arrow["{F_{\one,\one}^{(2)}}", from=3-1, to=3-5]
\end{tikzcd}
\]
}

Parts (c)(ii,iii) are represented by the left and right diagrams below, respectively.
{\small
\[
\begin{tikzcd}[row sep=0.5cm]
	{\one'\otimes'F(X)} && {F(\one)\otimes'F(X)} \\
	{F(\one)\otimes'F(X)} \\
	{F(\one\otimes X)} && {F(\one\otimes X)}
	\arrow["{\widecheck{F}\otimes'\id_{F(X)}}", from=1-1, to=1-3]
	\arrow["{\widecheck{F}\otimes'\id_{F(X)}}"', from=1-1, to=2-1]
	\arrow["{F^{(2)}_{\one,X}}", from=1-3, to=3-3]
	\arrow["{F^{(2)}_{\one,X}}"', from=2-1, to=3-1]
	\arrow["{\widehat{F}_{\one\otimes X}}", from=3-1, to=3-3]
\end{tikzcd}
\qquad
\begin{tikzcd}[row sep=0.5cm]
	{F(X\otimes Y)} && {F(X\otimes Y)\otimes'F(\one)} \\
	{F(X)\otimes' F(Y)} && {F(X\otimes Y)\otimes'F(\one)} \\
	{F(X)\otimes' F(Y)} && {F(X\otimes Y)}
	\arrow["{F_{(2)}^{X\otimes Y,\one}}", from=1-1, to=1-3]
	\arrow["{F_{(2)}^{X,Y}}"', from=1-1, to=2-1]
	\arrow["{\widehat{F}_{X\otimes Y}\otimes'\id_{F(\one)}}", from=1-3, to=2-3]
	\arrow["{\widehat{F}_{X}\otimes'\id_{F(Y)}}"', from=2-1, to=3-1]
	\arrow["{F^{(2)}_{X\otimes Y,\one}}", from=2-3, to=3-3]
	\arrow["{F^{(2)}_{X,Y}}", from=3-1, to=3-3]
\end{tikzcd}
\]
}
\end{definition}

\pagebreak

 Extended Frobenius monoidal functors are plentiful. Specifically, we have the following result; compare to Proposition~\ref{prop:sepExt}.

\begin{proposition} \label{prop:sepisEFM}
Separable Frobenius monoidal functors admit the structure of extended Frobenius monoidal functors.
\end{proposition}

\begin{proof}
Let $(F,F^{(2)},F^{(0)},F_{(2)},F_{(0)})$ be a separable Frobenius monoidal functor. Then, take $\widehat{F}=\Id_F$ and $\widecheck{F}=F^{(0)}$. It is then straightforward to verify that these choices of $\widehat{F}$ and $\widecheck{F}$ extend the Frobenius monoidal structure on $F$.
\end{proof}

\begin{example}
Strong monoidal functors are separable with $F_{(2)} := F^{(-2)}$ and $F_{(0)}:=F^{(-0)}$, so they are also extended Frobenius monoidal functors.
\end{example}

The next result is the desired extension of Proposition \ref{prop:pres}. See Appendix~B.1 for the proof (in the ArXiv preprint version of this article); namely, it involves lengthy commutative diagram arguments to verify that the formulas in the statement below yield an extended Frobenius algebra.

\begin{proposition} \label{prop:eFMpres}
An extended Frobenius monoidal functor $(F,F^{(2)},F^{(0)}, F_{(2)},F_{(0)},\widehat{F},\widecheck{F}):\cC\to\cC'$ induces a functor $\mathsf{ExtFrobAlg}(\cC)\to\mathsf{ExtFrobAlg}(\cC')$. For $A \in\mathsf{ExtFrobAlg}(\cC)$, we get $m_{F(A)}$,  $u_{F(A)}$, $\Delta_{F(A)}$, $\varep_{F(A)}$  as in Proposition~\ref{prop:pres}(a,b), with $\phi_{F(A)}= F(\phi_A)\hs \widehat{F}_A$ and $\theta_{F(A)}=F(\theta_A) \hs\widecheck{F}$. \hfill\qed
\end{proposition}

Since separable Frobenius monoidal functors are extended by Propositions \ref{prop:sepisEFM}, we obtain the following corollary of Proposition \ref{prop:eFMpres}.

\begin{corollary}
If $(F, F^{(2)}, F^{(0)}, F_{(2)}, F_{(0)}): \cC \to \cC'$ is a separable Frobenius monoidal functor, then it induces a functor $\ExtFrobAlg(\cC) \to \ExtFrobAlg(\cC')$. \qed
\end{corollary}

Now that we have defined extended Frobenius monoidal functors, the natural next thing to do is to arrange them into a 2-category. To do this, we need the following result, which extends Proposition \ref{prop:comp-bkg} to extended Frobenius monoidal functors. The proof of this theorem can be found in Appendix~B.2 (in the ArXiv preprint version of this article).

\begin{theorem} \label{thm:eFMcomp}
The composition of two extended Frobenius monoidal functors is again an extended Frobenius monoidal functor.\hfill\qed
\end{theorem}

To prove this, we let $(GF)^{(2)}$,  $(GF)^{(0)}$, $(GF)_{(2)}$, $(GF)_{(0)}$ be as in Proposition \ref{prop:comp-bkg}(a,b). Proposition \ref{prop:comp-bkg}(c) then implies that $GF$ is a Frobenius monoidal functor. We also define $\widehat{GF}:GF\Rightarrow GF$ by $\widehat{GF}_X:= G(\widehat{F}_X)\circ \widehat{G}_{F(X)}$ for all $X\in\cC$, and define $\widecheck{GF}:=G(\widecheck{F})\circ \widecheck{G}:\one''\to GF(\one)$.

\begin{remark} \label{rem:2-cat}
The collection of monoidal categories, extended Frobenius monoidal functors, and Frobenius natural transformations compatible with the extended Frobenius monoidal structures forms a 2-category, \textsc{ExtFrobMon}. Compare to Remark \ref{rem:2cats}. 
\end{remark}

\begin{remark} \label{rem:sp-case}
One can also obtain Proposition~\ref{prop:eFMpres} as a consequence of Theorem~\ref{thm:eFMcomp}. Take the  monoidal category $\overline{\one}$ consisting of a single object $\one$ and morphism $\id_{\one}$. Then, a Frobenius monoidal functor $(E,E^{(2)},E^{(0)},E_{(2)},E_{(0)}):\overline{\one}\to\cC$ is extendable if and only if $E(\one) \in \mathsf{ExtFrobAlg}(\cC)$. So, when $A\in\ExtFrobAlg(\cC)$, the functor $A^\#:\overline{\one}\to\cC$ with $A^\#(\one):=A$ is extended Frobenius monoidal.  Now if $(F,F^{(2)},F^{(0)},F_{(2)},F_{(0)},\widehat{F},\widecheck{F}):\cC\to\cC'$ is extended Frobenius monoidal, Theorem~\ref{thm:eFMcomp} implies that the functor $FA^\#:\overline{\one}\to\cC'$ is also extended Frobenius monoidal. Hence, $F(A)$ is an extended Frobenius algebra in $\cC'$ as in the proof of Proposition~\ref{prop:eFMpres}. Compare to \cite[Corollary~5]{DayPastro}.
\end{remark}


\subsection{Examples}  \label{sec:examples-EMF}

Following up with Propositions~\ref{prop:Ext-mon} and~\ref{prop:eFMbiprod}, consider the examples of extended Frobenius monoidal functors below.

\begin{example}\label{Ex:Ext-mon}
Let $(\cC,\otimes,\one,c)$ be a symmetric monoidal category, with an extended Frobenius algebra $B \in\mathsf{ExtFrobAlg}(\cC)$. Then, the functor $-\otimes B:\cC\to\cC$ is extended Frobenius with
{\small
\[(-\otimes B)^{(2)}_{X,Y}:= (\id_{X\otimes Y}\otimes m_B)(\id_X\otimes c_{B,Y}\otimes\id_B),
\qquad (-\otimes B)_{(2)}^{X,Y}:= (\id_X\otimes c_{Y,B}\otimes\id_B)(\id_{X\otimes Y}\otimes \Delta_B),
\]
\[
(-\otimes B)^{(0)}:= u_B, \qquad(-\otimes B)_{(0)}:=\varep_B, \qquad 
\widehat{(-\otimes B)}_X:= \id_X\otimes\phi_B, \qquad\widecheck{(-\otimes B)}:=\theta_B,\]
}
\noindent    for any $X,Y\in\cC$. We note further that when $B$ is not a separable Frobenius algebra, the Frobenius functor defined above is not separable.
\end{example}

\begin{example}\label{Ex:eFMbiprod}  Let $(\cC,\otimes,\one)$ be an additive monoidal category, with an extended Frobenius algebra $B\in\mathsf{ExtFrobAlg}(\cC)$. Then, the functor $-\osq B:\cC\to\cC$ is extended Frobenius with 
{\small
\[(-\osq B)^{(2)}_{X,Y}:= \pi_{X\otimes Y}\osq (m_B \circ \pi_{B\otimes B}),\qquad 
(-\osq B)_{(2)}^{X,Y}:= \iota_{X\otimes Y}\osq (\iota_{B\otimes B} \circ \Delta_B),
\]
\[(-\osq B)^{(0)}:= \id_\one\osq u_B, \qquad
(-\osq B)_{(0)}:= \id_\one\osq\varep_B,\qquad
\widehat{(-\osq B)}_X:= \pi_X\osq (\phi_B\circ\pi_B),  \qquad \widecheck{(-\osq B)}:=\id_\one\osq \theta_B,\]
}

\smallskip

\noindent for any $X,Y\in\cC$. Again, when $B$ is not a separable Frobenius algebra, the Frobenius functor defined above is not separable.
\end{example}


\appendix

\section{Graphical proof that integral Hopf implies Frobenius}
\label{sec:HF}
\setcounter{figure}{0}
\renewcommand{\thefigure}{A.\arabic{figure}}

 In this section, we give a graphical proof of Proposition~\ref{prop:hopftofrob}, showing that an integral Hopf algebra in a symmetric monoidal category $\cC$ is a Frobenius algebra in $\cC$. Recall axioms (S1) - (S5) from Figure~\ref{fig:symmon} in Section~\ref{sec:pre-mon}.1 above.

\subsection{Diagrams for integral Hopf algebras} 
Recall from Definition~\ref{def:HopfAlg} that a  Hopf algebra with invertible antipode in~$\cC$ is an object $H \in \cC$ equipped with morphisms $m:H \otimes H \to H$, $u: \one \to H$, $\HDelta: H \to H \otimes H$, $\Hvarep: H \to \one$, $S: H \to H$ with inverse $S^{-1}:H \to H$; this is  depicted in Figure~\ref{fig:Hopfmap}. These morphisms must satisfy the axioms in Figure~\ref{fig:Hopfax}. 
We also have that Hopf algebras with invertible antipode in~$\cC$ satisfy the identities in Figure~\ref{fig:Hopfiden}. Moreover, an integral and a cointegral of a Hopf algebra $H$ with invertible antipode in $\cC$ are given by morphisms $\Lambda: \one \to H$ and $\lambda: H \to \one$, respectively, satisfying the axioms depicted in Figure~\ref{fig:Hopfint}. 

\medskip


\begin{figure}[h!]
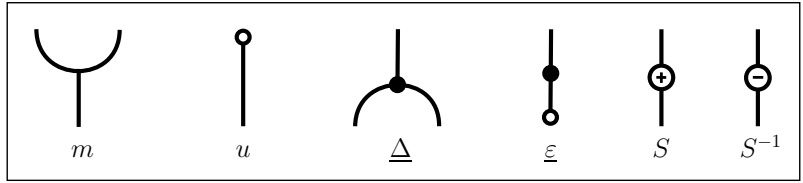

\scalebox{0.7}{
\tikzset{every picture/.style={line width=0.75pt}} 


}
\vspace{-.1in}
\caption{Structure morphisms for a Hopf algebra in $\cC$.}
\label{fig:Hopfmap}
\end{figure}


\tikzset{every picture/.style={line width=0.75pt}} 
\begin{figure}[h!]
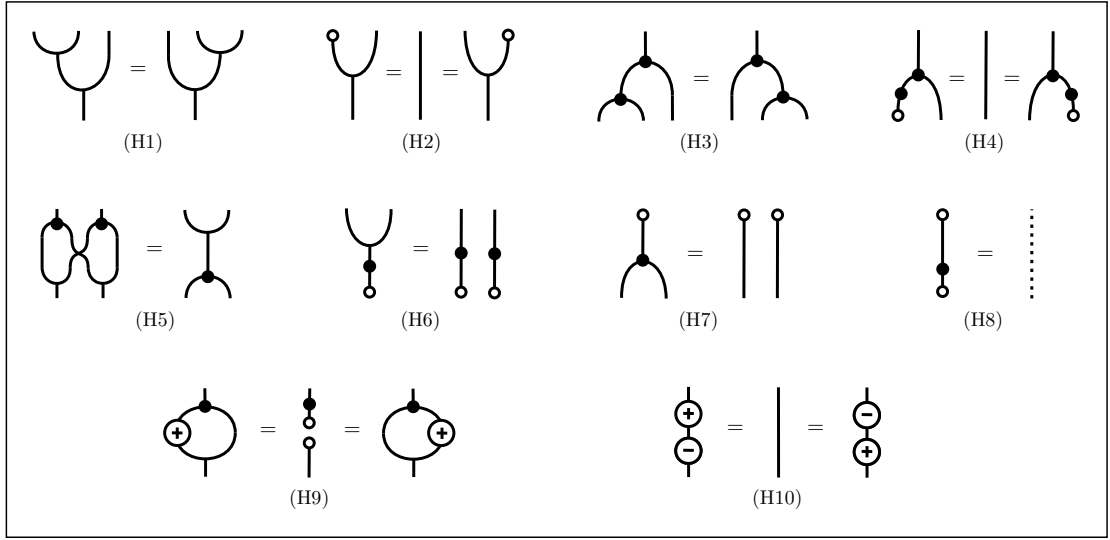

\scalebox{0.75}{
\tikzset{every picture/.style={line width=0.75pt}} 
%
}
\vspace{-.1in}
\caption{Axioms for a Hopf algebra with invertible antipode in $\cC$.}
\label{fig:Hopfax}
\end{figure}


\begin{figure}[h!]
\scalebox{0.75}{
\tikzset{every picture/.style={line width=0.75pt}} 
%
}
\vspace{-.1in}
\caption{Identities for a Hopf algebra in $\cC$.}
\label{fig:Hopfiden}
\end{figure}


\medskip


\begin{figure}[h!]
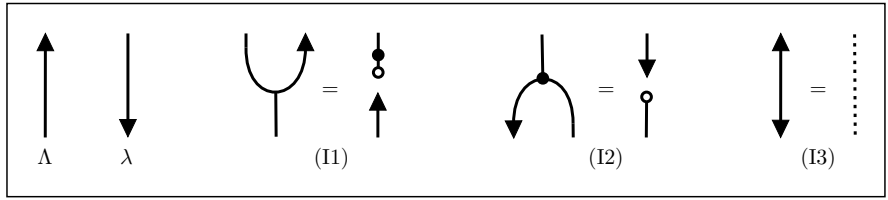

\scalebox{0.75}{
\tikzset{every picture/.style={line width=0.75pt}} 

%
}
\vspace{-.1in}
\caption{Normalized (co)integral for a Hopf algebra in $\cC$.}
\label{fig:Hopfint}
\end{figure}



\begin{lemma} \label{lem:appA}
We have the following identities.
\begin{enumerate}[\upshape (a)]
\item $(m  \otimes S)(\id_H \otimes \HDelta\hs \Lambda) = (\id_H \otimes m)(\id_H \otimes S \otimes \id_H) (\HDelta \hs m \otimes \id_H) (\id_H \otimes \Lambda \otimes \id_H) \HDelta$.
\smallskip
\item $\lambda \hs S \hs \Lambda = \id_\one$.
\end{enumerate}
\end{lemma}

\begin{proof}
Part (a) is proved in Figure~\ref{fig:lem(a)}, and part (b) is proved in Figure~\ref{fig:lem(b)}. References to Figures~\ref{fig:symmon},~\ref{fig:Hopfax},~\ref{fig:Hopfiden}, and~\ref{fig:Hopfint} are made throughout.
\end{proof}

\subsection{Proof of Proposition~\ref{prop:hopftofrob}}
We aim to show that
\begin{align*}
\Psi:\mathsf{IntHopfAlg}(\cC) &\to \mathsf{FrobAlg}(\cC)\\
(H,\hspace{0.01in} m, \hspace{0.01in} u, \hspace{0.01in} \HDelta, \hspace{0.01in} \Hvarep, \hspace{0.01in} S,  \hspace{0.01in} S^{-1}, \hspace{0.01in}  \Lambda, \hspace{0.01in} \lambda) 
&\mapsto (H, \hspace{0.01in} m, \hspace{0.01in} u, \hspace{0.01in} \Delta:= (m  \otimes S)(\id_H \otimes \HDelta\hs \Lambda), \;   \varep:=\lambda)
\end{align*}
is a well-defined functor, which acts as the identity on morphisms.

For the assignment of objects  under the functor $\Psi$, the coproduct $\Delta$ and counit $\varep$ are depicted in Figure~\ref{fig:DelEp}. The counitality axioms are then established in Figure~\ref{fig:counit}; the Frobenius laws are established in Figure~\ref{fig:Froblaw}; and the coassociativity axiom is established in Figure~\ref{fig:coassoc}. References to Figures~\ref{fig:Hopfax}--\ref{fig:lem(b)}  are made throughout. 

\pagebreak

Next, for the assignment of morphisms under $\Psi$, take a morphism of integral Hopf algebras 
\[
f:(H,m_H,u_H,\HDelta_H,\Hvarep_H,S_H^{\pm 1},\Lambda_H,\lambda_H)\to (K,m_K,u_K,\HDelta_K,\Hvarep_K,S_K^{\pm 1},\Lambda_K,\lambda_K).
\]
We will verify that $\Psi(f):=f$ is a morphism of Frobenius algebras from $(H,m_H,u_H,\Delta_H,\varep_H)$ to $(K,m_K,u_K,\Delta_K,\varep_K)$. We have multiplicativity and unitality for free, since the Hopf multiplications and units on $H$ and $K$ are the same as the Frobenius multiplications and units on $H$ and $K$. Next, we get Frobenius counitality immediately from the fact that $f$ is compatible with the cointegrals of $H$ and $K$; namely, the Frobenius counits of $H$ and $K$ are given by $\varep_H = \lambda_H$ and $\varep_K = \lambda_K$. Finally, we have that Frobenius comultiplicativity holds via the commutative diagram below.

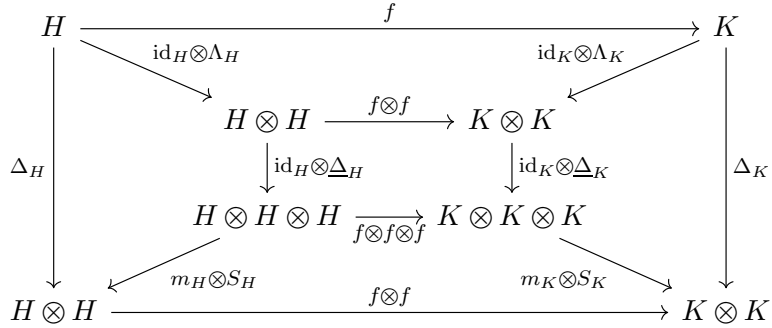
\vspace{.2in}

\begin{center}
\begin{tikzcd}
H \arrow[rrr,"f"]\arrow[ddd,"\Delta_H"']\arrow[rd,"\id_H\otimes\Lambda_H"] &&& K\arrow[ddd,"\Delta_K"]\arrow[ld,"\id_K\otimes\Lambda_K"']\\
&H\otimes H\arrow[d,"\id_H\otimes \HDelta_H"]\arrow[r,"f\otimes f"]&K\otimes K\arrow[d,"\id_K\otimes\HDelta_K"]&\\
&H\otimes H\otimes H\arrow[r,"f\otimes f\otimes f"']\arrow[ld,"m_H\otimes S_H"] &K\otimes K\otimes K\arrow[rd,"m_K\otimes S_K"']&\\
H\otimes H\arrow[rrr,"f\otimes f"] &&& K\otimes K
\end{tikzcd}
\end{center}

\vspace{.2in}

\noindent Here, the left and right regions commute by definition of $\Delta_H$ and $\Delta_K$. The top region commutes because $f$ is compatible with the integrals of $H$ and $K$. The bottom region commutes because $f$ is an algebra map and  is compatible with the antipodes of $H$ and $K$. Finally, the middle region commutes because $f$ is a coalgebra map between the Hopf algebras $H$ and $K$. 
\qed



\begin{figure}[h!]
\vspace{.5in}
\scalebox{0.7}{

\tikzset{every picture/.style={line width=0.75pt}} 


}
\caption{Proof of Lemma~\ref{lem:appA}(a).}
\label{fig:lem(a)}
\end{figure}



\begin{figure}[h!]
\vspace{0.7in}
\scalebox{0.8}{
\hspace{.15in}
\tikzset{every picture/.style={line width=0.75pt}} 

%
}
\vspace{-.3in}
\caption{Proof of Lemma~\ref{lem:appA}(b).}
\label{fig:lem(b)}
\vspace{0.5in}
\end{figure}



\begin{figure}[h!]
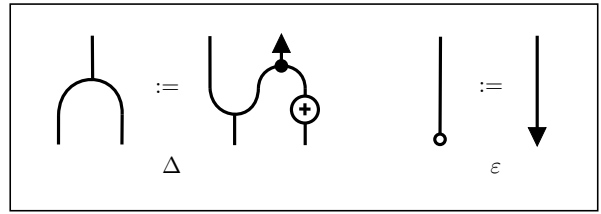

\scalebox{0.8}{

\tikzset{every picture/.style={line width=0.75pt}} 

%
}
\vspace{-.1in}
\caption{Coproduct and counit for the Frobenius-from-Hopf structure in $\cC$.}
\label{fig:DelEp}
\end{figure}

\vspace*{\fill}




\begin{figure}
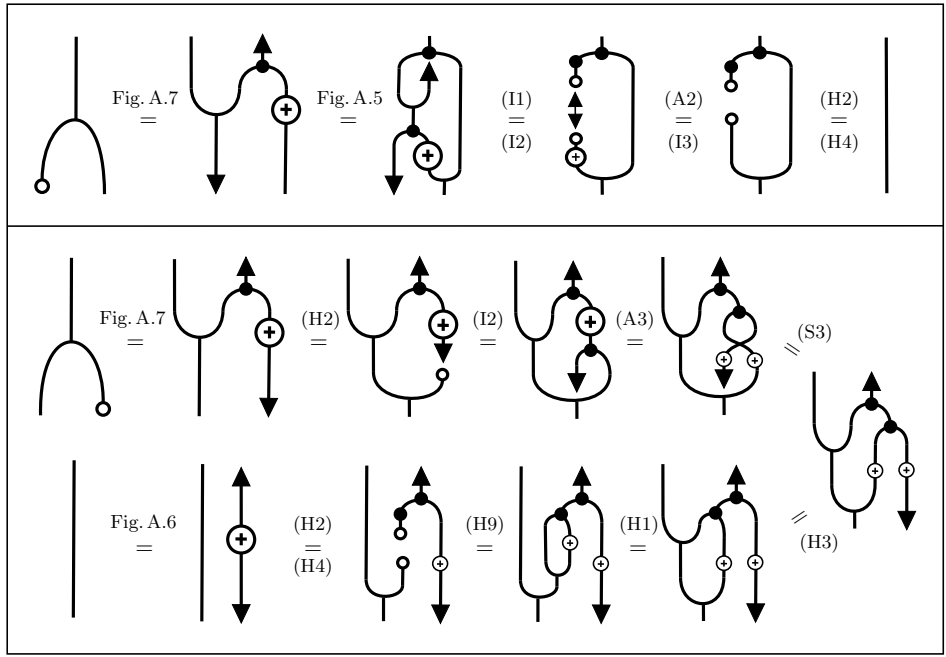

\vspace{0.2in}
\scalebox{0.8}{
\hspace{-.1in}
\tikzset{every picture/.style={line width=0.75pt}} 

%
}
\vspace{-.1in}
\caption{Proof of counitality for the Frobenius-from-Hopf structure in $\cC$.}
\label{fig:counit}
\vspace{0.3in}
\end{figure}


\begin{figure}[h!]
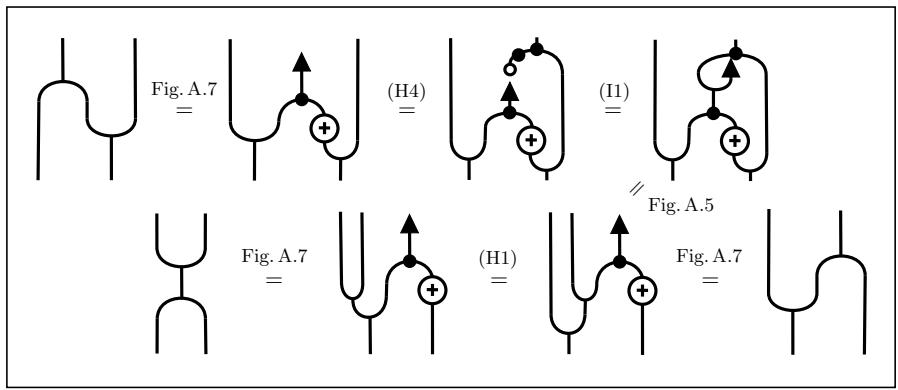

\scalebox{0.8}{
\vspace{1in}
\tikzset{every picture/.style={line width=0.75pt}} 

%
}
\vspace{-.1in}
\caption{Proof of the Frobenius laws for the Frobenius-from-Hopf structure in $\cC$.}
\label{fig:Froblaw}
\vspace{0.3in}
\end{figure}


\bigskip


\begin{figure}[h!]
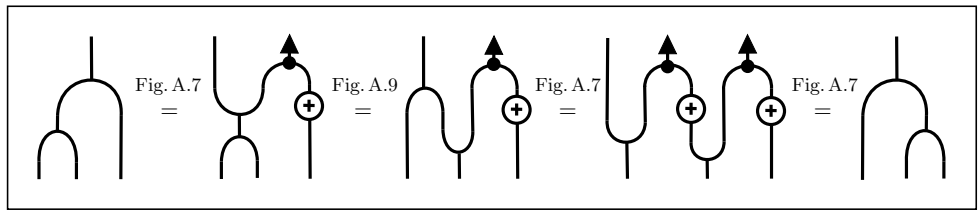

\scalebox{0.8}{

\tikzset{every picture/.style={line width=0.75pt}} 

%
}
\vspace{-.1in}
\caption{Proof of coassociativity for the Frobenius-from-Hopf structure in $\cC$.}
\label{fig:coassoc}
\end{figure}

\clearpage




\section{Proofs of selected results in Section~\ref{sec:eFrobMF} (preprint version only)}
\label{app:b}

We prove Proposition \ref{prop:eFMpres} in Section~\ref{app:b1}, and prove Proposition~\ref{thm:eFMcomp} in Section~\ref{app:b2}.


\subsection{Proof of Proposition \ref{prop:eFMpres}} \label{app:b1}

Given $(A,m_A,u_A,\Delta_A,\varep_A,\phi_A,\theta_A)\in\mathsf{ExtFrobAlg}(\cC)$, we first define an extended Frobenius algebra structure on $F(A)$. 
Let $m_{F(A)}$, $u_{F(A)}$, $\Delta_{F(A)}$, and $\varepsilon_{F(A)}$ be as in Proposition \ref{prop:pres}(a,b). By Proposition~\ref{prop:pres}(c), this makes $F(A)$ a Frobenius algebra in $\cC'$. Define
\[
\phi_{F(A)}:=F(\phi_A)\hs\widehat{F}_A, \qquad \theta_{F(A)}:=F(\theta_A)\hs\widecheck{F},
\]
 and note that by naturality, $\phi_{F(A)}:= F(\phi_A)\hs\widehat{F}_A=\widehat F_A\hs F(\phi_A)$. We will now show that $\phi_{F(A)}$ and $\theta_{F(A)}$ satisfy the conditions in Definition \ref{def:extFrobC}(b).

To verify Definition~\ref{def:extFrobC}(b)(i) for $F(A)$, we first show that $\phi_{F(A)}$ is a Frobenius algebra morphism. Commutativity of Diagram~\ref{diag:phi is frob} verifies $m_{F(A)}(\phi_{F(A)}\otimes'\phi_{F(A)})=\phi_{F(A)}\hs m_{F(A)}$. Regions (1), (2), (5), and~(8) commute by definition, (3) by monoidality of $\widehat F$, (4) and (6) by naturality, and (7) by multiplicativity of $\phi_A$. 
Likewise, comonoidality of $\widehat{F}$ gives $(\phi_{F(A)}\otimes'\phi_{F(A)})\Delta_{F(A)}=\Delta_{F(A)}\hs\phi_{F(A)}.$ 

Commutativity of Diagram~\ref{diag:phi is frob 2} shows that $u_{F(A)}=\phi_{F(A)}\hs u_{F(A)}$. Regions (1), (4), and (6) commute by definition, (2) by monoidality of $\widehat F$,  (3) by $\phi_{A}$ being an algebra morphism, and (5) by naturality. Using that $\widehat{F}$ is comonoidal, an analogous argument shows that $\varepsilon_{F(A)}=\varepsilon_{F(A)}\hs\phi_{F(A)}$, concluding the proof that $\phi_{F(A)}$ is a morphism of Frobenius algebras in $\cC'$.


\setcounter{figure}{0}
\renewcommand{\figurename}{Diagram}
\renewcommand{\thefigure}{B.\arabic{figure}}

\begin{figure}[h!]
\hspace{-1.5cm} \adjustbox{scale=0.95}{ 
  \begin{tikzcd}[every label/.append style = {font = \tiny}, nodes={font=\footnotesize}, column sep=7ex, row sep=2ex]
	{F(A)\otimes'F(A)} && {F(A\otimes A)} && {F(A)} \\
	\\
	{F(A)\otimes'F(A)} && {F(A\otimes A)} && {F(A)} \\
	\\
	{F(A)\otimes'F(A)} && {F(A\otimes A)} && {F(A)}
	\arrow[""{name=0, anchor=center, inner sep=0}, "{{{F^{(2)}_{A,A}}}}"'{description}, from=1-1, to=1-3]
	\arrow[""{name=1, anchor=center, inner sep=0}, "{{{m_{F(A)}}}}", curve={height=-30pt}, from=1-1, to=1-5]
	\arrow[""{name=2, anchor=center, inner sep=0}, "{{{\widehat{F}_A \otimes' \widehat{F}_A}}}", from=1-1, to=3-1]
	\arrow[""{name=3, anchor=center, inner sep=0}, "{{{\phi_{F(A)}\otimes'\phi_{F(A)}}}}"',shift right=3, curve={height=30pt}, from=1-1, to=5-1]
	\arrow[""{name=4, anchor=center, inner sep=0}, "{{{F(m_A)}}}"{description}, from=1-3, to=1-5]
	\arrow["{{{\widehat{F}_{A\otimes A}}}}", from=1-3, to=3-3]
	\arrow[""{name=5, anchor=center, inner sep=0}, "{{{\widehat{F}_A}}}"', from=1-5, to=3-5]
	\arrow[""{name=6, anchor=center, inner sep=0}, "{{{\phi_{F(A)}}}}", shift right=0,curve={height=-30pt}, from=1-5, to=5-5]
	\arrow[""{name=7, anchor=center, inner sep=0}, "{{{F^{(2)}_{A,A}}}}"{description}, from=3-1, to=3-3]
	\arrow["{{{F(\phi_A) \otimes' F(\phi_A)}}}", from=3-1, to=5-1]
	\arrow[""{name=8, anchor=center, inner sep=0}, "{{{F(m_A)}}}"{description}, from=3-3, to=3-5]
	\arrow["{{{F(\phi_A \otimes \phi_A)}}}", from=3-3, to=5-3]
	\arrow["{{{F(\phi_A)}}}"', from=3-5, to=5-5]
	\arrow[""{name=9, anchor=center, inner sep=0}, "{{{F^{(2)}_{A,A}}}}"{description}, from=5-1, to=5-3]
	\arrow[""{name=10, anchor=center, inner sep=0}, "{{{m_{F(A)}}}}"', curve={height=30pt}, from=5-1, to=5-5]
	\arrow[""{name=11, anchor=center, inner sep=0}, "{{{F(m_A)}}}"{description}, from=5-3, to=5-5]
	\arrow["{{{\textcolor{blue}{(1)}}}}"{description,pos=.3}, shift left=0,draw=none, from=1-3, to=1]
	\arrow["{{{\textcolor{blue}{(2)}}}}"{description, pos=0.7}, shift left, draw=none, from=3, to=2]
	\arrow["{{{\textcolor{blue}{(3)}}}}"{description}, shift left=-4, curve={height=-12pt}, draw=none, from=0, to=7]
	\arrow["{{{\textcolor{blue}{(4)}}}}"{description}, draw=none, from=4, to=8]
	\arrow["{{{\textcolor{blue}{(5)}}}}"{description,pos=.4}, draw=none, from=3-5, to=6]
	\arrow["{{{\textcolor{blue}{(6)}}}}", shift left=5, draw=none, from=7, to=9]
	\arrow["{{{\textcolor{blue}{(7)}}}}"{description}, shift left=2, draw=none, from=8, to=11]
	\arrow["{{{\textcolor{blue}{(8)}}}}"{description,pos=.3}, shift right=0, draw=none, from=5-3, to=10]
\end{tikzcd}}
\vspace{-.1in}
\caption{$\phi_{F(A)}$ is multiplicative.}
\label{diag:phi is frob}
\end{figure}
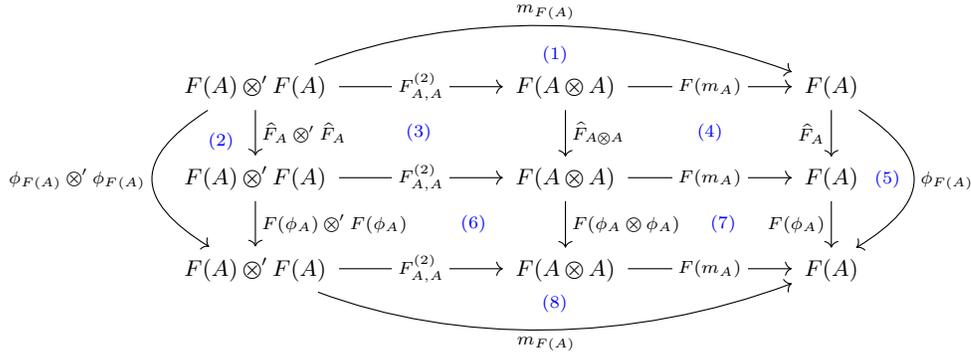

\vspace{-.1in}

\begin{figure}[h!]
\adjustbox{scale=0.95}{ 
\begin{tikzcd}[every label/.append style = {font = \tiny}, nodes={font=\footnotesize},column sep=6ex]
	{\one'} && {F(\one)} && {F(A)} \\
	&& {F(\one)} && {F(A)} \\
	&&&& {F(A)}
	\arrow["{{{F^{(0)}}}}"{description}, from=1-1, to=1-3]
	\arrow[""{name=0, anchor=center, inner sep=0}, "{{{u_{F(A)}}}}", shift left, curve={height=-18pt}, from=1-1, to=1-5]
	\arrow[""{name=1, anchor=center, inner sep=0}, "{{{F^{(0)}}}}"{description}, from=1-1, to=2-3]
	\arrow[""{name=2, anchor=center, inner sep=0}, "{{{u_{F(A)}}}}"', shift right, curve={height=24pt}, from=1-1, to=3-5]
	\arrow["{{{F(u_A)}}}"{description}, from=1-3, to=1-5]
	\arrow["{{{\widehat{F}_{\one}}}}", from=1-3, to=2-3]
	\arrow[""{name=3, anchor=center, inner sep=0}, "{{{F(u_A)}}}"{description, pos=.65},shift right=.5, from=1-3, to=2-5]
	\arrow[""{name=4, anchor=center, inner sep=0}, "{{{F(\phi_A)}}}"'{pos=.3}, from=1-5, to=2-5]
	\arrow[""{name=5, anchor=center, inner sep=0}, "{{{\phi_{F(A)}}}}", shift left, curve={height=-30pt}, from=1-5, to=3-5]
	\arrow["{{{\textcolor{blue}{(5)}}}}"{description}, shift right=.5, draw=none, from=2-3, to=2-5]
	\arrow["{{{F(u_A)}}}"{description}, from=2-3, to=3-5]
	\arrow["{{{\widehat{F}_A}}}"'{pos=.3}, from=2-5, to=3-5]
	\arrow["{{{\textcolor{blue}{(1)}}}}"{description, pos=0.7}, draw=none, from=0, to=1-3]
	\arrow["{{\textcolor{blue}{(2)}}}"{description,pos=.7}, shift left = -1, draw=none, from=1, to=1-3]
	\arrow["{{{\textcolor{blue}{(3)}}}}"{description,pos=0.2},shift right=-2.5, draw=none, from=3, to=1-5]
	\arrow["{{{\textcolor{blue}{(6)}}}}"{description,pos=.4}, draw=none, from=2-5, to=5]
	\arrow["{{{\textcolor{blue}{(4)}}}}"{description, pos=0.3}, draw=none, from=2-3, to=2]
\end{tikzcd}}
\vspace{-.1in}
\caption{$\phi_{F(A)}$ is unital.}
\label{diag:phi is frob 2}
\end{figure}

\vspace{.1in}

Lastly, to see that $\phi_{F(A)}$ is an involution, note that 
\begin{equation*}
\phi_{F(A)}\circ\phi_{F(A)}=F(\phi_A)\circ\widehat{F}_A\circ \widehat{F}_A\circ F(\phi_A)
=F(\phi_A\circ\phi_A)
=\id_{F(A)},
\end{equation*}
where we use $\phi_{F(A)}:= F(\phi_A)\hs\widehat{F}_A=\widehat F_A\hs F(\phi_A)$, Definition~\ref{def:eFMfunc}(c)(i), and  $\phi_A$ being an involution. 

\pagebreak

Next, Definition~\ref{def:extFrobC}(b)(ii) for $F(A)$ follows from commutativity of Diagram \ref{diag:FA 3.1(b)(ii)} below.

\begin{figure}[h!]
\adjustbox{scale=.88}{\begin{tikzcd}[every label/.append style = {font = \tiny}, nodes={font=\footnotesize}, scale cd = 1.05, column sep=4ex, row sep=4ex]
	{\one'\otimes'F(A)} &&& {F(\one)\otimes'F(A)} && {F(A)\otimes'F(A)} \\
& {F(\one)\otimes'F(A)} && {F(\one\otimes A)} \\
&& {F(\one\otimes A)} && {F(A\otimes A)} \\
& {F(A)\otimes'F(A)} & {F(A\otimes A)} & {F(A\otimes A)} & {F(A)} & {F(A)} \\
&&&& {F(A)}
\arrow["{{{\widecheck{F}\otimes' \id_{F(A)}}}}"'{pos=.6}, from=1-1, to=1-4]
\arrow[""{name=0, anchor=center, inner sep=0}, "{{{\theta_{F(A)}\,\otimes'\,\id_{F(A)}}}}", curve={height=-30pt}, from=1-1, to=1-6]
\arrow[from=1-1, to=2-2]
\arrow["{{{\widecheck{F}\,\otimes'\,\id_{F(A)}}}}"{pos=.5}, shift right=1.5, draw=none, from=1-1, to=2-2]
\arrow[""{name=1, anchor=center, inner sep=0}, "{{{\theta_{F(A)}\otimes'\id_{F(A)}}}}"', curve={height=30pt}, from=1-1, to=4-2]
\arrow["{{{F(\theta_A) \otimes' \id_{F(A)}}}}"', from=1-4, to=1-6]
\arrow[""{name=2, anchor=center, inner sep=0}, "{{{F^{(2)}_{\one,A}}}}", from=1-4, to=2-4]
\arrow[""{name=3, anchor=center, inner sep=0}, "{{{F^{(2)}_{A,A}}}}"{description}, from=1-6, to=3-5]
\arrow[""{name=4, anchor=center, inner sep=0}, "{{{m_{F(A)}}}}", from=1-6, to=4-6]
\arrow["{{\textcolor{blue}{(3)}}}"{pos=.45}, draw=none,shift right=.75, from=2-2, to=2-4]
\arrow["{{{F^{(2)}_{\one,A}}}}"{pos=.4},draw=none,shift right=1.5, from=2-2, to=3-3]
\arrow[from=2-2, to=3-3]
\arrow[""{name=5, anchor=center, inner sep=0}, "{{{F(\theta_A)\,\otimes'\,\id_{F(A)}}}}"'{description}, from=2-2, to=4-2]
\arrow["{{{\widehat{F}_{\one\otimes A}}}}"'{pos=0.4}, draw=none, shift right=-1.25, from=2-4, to=3-3]
\arrow[from=2-4, to=3-3]
\arrow["{{{F(\theta_A \otimes \id_{A})}}}" {pos=0.2},draw=none, shift right=1.5, from=2-4, to=3-5]
\arrow[from=2-4, to=3-5]
\arrow[""{name=6, anchor=center, inner sep=0}, "{{{F(\theta_A \otimes \id_{A})}}}"{description}, from=2-4, to=4-4]
\arrow["{{{F(\theta_A \otimes \id_{A})}}}"{pos=0.4}, from=3-3, to=4-3]
\arrow["{{{F(m_A)}}}"{pos=.4}, draw=none, shift right =1.5, from=3-5, to=4-6]
\arrow[from=3-5, to=4-6]
\arrow[""{name=7, anchor=center, inner sep=0}, "{{{F^{(2)}_{A,A}}}}", from=4-2, to=4-3]
\arrow["{{{m_{F(A)}}}}"', curve={height=24pt}, from=4-2, to=5-5]
\arrow[""{name=8, anchor=center, inner sep=0}, "{{{F(m_A)}}}"'{description,pos=0.5}, from=4-3, to=5-5]
\arrow["{{\textcolor{blue}{(8)}}}"{description,pos=.6}, shift right =-2, draw=none, from=4-4, to=3-5]
\arrow["{{{\widehat{F}_{A\otimes A}}}}"', from=4-4, to=4-3]
\arrow[""{name=9, anchor=center, inner sep=0}, "{{{F(m_A)}}}", from=4-4, to=4-5]
\arrow["{{{\widehat{F}_A}}}", from=4-5, to=5-5]
\arrow[""{name=10, anchor=center, inner sep=0}, "{{{F(\phi_A)}}}"'{pos=0.6}, from=4-6, to=4-5]
\arrow[""{name=11, anchor=center, inner sep=0}, "{{{\phi_{F(A)}}}}", curve={height=-8pt}, from=4-6, to=5-5]
\arrow["{{\textcolor{blue}{(1)}}}"{description,pos=.4}, shift right=2, draw=none, from=0, to=1-4]
\arrow["{{\textcolor{blue}{(4)}}}"{description,pos=.6}, draw=none, from=2, to=3]
\arrow["{{\textcolor{blue}{(5)}}}"{description}, draw=none, from=3-5, to=4]
\arrow["{{\textcolor{blue}{(6)}}}"{pos=.7}, shift right=1, draw=none, from=4-2, to=3-3]
\arrow["{{\textcolor{blue}{(2)}}}"{description, pos=0.5}, shift right=5, draw=none, from=2-2, to=1]
\arrow["{{\textcolor{blue}{(7)}}}"{description,pos=.55}, draw=none, shift right=1, from=4-3, to=2-4]
\arrow["{{\textcolor{blue}{(10)}}}"'{pos=0.7}, shift left=-1, draw=none, from=8, to=9]
\arrow["{{\textcolor{blue}{(11)}}}"{description, pos=0.45}, shift right =-2, draw=none, from=11, to=10]
\arrow["{{\textcolor{blue}{(9)}}}"{description,pos=.4}, shift right=2, draw=none, from=4-2, to=5-5]
\end{tikzcd}}
\vspace{-.1in}
\caption{$F(A)$ satisfies Definition~\ref{def:extFrobC}(b)(ii).}
\label{diag:FA 3.1(b)(ii)}
\end{figure}

\noindent Regions (1), (2), (5), (9), and (11) commute by definition, (4), (6), (7), and (10) by naturality,~(3) by Definition \ref{def:eFMfunc}(c)(ii), and (8) by Definition \ref{def:extFrobC}(b)(ii) for $A$.

Lastly, Definition~\ref{def:extFrobC}(b)(iii) for $F(A)$ holds by commutativity of Diagram \ref{F(A) is ext fr},
where regions (1), (2), (3), (8), (20), and (21) commute by definition, (5), (6), and (9)-(18) by naturality, (4) by Definition \ref{def:eFMfunc}(b), (7) by Definition \ref{def:eFMfunc}(c)(iii), and (19) by Definition \ref{def:extFrobC}(b)(iii) for $A$. This completes the proof that $F(A)\in\mathsf{ExtFrobAlg}(\cC')$.

\begin{figure}[h!]
	\adjustbox{scale=.86,center}{
\begin{tikzcd}[every label/.append style = {font = \tiny}, nodes={font=\footnotesize}, column sep=0.8ex, row sep=5ex]
& {\one'} & {F(\one)} & {F(A)} && {F(A\otimes A)} & {F(A)\otimes' F(A)} \\
{\one'\otimes'\one'} & {F(\one\otimes\one)} && {F(A\otimes\one)} && {F(A\otimes A\otimes\one)} \\
& {F(\one)\otimes'F(\one)} && {F(A)\otimes'F(\one)} && {F(A\otimes A)\otimes'F(\one)} & {F(A)\otimes'F(A)} \\
& {F(\one)\otimes'F(\one)} && {F(A)\otimes'F(\one)} && {F(A\otimes A)\otimes'F(\one)} \\
{F(\one)\otimes'F(\one)} & {F(\one\otimes\one)} && {F(A\otimes\one)} && {F(A\otimes A\otimes\one)} & {F(A)\otimes'F(A)} \\
&&& {F(\one)} & {F(A)} & {F(A\otimes A)} & {F(A\otimes A)} \\
{F(A)\otimes'F(A)} && {F(A\otimes A)} &&&& {F(A)}
\arrow["{{{{{{F^{(0)}}}}}}}"'{pos=0.5}, from=1-2, to=1-3]
\arrow[""{name=0, anchor=center, inner sep=0}, "{{{{{{u_{F(A)}}}}}}}", curve={height=-25pt}, from=1-2, to=1-4]
\arrow["{{{{{{r_{\one}^{-1}}}}}}}"'{pos=.4},shift right =-2, draw=none,from=1-2, to=2-1]
\arrow[from=1-2, to=2-1]
\arrow["{{{{{{F(u_A)}}}}}}"', from=1-3, to=1-4]
\arrow["{{{{{{F(r_{\one}^{-1})}}}}}}"'{pos=.6},shift right=-1.5, draw=none,from=1-3, to=2-2]
\arrow[from=1-3, to=2-2]
\arrow["{{{{{\textcolor{blue}{(5)}}}}}}"{description, pos=0.2},shift right =3, draw=none, from=1-3, to=2-4]
\arrow["{{{{{{F(\Delta_A)}}}}}}"'{pos=0.5}, from=1-4, to=1-6]
\arrow[""{name=1, anchor=center, inner sep=0}, "{{{{{{\Delta_{F(A)}}}}}}}", curve={height=-25pt}, from=1-4, to=1-7]
\arrow["{{{{{{F(r_A^{-1})}}}}}}"{pos=.4}, from=1-4, to=2-4]
\arrow["{{{{{\textcolor{blue}{(6)}}}}}}"{description}, draw=none, from=1-4, to=2-6]
\arrow["{{{{{{F_{(2)}^{A,A}}}}}}}"'{pos=0.5},shift right =1,draw=none, from=1-6, to=1-7]
\arrow[ from=1-6, to=1-7]
\arrow["{{{{{{F(r_{A\otimes A}^{-1})}}}}}}"'{pos=0.5}, from=1-6, to=2-6]
\arrow[""{name=2, anchor=center, inner sep=0}, "{{{{{{\widehat{F}_{A}\otimes'\id_{F(\one)}}}}}}}"{description}, from=1-7, to=3-7]
\arrow[""{name=3, anchor=center, inner sep=0}, "{{{{{{\phi_{F(A)}\otimes'\id_{F(A)}}}}}}}"{description,pos=.4},shift right=-6, curve={height=-60pt}, from=1-7, to=5-7]
\arrow["{{{{{(\textcolor{blue}{4)}}}}}}"{description,pos=.4},shift right =2, draw=none, from=2-1, to=3-2]
\arrow[""{name=4, anchor=center, inner sep=0}, "{{{{{{\widecheck{F}\otimes'\widecheck{F}}}}}}}"', from=2-1, to=5-1]
\arrow[""{name=5, anchor=center, inner sep=0}, "{{{{{{\theta_{F(A)}\otimes'\theta_{F(A)}}}}}}}"{description, pos=0.4},shift right=3, curve={height=50pt}, from=2-1, to=7-1]
\arrow["{{{{{{F(u_A\otimes\id_\one)}}}}}}", from=2-2, to=2-4]
\arrow["{{{{{{F^{\one,\one}_{(2)}}}}}}}"', from=2-2, to=3-2]
\arrow["{{{{{\textcolor{blue}{(9)}}}}}}"{description},shift right=6, draw=none, from=2-2, to=2-4]
\arrow["{{{{{{F(\Delta_A\otimes\id_{\one})}}}}}}", from=2-4, to=2-6]
\arrow["{{{{{{F^{A,\one}_{(2)}}}}}}}"', from=2-4, to=3-4]
\arrow["{{{{{\textcolor{blue}{(10)}}}}}}"{description},shift right =6, draw=none, from=2-4, to=2-6]
\arrow[""{name=6, anchor=center, inner sep=0}, "{{{{{{F^{A\otimes A,\one}_{(2)}}}}}}}"'{pos=.4}, from=2-6, to=3-6]
\arrow["{{{{{{F(u_A)\otimes'F(\id_\one)}}}}}}",shift right =-1, draw=none, from=3-2, to=3-4]
\arrow[ from=3-2, to=3-4]
\arrow["{{{{{{\widehat{F}_{\one}\otimes'\id_{F(\one)}}}}}}}"', from=3-2, to=4-2]
\arrow["{{{{{\textcolor{blue}{(11)}}}}}}"{description,pos=.2},shift right =6, draw=none, from=3-2, to=3-4]
\arrow["{{{{{{F(\Delta_A)\otimes' F(\id_\one)}}}}}}",shift right=-1,draw=none, from=3-4, to=3-6]
\arrow[ from=3-4, to=3-6]
\arrow["{{{{{{\widehat{F}_A\otimes'\id_{F(\one)}}}}}}}"'{pos=.4}, from=3-4, to=4-4]
\arrow["{{{{{\textcolor{blue}{(12)}}}}}}"{description,pos=.02},shift right =6, draw=none, from=3-4, to=3-6]
\arrow["{{{{{{\widehat{F}_{A\otimes A}\otimes'\id_{F(\one)}}}}}}}"'{pos=.4}, from=3-6, to=4-6]
\arrow["{{{{{{F(\phi_A)\otimes'F(\id_A)}}}}}}"{pos=.2},draw=none,shift right=.75, from=3-7, to=5-7]
\arrow[""{name=7, anchor=center, inner sep=0},from=3-7, to=5-7]
\arrow["{{{{{{F^{(2)}_{A,A}}}}}}}"{description, pos=0.3}, curve={height=-15pt}, from=3-7, to=6-6]
\arrow["{{{{{{F(u_A)\otimes'F(\id_\one)}}}}}}",shift right =-1, draw=none, from=4-2, to=4-4]
\arrow[from=4-2, to=4-4]
\arrow["{{{{{{F_{\one,\one}^{(2)}}}}}}}"'{pos=.4}, from=4-2, to=5-2]
\arrow["{{{{{\textcolor{blue}{(13)}}}}}}"{description},shift right =6, draw=none, from=4-2, to=4-4]
\arrow["{{{{{{F(\Delta_A)\otimes' F(\id_\one)}}}}}}",shift right=-1,draw=none, from=4-4, to=4-6]
\arrow[from=4-4, to=4-6]
\arrow["{{{{{{F_{A,\one}^{(2)}}}}}}}"'{pos=.4}, from=4-4, to=5-4]
\arrow["{{{{{\textcolor{blue}{(14)}}}}}}"{description},shift right =6, draw=none, from=4-4, to=4-6]
\arrow["{{{{{{F_{A\otimes A,\one}^{(2)}}}}}}}"'{pos=.4}, from=4-6, to=5-6]
\arrow["{{{{{{F_{\one,\one}^{(2)}}}}}}}"', shift right =1,draw=none, from=5-1, to=5-2]
\arrow[from=5-1, to=5-2]
\arrow["{{{{{{F(\theta_A)\otimes' F(\theta_A)}}}}}}"{description,pos=.35}, from=5-1, to=7-1]
\arrow["{{{{{\textcolor{blue}{(15)}}}}}}"{description,pos=.25},shift right =-12, draw=none, from=7-1, to=7-3]
\arrow["{{{{{{F(u_A\otimes\id_\one)}}}}}}", from=5-2, to=5-4]
\arrow[ "{{{{{{F(r_\one)}}}}}}"{pos=.4}, shift right =1,draw=none, from=5-2, to=6-4]
\arrow[""{name=8, anchor=center, inner sep=0}, from=5-2, to=6-4]
\arrow["{{{{{{F(\theta_A\otimes\theta_A)}}}}}}"{description}, from=5-2, to=7-3]
\arrow["{{{{{{F(\Delta_A\otimes\id_{\one})}}}}}}", from=5-4, to=5-6]
\arrow["{{{{{{F(r_A)}}}}}}"'{pos=.4},shift right =-1, draw=none, from=5-4, to=6-5]
\arrow[from=5-4, to=6-5]
\arrow["{{{{{\textcolor{blue}{(17)}}}}}}"{description,pos=0.4},shift right=-18, draw=none, from=5-4, to=6-4]
\arrow["{{{{{{F(r_{A\otimes A})}}}}}}"', from=5-6, to=6-6]
\arrow[""{name=9, anchor=center, inner sep=0}, "{{{{{{F^{(2)}_{A,A}}}}}}}"', from=5-7, to=6-7]
\arrow[""{name=10, anchor=center, inner sep=0}, "{{{{{{m_{F(A)}}}}}}}"{pos=0.5}, curve={height=-30pt},shift right =-2, from=5-7, to=7-7]
\arrow["{{{{{{F(u_A)}}}}}}"',shift right =1,draw=none, from=6-4, to=6-5]
\arrow[from=6-4, to=6-5]
\arrow["{{{{{{F(\Delta_A)}}}}}}"',shift right =1, draw=none, from=6-5, to=6-6]
\arrow[from=6-5, to=6-6]
\arrow["{{{{{\textcolor{blue}{(18)}}}}}}"{pos=0.6}, shift right=5, draw=none, from=6-6, to=5-7]
\arrow["{{{{{{F(\phi_A\otimes\id_A)}}}}}}"',shift right =1, draw=none, from=6-6, to=6-7]
\arrow[""{name=11, anchor=center, inner sep=0}, from=6-6, to=6-7]
\arrow["{{{{{{F(m_A)}}}}}}"'{pos=.6}, from=6-7, to=7-7]
\arrow["{{{{{{F_{A,A}^{(2)}}}}}}}", from=7-1, to=7-3]
\arrow[""{name=12, anchor=center, inner sep=0}, "{{{{{{m_{F(A)}}}}}}}"'{pos=.7}, curve={height=35pt}, from=7-1, to=7-7]
\arrow["{{{{{{F(m_A)}}}}}}", from=7-3, to=7-7]
\arrow["{{{{{\textcolor{blue}{(1)}}}}}}"{description,pos=.5}, shift right=-4, draw=none, from=1-2, to=1-4]
\arrow["{{{{{\textcolor{blue}{(2)}}}}}}"{description,pos=.5},shift right =-4, draw=none, from=1-4, to=1-7]
\arrow["{{{{{\textcolor{blue}{(3)}}}}}}"'{pos=0.5},shift right=6, draw=none, from=2-1, to=7-1]
\arrow["{{{{{\textcolor{blue}{(7)}}}}}}"'{pos=.45},shift right =-16, draw=none, from=1-6, to=4-6]
\arrow["{{{{{\textcolor{blue}{(8)}}}}}}"{pos=0.7},shift right =-9, draw=none, from=1-7, to=3-7]
\arrow["{{{{{\textcolor{blue}{(16)}}}}}}"{description,pos=0.4},shift right =6, draw=none, from=5-4, to=6-4]
\arrow["{{{{{\textcolor{blue}{(20)}}}}}}"{description, pos=0.4}, shift left, draw=none, from=9, to=10]
\arrow["{{{{{\textcolor{blue}{(19)}}}}}}"{description}, draw=none, from=7-3, to=6-4]
\arrow["{{{{{\textcolor{blue}{(21)}}}}}}"{description},shift right=6, draw=none, from=7-1, to=7-7]
\end{tikzcd}}
\hspace{-.1in}
\caption{$F(A)$ satisfies Definition~\ref{def:extFrobC}(b)(iii).}
\label{F(A) is ext fr}
\end{figure}

\pagebreak

It remains to show that if $f:(A,m_A,u_A,\Delta_A,\varep_A,\phi_A,\theta_A)\to (B,m_B,u_B,\Delta_B,\varep_B,\phi_B,\theta_B)$ is a morphism of extended Frobenius algebras in $\cC$, then $F(f):F(A)\to F(B)$ is a morphism of extended Frobenius algebras in $\cC'$. 
 By Proposition~\ref{prop:pres}(c), $F(f)$ is a morphism of Frobenius algebras in $\cC'$, so it is enough to verify that $F(f)\hs \phi_{F(A)}=\phi_{F(B)}\hs F(f)$  and $F(f)\hs \theta_{F(A)}=\theta_{F(B)}$ in $\cC'$.
The first equation follows from Diagram \ref{F(f) is ext 1}, where regions (1) and (4) commute by the definitions of $\phi_{F(A)}$ and $\phi_{F(B)}$, respectively, (2) by naturality of $\widehat{F}$, and (3) because $f$ is a morphism of extended Frobenius algebras in $\cC$. For the second equation, observe that regions (1) and (3) in Diagram \ref{F(f) is ext 2} commute by the definitions of $\theta_{F(A)}$ and $\theta_{F(B)}$, respectively, and region (2) commutes because $f$ is a morphism of extended Frobenius algebras in $\cC$.

\begin{figure}[h!]
\begin{minipage}[b][2in][b]{.49\textwidth}
\adjustbox{scale=.9,center}{
\begin{tikzcd}[every label/.append style = {font = \tiny}, nodes={font=\footnotesize}, scale cd = 1.1,  column sep=7ex, row sep=3ex]
	{F(A)} && {F(A)} && {F(A)} \\
	\\
	{F(B)} && {F(B)} && {F(B)}
	\arrow["{\widehat{F}_A}", from=1-1, to=1-3]
	\arrow[""{name=0, anchor=center, inner sep=0}, "{\phi_{F(A)}}", curve={height=-40pt}, from=1-1, to=1-5]
	\arrow["{F(f)}", from=1-1, to=3-1]
	\arrow["{\textcolor{blue}{(2)}}"{description}, draw=none, from=1-1, to=3-3]
	\arrow["{F(\phi_A)}", from=1-3, to=1-5]
	\arrow["{F(f)}", from=1-3, to=3-3]
	\arrow["{F(f)}", from=1-5, to=3-5]
	\arrow["{\widehat{F}_B}"', from=3-1, to=3-3]
	\arrow[""{name=1, anchor=center, inner sep=0}, "{\phi_{F(B)}}"', curve={height=40pt}, from=3-1, to=3-5]
	\arrow["{\textcolor{blue}{(3)}}"{description}, draw=none, from=3-3, to=1-5]
	\arrow["{F(\phi_B)}"', from=3-3, to=3-5]
	\arrow["{\textcolor{blue}{(1)}}"{description, pos=0.6}, draw=none, from=0, to=1-3]
	\arrow["{\textcolor{blue}{(4)}}"{description, pos=0.4}, draw=none, from=3-3, to=1]
\end{tikzcd}}
\vspace{-.1in}
\caption{$F(f)$ respects $\phi$.}
\label{F(f) is ext 1}
\end{minipage}
\begin{minipage}[b][2in][b]{.49\textwidth}
\adjustbox{scale=.9,center}{
\begin{tikzcd}[every label/.append style = {font = \tiny}, nodes={font=\footnotesize}, scale cd = 1.1, column sep=7ex, row sep=3ex]
	&&&& {F(A)} \\
	{\one'} && {F(\one)} \\
	&&&& {F(B)}
	\arrow[""{name=0, anchor=center, inner sep=0}, "{F(f)}", from=1-5, to=3-5]
	\arrow[""{name=1, anchor=center, inner sep=0}, "{\theta_{F(A)}}", curve={height=-15pt}, from=2-1, to=1-5]
	\arrow["{\widecheck{F}}"{pos=.6}, from=2-1, to=2-3]
	\arrow[""{name=2, anchor=center, inner sep=0}, "{\theta_{F(B)}}"', curve={height=15pt}, from=2-1, to=3-5]
	\arrow["{F(\theta_A)}"{description}, from=2-3, to=1-5]
	\arrow["{F(\theta_B)}"{description}, from=2-3, to=3-5]
	\arrow["{\textcolor{blue}{(1)}}"{description}, draw=none, from=2-1, to=1-5]
	\arrow["{\textcolor{blue}{(3)}}"{description}, draw=none, from=2-1, to=3-5]
	\arrow["{\textcolor{blue}{(2)}}"{description, pos=0.7}, draw=none, from=2-3, to=0]
\end{tikzcd}}
\vspace{.57cm}
\caption{$F(f)$ respects $\theta$.}
\label{F(f) is ext 2}
\end{minipage}
\end{figure}

\vspace{.1in}

This completes the proof of Proposition~\ref{prop:eFMpres}.
\qed

\bigskip


\subsection{Proof of Proposition~\ref{thm:eFMcomp}} \label{app:b2}
 Let \[(F,F^{(2)},F^{(0)},F_{(2)},F_{(0)},\widehat{F},\widecheck{F}):(\cC,\otimes,\one)\to(\cC',\otimes',\one');\] 
 \[(G,G^{(2)},G^{(0)},G_{(2)},G_{(0)},\widehat{G},\widecheck{G}):(\cC',\otimes',\one')\to(\cC'',\otimes'',\one'')\]
be two extended Frobenius monoidal functors. To show that the composition 
\[
GF:(\cC,\otimes,\one)\to(\cC'',\otimes'',\one'')
\]
admits the structure of an extended Frobenius monoidal functor, let $(GF)^{(2)}$,  $(GF)^{(0)}$, $(GF)_{(2)}$, and $(GF)_{(0)}$ be as in Proposition \ref{prop:comp-bkg}(a,b). Proposition \ref{prop:comp-bkg}(c)  gives that this makes $GF$ into a Frobenius monoidal functor. Now, define $\widehat{GF}:GF\Rightarrow GF$ by $\widehat{GF}_X:= G(\widehat{F}_X)\circ \widehat{G}_{F(X)}$ for all $X\in\cC$, and define $\widecheck{GF}:=G(\widecheck{F})\circ \widecheck{G}:\one''\to GF(\one)$. We need to show that $\widehat{GF}$ and $\widecheck{GF}$ extend the above Frobenius monoidal structure on $GF$.

Note first that the composition of (co)monoidal natural transformations is again (co)monoidal, so $\widehat{GF}$ is a Frobenius monoidal natural transformation. So, Definition~\ref{def:eFMfunc}(a) holds for $GF$.

That Definition~\ref{def:eFMfunc}(b) is satisfied by $GF$ follows from commutativity of Diagram \ref{GF satisfies b)}: regions (1), (2), (8), (18), (25), and (26) commute by definition, (4)-(6), (9)-(17), and (19)-(23) by naturality,  (3) and (24) by Definition \ref{def:eFMfunc}(b) for $G$ and $F$ respectively, and  (7) by Definition~\ref{def:eFMfunc}(c)(iii) for $G$.

To see that Definition~\ref{def:eFMfunc}(c)(i) holds for $GF$, see Diagram \ref{c)i)}. Regions (1) and (3) commute by definition of $\widehat{GF}$, and regions (2) and (4) commute by Definition \ref{def:eFMfunc}(c)(i) for $F$ and $G$ respectively. 

Next, $GF$ satisfies Definition~\ref{def:eFMfunc}(c)(ii) by Diagram \ref{c)ii)}: regions (1), (4), (7), (8), and (11) commute by definition; (3), (5), (6), and (9) by naturality; and (2) and (10) by Definition \ref{def:eFMfunc}(c)(ii) for $G$ and $F$ respectively. 

Finally, to see that Definition~\ref{def:eFMfunc}(c)(iii) is satisfied by $GF$, consider Diagram~\ref{GF satisfies c)iii)}: regions (1), (2), (5), (6), (25), and (26) commute by definition; (4), (7)-(11), and (14)-(24) by naturality; and (3), (12), and (13) by Definition \ref{def:eFMfunc}(c)(iii) for $F$ and $G$ respectively.

\smallskip

This concludes the proof of Proposition~\ref{thm:eFMcomp}.
\qed
\vspace{.3in}

\begin{figure}[h!]\adjustbox{scale=.8,center}{
    \begin{tikzcd}[every label/.append style = {font = \tiny}, nodes={font=\footnotesize}, scale cd = 1.15, column sep=5ex, row sep=6ex]
	&&& {GF(X)} \\
	{GF(X)} && {GF(X)} && {GF(X)} && {GF(X)}
	\arrow["{{G(\widehat{F}_X)}}"{pos=0.5},draw=none, shift right=1.5, from=1-4, to=2-5]
        \arrow[from=1-4, to=2-5]
	\arrow[""{name=0, anchor=center, inner sep=0}, "{{\widehat{GF}_X}}", curve={height=-12pt}, from=1-4, to=2-7]
	\arrow[""{name=1, anchor=center, inner sep=0}, "{{\widehat{GF}_X}}", curve={height=-12pt}, from=2-1, to=1-4]
	\arrow["{{\widehat{G}_{F(X)}}}", from=2-1, to=2-3]
	\arrow[""{name=2, anchor=center, inner sep=0}, "{{\id_{GF(X)}}}"', curve={height=30pt}, from=2-1, to=2-7]
	\arrow["{{G(\widehat{F}_X)}}"{pos=0.5}, shift right = 1.5, draw=none, from=2-3, to=1-4]
        \arrow[from=2-3, to=1-4]
	\arrow[""{name=3, anchor=center, inner sep=0}, "{{\id_{GF(X)}}}", from=2-3, to=2-5]
	\arrow["{{\widehat{G}_{F(X)}}}", from=2-5, to=2-7]
	\arrow["{\textcolor{blue}{(3)}}"{description}, shift left=3, draw=none, from=0, to=2-5]
	\arrow["{\textcolor{blue}{(2)}}"{description, pos=0.4}, draw=none, from=1-4, to=3]
	\arrow["{\textcolor{blue}{(1)}}"{description}, draw=none, from=2-3, to=1]
	\arrow["{\textcolor{blue}{(4)}}"{description}, draw=none, from=3, to=2]
\end{tikzcd}}
\vspace{-.1in}
\caption{$GF$ satisfies Definition~\ref{def:eFMfunc}(c)(i).}
\label{c)i)}
\end{figure}

\vspace{0.3in}

\begin{figure}[h!]
    \adjustbox{scale=.8,center}{\begin{tikzcd}[every label/.append style = {font = \tiny}, nodes={font=\footnotesize}, scale cd = 1.15, column sep=3ex, row sep=6ex]
	{\one''\otimes''GF(X)} && {G(\one')\otimes''GF(X)} && {GF(\one)\otimes''GF(X)} && {GF(\one\otimes X)} \\
	&& {G(\one'\otimes'F(X))} && {G(F(\one)\otimes'F(X))} \\
	& {G(\one')\otimes''GF(X)} & {G(\one'\otimes'F(X))} && {G(F(\one)\otimes'F(X))} && {GF(\one\otimes X)} \\
	\\
	{GF(\one)\otimes''GF(X)} && {G(F(\one)\otimes'F(X))} &&&& {GF(\one\otimes X)}
	\arrow["{{\widecheck{G}\otimes''GF(\id_{X})}}", from=1-1, to=1-3]
	\arrow[""{name=0, anchor=center, inner sep=0}, "{{\widecheck{GF}\otimes''GF(\id_X)}}", curve={height=-60pt}, from=1-1, to=1-5]
	\arrow[""{name=1, anchor=center, inner sep=0}, "{{\widecheck{G}\otimes''GF(\id_{X})}}", draw=none, shift right = 2, from=1-1, to=3-2]
        \arrow[from=1-1, to=3-2]
	\arrow[""{name=2, anchor=center, inner sep=0}, "{{\widecheck{GF}\otimes''GF(\id_X)}}"', from=1-1, to=5-1]
	\arrow["{{G(\widecheck{F})\otimes''GF(\id_{X})}}"',draw=none, shift right =1, from=1-3, to=1-5]
        \arrow[from=1-3, to=1-5]
	\arrow["{{G^{(2)}_{\one',F(X)}}}"', from=1-3, to=2-3]
	\arrow["{{(GF)^{(2)}_{\one,X}}}", from=1-5, to=1-7]
	\arrow["{{G^{(2)}_{F(\one),F(X)}}}", from=1-5, to=2-5]
	\arrow[""{name=3, anchor=center, inner sep=0}, "{{\widehat{G}_{F(\one\otimes X)}}}"'{pos=.8}, from=1-7, to=3-7]
	\arrow[""{name=4, anchor=center, inner sep=0}, "{{\widehat{GF}_{\one\otimes X}}}", curve={height=-60pt}, from=1-7, to=5-7]
	\arrow["{{G(\widecheck{F}\otimes'F(\id_{X}))}}"',draw=none,shift right =1, from=2-3, to=2-5]
        \arrow[from=2-3, to=2-5]
	\arrow["{{\widehat{G}_{\one'\otimes' F(X)}}}"', from=2-3, to=3-3]
	\arrow[""{name=5, anchor=center, inner sep=0}, "{{G(F^{(2)}_{\one,X})}}"'{pos=.3}, shift right = 0, draw=none, bend right=15, from=2-5, to=1-7]
        \arrow[bend right=20, from=2-5, to=1-7]
        \arrow[""{name=6, anchor=center, inner sep=0}, "{{\widehat{G}_{F(\one)\otimes' F(X)}}}", from=2-5, to=3-5]
	\arrow["{{G^{(2)}_{\one',F(X)}}}"', draw=none, shift right =1,from=3-2, to=3-3]
        \arrow[from=3-2, to=3-3]
	\arrow[""{name=3251, anchor=center, inner sep=0},"{{G(\widecheck{F})\otimes''GF(\id_{X})}}"{pos=0.5},draw =none, shift right=2, from=3-2, to=5-1]
        \arrow[from=3-2, to=5-1]
	\arrow["{{G(\widecheck{F}\otimes'F(\id_{X}))}}"',draw=none, shift right=1, from=3-3, to=3-5]
        \arrow[from=3-3, to=3-5]
	\arrow[""{name=3353, anchor=center, inner sep=0},"{{G(\widecheck{F}\otimes'F(\id_{X}))}}", from=3-3, to=5-3]
	\arrow["{{G(F^{(2)}_{\one,X})}}"',draw=none, shift right =1, from=3-5, to=3-7]
        \arrow[from=3-5, to=3-7]
	\arrow["{{G(\widehat{F}_{\one\otimes X})}}"', from=3-7, to=5-7]
	\arrow["{{G^{(2)}_{F(\one),F(X)}}}"'{description}, from=5-1, to=5-3]
	\arrow[""{name=7, anchor=center, inner sep=0}, "{{(GF)^{(2)}_{\one,X}}}"', curve={height=60pt}, from=5-1, to=5-7]
	\arrow[""{name=8, anchor=center, inner sep=0}, "{{G(F^{(2)}_{\one,X})}}"'{description}, from=5-3, to=5-7]
	\arrow["{\textcolor{blue}{(1)}}"{description}, draw=none, from=0, to=1-3]
	\arrow["{\textcolor{blue}{(2)}}"{pos=.8},shift right =1, draw=none, from=1, to=2-3]
	\arrow["{\textcolor{blue}{(8)}}"{description}, draw=none, from=2, to=3-2]
	\arrow["{\textcolor{blue}{(6)}}"{description},shift right =-3, draw=none, from=2-5, to=3-7]
	\arrow["{\textcolor{blue}{(7)}}"{description}, draw=none, from=3, to=4]
	\arrow["{\textcolor{blue}{(4)}}"{description,pos=.65}, shift right =-1,draw=none, from=2-5, to=1-7]
	\arrow["{\textcolor{blue}{(10)}}"{description}, draw=none, from=3-5, to=8]
	\arrow["{\textcolor{blue}{(11)}}"{description}, draw=none, from=5-3, to=7]
        \arrow["{\textcolor{blue}{(9)}}"'{pos=.6},shift right =-1, draw=none, from=3251, to=3353]
        \arrow["{\textcolor{blue}{(3)}}"{description},shift right =2, draw=none, from=1-3, to=2-5]
        \arrow["{\textcolor{blue}{(5)}}"{description},shift right =2, draw=none, from=2-3, to=3-5]
\end{tikzcd}}
\vspace{-.1in}
\caption{$GF$ satisfies Definition~\ref{def:eFMfunc}(c)(ii).}
\label{c)ii)}
\end{figure}

\vspace{.1in}

\begin{landscape}
 \begin{figure}[h!]
\adjustbox{scale=.9,center}{
\begin{tikzcd}[every label/.append style = {font = \tiny}, nodes={font=\scriptsize}, column sep=5.8ex, row sep=8ex]
	{\one''} & {G(\one')} & {GF(\one)} & {GF(\one\otimes \one)} & {G(F(\one)\otimes' F(\one))} & {GF(\one)\otimes'' GF(\one)} \\
{\one''\otimes'' \one''} & {G(\one'\otimes'\one')} & {G(F(\one)\otimes' \one')} & {G(F(\one\otimes \one)\otimes'\one')} & {G(F(\one)\otimes' F(\one)\otimes'\one')} \\
& {G(\one')\otimes''G(\one')} & {GF(\one)\otimes''G( \one')} & {GF(\one\otimes \one)\otimes''G(\one')} & {G(F(\one)\otimes' F(\one))\otimes''G(\one')} & {GF(\one)\otimes'' GF(\one)} \\
& {G(\one')\otimes''G(\one')} & {GF(\one)\otimes''G( \one')} & {GF(\one\otimes \one)\otimes''G(\one')} & {G(F(\one)\otimes' F(\one))\otimes''G(\one')} \\
{G(\one')\otimes'' G(\one')} & {G(\one'\otimes'\one')} & {G(F(\one)\otimes' \one')} & {G(F(\one\otimes \one)\otimes'\one')} & {G(F(\one)\otimes' F(\one)\otimes'\one')} \\
&& {GF(\one)} & {GF(\one\otimes \one)} & {G(F(\one)\otimes'F( \one))} & {GF(\one)\otimes'' GF(\one)} \\
{GF(\one)\otimes''GF(\one)} & {G(F(\one)\otimes'F(\one))} & {G(\one')} & {G(F(\one)\otimes'F( \one))} \\
&&& {GF(\one\otimes\one)} & {GF(\one)}
\arrow["{{G^{(0)}}}"{pos=0.5}, from=1-1, to=1-2]
\arrow[""{name=0, anchor=center, inner sep=0}, "{(GF)^{(0)}}"{pos=0.5}, curve={height=-30pt}, from=1-1, to=1-3]
\arrow["{{(r''_{\one''})^{-1}}}"', from=1-1, to=2-1]
\arrow["{G(F^{(0)})}"{pos=0.5}, from=1-2, to=1-3]
\arrow[""{name=1, anchor=center, inner sep=0}, "{G{((r'_{\one'})^{-1})}}"', from=1-2, to=2-2]
\arrow["{GF(r_{\one}^{-1})}", from=1-3, to=1-4]
\arrow[""{name=2, anchor=center, inner sep=0}, "{G((r'_{F(\one)})^{-1})}"'{pos=0.4}, from=1-3, to=2-3]
\arrow["{G(F_{(2)}^{\one,\one})}"{pos=0.6}, from=1-4, to=1-5]
\arrow[""{name=3, anchor=center, inner sep=0}, "{(GF)_{(2)}^{\one,\one}}"{pos=0.5}, curve={height=-30pt}, from=1-4, to=1-6]
\arrow[""{name=4, anchor=center, inner sep=0}, "{G((r'_{F(\one\otimes \one)})^{-1})}"'{pos=0.4}, from=1-4, to=2-4]
\arrow["{G_{(2)}^{F(\one),F(\one)}}"'{pos=0.6}, from=1-5, to=1-6]
\arrow[""{name=5, anchor=center, inner sep=0}, "{G((r'_{F(\one)\otimes' F(\one)})^{-1})}"'{pos=0.4}, from=1-5, to=2-5]
\arrow[""{name=6, anchor=center, inner sep=0}, "{\widehat G_{F(\one)}\otimes'' \id_{GF(\one)}}"{description,pos=0.6}, from=1-6, to=3-6]
\arrow[""{name=7, anchor=center, inner sep=0}, "{\widehat{GF_{\one}}\otimes'' \id_{GF(\one)}}"{description, pos=0.6}, shift left=4, curve={height=-50pt}, from=1-6, to=6-6]
\arrow[""{name=8, anchor=center, inner sep=0}, "{{\widecheck G\otimes'' \widecheck G}}", from=2-1, to=5-1]
\arrow[""{name=9, anchor=center, inner sep=0}, "{{\widecheck{GF}\otimes''\widecheck{GF}}}"{description, pos=0.4}, shift right=4, curve={height=40pt}, from=2-1, to=7-1]
\arrow["{G(F^{(0)}\otimes ' \id_{\one '})}"{pos=0.5},shift right =-.5,draw=none,from=2-2, to=2-3]
\arrow[from=2-2, to=2-3]
\arrow[""{name=10, anchor=center, inner sep=0}, "{G_{(2)}^{\one',\one'}}"', from=2-2, to=3-2]
\arrow["{G(F(r_{\one}^{-1})\otimes'\id_{\one'})}",draw=none, shift left, from=2-3, to=2-4]
\arrow[from=2-3, to=2-4]
\arrow[""{name=11, anchor=center, inner sep=0}, "{G_{(2)}^{F(\one),\one'}}"'{pos=0.4}, from=2-3, to=3-3]
\arrow["{G(F_{(2)}^{\one,\one}\otimes' \id_{\one'})}",shift left, draw=none, from=2-4, to=2-5]
\arrow[from=2-4, to=2-5]
\arrow[""{name=12, anchor=center, inner sep=0}, "{G_{(2)}^{F(\one\otimes \one),\one'}}"'{pos=0.4}, from=2-4, to=3-4]
\arrow[""{name=13, anchor=center, inner sep=0}, "{G_{(2)}^{F(\one)\otimes' F(\one),\one'}}"'{pos=0.4}, from=2-5, to=3-5]
\arrow["{G(F^{(0)})\otimes '' G(\id_{\one '})}",draw=none, shift left, from=3-2, to=3-3]
\arrow[from=3-2, to=3-3]
\arrow[""{name=14, anchor=center, inner sep=0}, "{\widehat{G}_{\one '}\otimes'' G(\id_{\one'})}"', from=3-2, to=4-2]
\arrow["{GF(r_{\one}^{-1})\otimes''G(\id_{\one'})}",draw=none, shift left, from=3-3, to=3-4]
\arrow[from=3-3, to=3-4]
\arrow[""{name=15, anchor=center, inner sep=0}, "{\widehat{G}_{F(\one)}\otimes'' G(\id_{\one'})}"'{pos=0.4}, from=3-3, to=4-3]
\arrow["{G(F_{(2)}^{\one,\one})\otimes'' G(\id_{\one'})}",draw=none, shift left, from=3-4, to=3-5]
\arrow[from=3-4, to=3-5]
\arrow[""{name=16, anchor=center, inner sep=0}, "{\widehat{G}_{F(\one\otimes \one)}\otimes'' G(\id_{\one'})}"'{pos=0.4}, from=3-4, to=4-4]
\arrow[""{name=17, anchor=center, inner sep=0}, "{\widehat{G}_{F(\one)\otimes' F( \one)}\otimes'' G(\id_{\one'})}"'{pos=0.4}, from=3-5, to=4-5]
\arrow["{G^{(2)}_{F(\one),F(\one)}}"{pos=0.7},draw=none, shift right=2, curve={height=-20pt}, from=3-6, to=6-5]
\arrow[curve={height=-20pt}, from=3-6, to=6-5]
\arrow[""{name=18, anchor=center, inner sep=0}, "{G(\widehat F(\one))\otimes'' \id_{GF(\one)}}"'{description,pos=0.6}, from=3-6, to=6-6]
\arrow["{G(F^{(0)})\otimes '' G(\id_{\one '})}",draw=none, shift left, from=4-2, to=4-3]
\arrow[from=4-2, to=4-3]
\arrow[""{name=19, anchor=center, inner sep=0}, "{G^{(2)}_{\one',\one'}}"', from=4-2, to=5-2]
\arrow["{GF(r_{\one}^{-1})\otimes''G(\id_{\one'})}",draw=none, shift left, from=4-3, to=4-4]
\arrow[from=4-3, to=4-4]
\arrow[""{name=20, anchor=center, inner sep=0}, "{G^{(2)}_{F(\one),\one'}}"'{pos=0.4}, from=4-3, to=5-3]
\arrow["{G(F_{(2)}^{\one,\one})\otimes'' G(\id_{\one'})}",draw=none, shift left, from=4-4, to=4-5]
\arrow[from=4-4, to=4-5]
\arrow[""{name=21, anchor=center, inner sep=0}, "{G^{(2)}_{F(\one\otimes \one),\one'}}"'{pos=0.4}, from=4-4, to=5-4]
\arrow[""{name=22, anchor=center, inner sep=0}, "{G^{(2)}_{F(\one)\otimes' F(\one),\one'}}"'{pos=0.4}, from=4-5, to=5-5]
\arrow["{{G_{\one', \one'}^{(2)}}}", from=5-1, to=5-2]
\arrow[""{name=23, anchor=center, inner sep=0}, "{{G(\widecheck F)\otimes''G(\widecheck F)}}"{pos=0.3}, shift right, from=5-1, to=7-1]
\arrow["{G(F^{(0)}\otimes ' \id_{\one '})}",draw=none, shift left, from=5-2, to=5-3]
\arrow[from=5-2, to=5-3]
\arrow[""{name=24, anchor=center, inner sep=0}, "{{G(\widecheck F\otimes'\widecheck F)}}"', from=5-2, to=7-2]
\arrow["{G{(r'_{\one'})}}"'{pos=.5},draw=none, shift left =1, from=5-2, to=7-3]
\arrow[""{name=25, anchor=center, inner sep=0}, from=5-2, to=7-3]
\arrow["{G(F(r_{\one}^{-1})\otimes'\id_{\one'})}",draw=none, shift left, from=5-3, to=5-4]
\arrow[from=5-3, to=5-4]
\arrow[""{name=26, anchor=center, inner sep=0}, "{G(r'_{F(\one)})}"'{pos=.4}, from=5-3, to=6-3]
\arrow["{G(F_{(2)}^{\one,\one}\otimes' \id_{\one'})}",draw=none, shift left, from=5-4, to=5-5]
\arrow[from=5-4, to=5-5]
\arrow[""{name=27, anchor=center, inner sep=0}, "{G(r'_{F(\one\otimes \one)})}"'{pos=0.4}, from=5-4, to=6-4]
\arrow[""{name=28, anchor=center, inner sep=0}, "{G(r'_{F(\one)\otimes' F(\one)})}"'{pos=0.4}, from=5-5, to=6-5]
\arrow["{GF(r_{\one}^{-1})}", from=6-3, to=6-4]
\arrow["{G(F_{(2)}^{\one,\one})}", from=6-4, to=6-5]
\arrow["{G(\widehat F_{\one}\otimes' \id_{F(\one)})}"'{pos=.5},draw=none,shift right =-1, from=6-5, to=7-4]
\arrow[""{name=29, anchor=center, inner sep=0},from=6-5, to=7-4]
\arrow["{G^{(2)}_{F(\one),F(\one)}}"'{pos=.5},shift right =-11,draw=none, from=6-6, to=7-4]
\arrow[shift right =1,curve={height=-30pt}, from=6-6, to=7-4]
\arrow[ "{(GF)^{(2)}_{\one,\one}}",draw=none, from=6-6, to=8-5]
\arrow[""{name=30, anchor=center, inner sep=0},shift left =1, from=6-6, to=8-5]
\arrow["{{G_{F(\one), F(\one)}^{(2)}}}",draw=none, shift left, from=7-1, to=7-2]
\arrow[from=7-1, to=7-2]
\arrow[""{name=31, anchor=center, inner sep=0}, "{(GF)^{(2)}_{\one,\one}}"'{pos=0.25}, curve={height=70pt}, from=7-1, to=8-5]
\arrow[""{name=32, anchor=center, inner sep=0}, "{G(F^{(2)}_{\one,\one})}"'{pos=0.5}, from=7-2, to=8-4]
\arrow["{G(F^{(0)})}"', from=7-3, to=6-3]
\arrow[""{name=33, anchor=center, inner sep=0}, "{G(F^{(2)}_{\one,\one})}", from=7-4, to=8-4]
\arrow["{GF(r_{\one})}", from=8-4, to=8-5]
\arrow["{\textcolor{blue}{(1)}}"{description}, shift right=-5, draw=none, from=1-1, to=1-3]
\arrow["{\textcolor{blue}{(4)}}"{pos=0.4}, shift right=-10, draw=none, from=1-2, to=2-2]
\arrow["{\textcolor{blue}{(5)}}"{pos=0.4},shift right =-10, draw=none, from=1-3, to=2-3]
\arrow["{\textcolor{blue}{(2)}}"{description}, shift right=-5, draw=none, from=1-4, to=1-6]
\arrow["{\textcolor{blue}{(6)}}"{pos=0.4},shift right =-10, draw=none, from=1-4, to=2-4]
\arrow["{\textcolor{blue}{(8)}}"{description,pos=.5},shift right=-10, draw=none, from=1-6, to=6-6]
\arrow["{\textcolor{blue}{(3)}}"{description},shift right =10, draw=none, from=2-1, to=2-2]
\arrow["{\textcolor{blue}{(18)}}"'{pos=0.5}, shift right=5, draw=none, from=2-1, to=7-1]
\arrow["{\textcolor{blue}{(9)}}"{pos=0.4},shift right =-10, draw=none, from=2-2, to=3-2]
\arrow["{\textcolor{blue}{(10)}}"{pos=0.4},shift right =-10, draw=none, from=2-3, to=3-3]
\arrow["{\textcolor{blue}{(11)}}"{pos=0.4},shift right =-10, draw=none, from=2-4, to=3-4]
\arrow["{\textcolor{blue}{(7)}}"{description}, draw=none, from=2-5, to=3-6]
\arrow["{\textcolor{blue}{(12)}}"{pos=0.4},shift right =-3, draw=none, from=3-2, to=4-2]
\arrow["{\textcolor{blue}{(13)}}"{pos=0.4},shift right =-3,draw=none, from=3-3, to=4-3]
\arrow["{\textcolor{blue}{(14)}}"{pos=0.4},shift right =-5, draw=none, from=3-4, to=4-4]
\arrow["{\textcolor{blue}{(15)}}"{pos=0.4},shift right =-10, draw=none, from=4-2, to=5-2]
\arrow["{\textcolor{blue}{(16)}}"{pos=0.4},shift right =-10, draw=none, from=4-3, to=5-3]
\arrow["{\textcolor{blue}{(17)}}"{pos=0.4},shift right =-10, draw=none, from=4-4, to=5-4]
\arrow["{\textcolor{blue}{(19)}}"{description}, shift right=8, draw=none, from=5-1, to=7-2]
\arrow["{\textcolor{blue}{(20)}}"{description}, shift left=4, draw=none, from=25, to=6-3]
\arrow["{\textcolor{blue}{(24)}}"{description, pos=0.5}, draw=none, from=7-3, to=7-4]
\arrow["{\textcolor{blue}{(21)}}"{pos=0.4},shift right =-10, draw=none, from=5-3, to=6-3]
\arrow["{\textcolor{blue}{(22)}}"{pos=0.4},shift right =-10, draw=none, from=5-4, to=6-4]
\arrow["{\textcolor{blue}{(23)}}"{description,pos=0.5}, draw=none, from=6-5, to=6-6]
\arrow["{\textcolor{blue}{(26)}}"{pos=0.6}, draw=none, from=32, to=31]
\arrow["{\textcolor{blue}{(25)}}"{description,pos=.8}, draw=none, from=6-5, to=8-5]
\end{tikzcd}}
\caption{$GF$ satisfies Definition~\ref{def:eFMfunc}(b).}
\label{GF satisfies b)}
\end{figure}
\end{landscape}

\begin{landscape}
\begin{figure}[h!]
	\adjustbox{scale=.8,center}{
\begin{tikzcd}[every label/.append style = {font = \tiny}, nodes={font=\scriptsize}, column sep=4ex, row sep=11ex]
	{GF(X\otimes Y\otimes \one)} & {GF(X\otimes Y)} & {G(F(X)\otimes'F(Y))} & {G(F(X)\otimes'F(Y))} & {GF(X)\otimes''GF(Y)} \\
{G(F(X\otimes Y)\otimes'F(\one))} & {G(F(X\otimes Y)\otimes'F(\one))} & {GF(X\otimes Y\otimes\one)} & {GF(X\otimes Y)} & {GF(X)\otimes''GF(Y)} \\
{GF(X\otimes Y)\otimes''GF(\one)} & {G(F(X\otimes Y)\otimes'F(\one)\otimes'\one')} & {G(F(X\otimes Y\otimes\one)\otimes'\one')} & {G(F(X\otimes Y)\otimes'\one')} & {G(F(X)\otimes' F(Y)\otimes'\one')} & {GF(X)\otimes''GF(Y)} \\
{GF(X\otimes Y)\otimes''GF(\one)} & {G(F(X\otimes Y)\otimes'F(\one))\otimes''G(\one')} & {GF(X\otimes Y\otimes\one)\otimes''G(\one')} & {GF(X\otimes Y)\otimes''G(\one')} & {G(F(X)\otimes'F(Y))\otimes''G(\one')} \\
& {G(F(X\otimes Y)\otimes'F(\one))\otimes''G(\one')} & {GF(X\otimes Y\otimes\one)\otimes''G(\one')} & {GF(X\otimes Y)\otimes''G(\one')} & {G(F(X)\otimes'F(Y))\otimes''G(\one')} & {G(F(X)\otimes'F(Y))} \\
{GF(X\otimes Y)\otimes''GF(\one)} & {G(F(X\otimes Y)\otimes'F(\one)\otimes'\one')} & {G(F(X\otimes Y\otimes\one)\otimes'\one')} & {G(F(X\otimes Y)\otimes'\one')} & {G(F(X)\otimes' F(Y)\otimes'\one')} \\
& {G(F(X\otimes Y)\otimes'F(\one))} & {GF(X\otimes Y\otimes\one)} &&& {GF(X\otimes Y)}
\arrow[""{name=0, anchor=center, inner sep=0}, "{{{{{{{{G(F_{(2)}^{X\otimes Y,\one})}}}}}}}}"{pos=0.5},shift right =1, from=1-1, to=2-1]
\arrow[""{name=1, anchor=center, inner sep=0}, "{{{{{{{{(GF)_{(2)}^{X\otimes Y,\one}}}}}}}}}"{description, pos=0.8}, shift right=5, curve={height=50pt}, from=1-1, to=3-1]
\arrow["{{{{{{{{GF(r_{X\otimes Y})}}}}}}}}"{pos=0.5}, from=1-1, to=1-2]
\arrow[""{name=2, anchor=center, inner sep=0}, "{{{{{{{{G(F_{(2)}^{X,Y})}}}}}}}}"'{pos=0.5}, from=1-2, to=1-3]
\arrow[""{name=3, anchor=center, inner sep=0}, "{{{{{{{{(GF)_{(2)}^{X,Y}}}}}}}}}"{description, pos=0.4}, curve={height=-50pt}, from=1-2, to=1-5]
\arrow["{{{{{{{{G(\widehat{F}_X\,\otimes'\,\id_{F(Y)})}}}}}}}}"',draw=none, shift right =1, from=1-3, to=1-4]
\arrow[from=1-3, to=1-4]
\arrow[""{name=4, anchor=center, inner sep=0}, "{{{{{{{{G_{(2)}^{F(X),F(Y)}}}}}}}}}"{description, pos=0.3}, curve={height=-20pt}, from=1-3, to=1-5]
\arrow["{{{{{{{{G(F^{(2)}_{X,Y})}}}}}}}}"', from=1-4, to=2-4]
\arrow[""{name=5, anchor=center, inner sep=0}, "{{{{{{{{G_{(2)}^{F(X),F(Y)}}}}}}}}}"{description}, from=1-4, to=2-5]
\arrow["{{{{{{{{G((r'_{F(X)\otimes' F(Y)})^{-1})}}}}}}}}"{description, pos=0.7}, from=1-4, to=3-5]
\arrow[""{name=6, anchor=center, inner sep=0}, "{{{{{{{{G(\widehat{F}_X)\,\otimes''\,\id_{GF(Y)}}}}}}}}}"{pos=0.6}, from=1-5, to=2-5]
\arrow[""{name=7, anchor=center, inner sep=0}, "{{{{{{{{\widehat{GF}_{X}\,\otimes''\,\id_{GF(Y)}}}}}}}}}"{description, pos=0.7},curve={height=-70pt}, from=1-5, to=3-6]
\arrow["{{{{{{{{G_{(2)}^{F(X\otimes Y),F(\one)}}}}}}}}}"{pos=0.5},shift right =1,from=2-1, to=3-1]
\arrow["{{{{{{{{G(\widehat{F}_{X\otimes Y}\,\otimes'\,\id_{F(\one)})}}}}}}}}"{pos=0.5},draw=none, shift right=-1, from=2-1, to=2-2]
\arrow[from=2-1, to=2-2]
\arrow[ "{{{{{{{{G(F^{(2)}_{X\otimes Y,\one})}}}}}}}}",draw=none, shift right =-1, from=2-2, to=2-3]
\arrow[""{name=8, anchor=center, inner sep=0}, from=2-2, to=2-3]
\arrow["{{{{{{{{G((r'_{F(X\otimes Y)\,\otimes'\,F(\one)})^{-1})}}}}}}}}"{pos=0.5}, from=2-2, to=3-2]
\arrow["{{{{\textcolor{blue}{(8)}}}}}"{description, pos=0.5}, shift right =10, draw=none, from=2-3, to=3-3]
\arrow[""{name=9, anchor=center, inner sep=0}, "{{G_{(2)}^{F(X\otimes Y),F(\one)}}}"{description, pos=0.3},from=2-2, to=4-1]
\arrow["{{{{{{{{GF(r_{X\otimes Y})}}}}}}}}",draw=none, shift right =-1, from=2-3, to=2-4]
\arrow[ from=2-3, to=2-4]
\arrow["{{{{{{{{G((r'_{F(X\otimes Y\otimes\one)})^{-1})}}}}}}}}"{pos=0.5}, from=2-3, to=3-3]
\arrow["{{{{\textcolor{blue}{(9)}}}}}"{description, pos=0.5}, shift left=-10, draw=none, from=2-4, to=3-4]
\arrow["{{{{{{{{G((r'_{F(X\otimes Y)})^{-1})}}}}}}}}"{pos=0.6}, from=2-4, to=3-4]
\arrow[""{name=10, anchor=center, inner sep=0}, "{{{{{{{{\widehat{G}_{F(X)}\,\otimes''\,\id_{GF(Y)}}}}}}}}}"{description, pos=0.5}, from=2-5, to=3-6]
\arrow["{{{{{{{{G(\widehat{F}_{X\otimes Y})\otimes''\id_{GF(\one)}}}}}}}}}"{description,pos=0.3}, shift right =2.1, from=3-1, to=4-1]
\arrow[""{name=12, anchor=center, inner sep=0}, "{{{{{{{{\widehat{GF}_{X\otimes Y}\otimes''\id_{GF(\one)}}}}}}}}}"{description, pos=0.8}, shift right=14, curve={height=50pt}, from=3-1, to=6-1]
\arrow[ "{{{{{{{{G(F^{(2)}_{X\otimes Y,\one}\,\otimes'\,\id_{\one'})}}}}}}}}", draw=none, shift left=1, from=3-2, to=3-3]
\arrow[""{name=13, anchor=center, inner sep=0}, from=3-2, to=3-3]
\arrow["{{{{{{{{G_{(2)}^{F(X\otimes Y)\,\otimes'\, F(\one),\one'}}}}}}}}}"{pos=0.5}, from=3-2, to=4-2]
\arrow[ "{{{{{{{{G(F(r_{X\otimes Y})\,\otimes'\,\id_{\one'})}}}}}}}}"{pos=0.4},draw=none, shift left=1, from=3-3, to=3-4]
\arrow[""{name=14, anchor=center, inner sep=0}, from=3-3, to=3-4]
\arrow["{{{{{{{{G_{(2)}^{F(X\otimes Y\otimes\one),\one'}}}}}}}}}"{pos=0.5}, from=3-3, to=4-3]
\arrow["{{{{{{{{G_{(2)}^{F(X\otimes Y),\one'}}}}}}}}}"{pos=0.5}, from=3-4, to=4-4]
\arrow["{{{{\textcolor{blue}{(11)}}}}}"{description}, shift left=-10, draw=none, from=3-5, to=4-5]
\arrow["{{{{{{{{G(F^{(2)}_{X,Y}\,\otimes'\,\id_{\one'})}}}}}}}}"'{pos=.5},draw=none, shift right =1, from=3-5, to=3-4]
\arrow[""{name=15, anchor=center, inner sep=0}, from=3-5, to=3-4]
\arrow["{{{{{{{{G_{(2)}^{F(X)\otimes'F(Y),\one'}}}}}}}}}"{pos=0.5}, from=3-5, to=4-5]
\arrow[""{name=16, anchor=center, inner sep=0}, "{{{{{{{{G^{(2)}_{F(X),F(Y)}}}}}}}}}"'{pos=0.5}, from=3-6, to=5-6]
\arrow[""{name=17, anchor=center, inner sep=0}, "{{{{{{{{(GF)^{(2)}_{X,Y}}}}}}}}}"{description, pos=0.3}, shift left=4, curve={height=-40pt}, from=3-6, to=7-6]
\arrow[""{name=18, anchor=center, inner sep=0}, "{{{{{{{{\widehat{G}_{F(X\otimes Y)}\,\otimes''\,\id_{GF(\one)}}}}}}}}}"{description,pos=0.4},shift right =1, from=4-1, to=6-1]
\arrow["{{{{{{{{G(F^{(2)}_{X\otimes Y,\one})\,\otimes''\,G(\id_{\one'})}}}}}}}}"{pos=0.5},draw=none, shift left =1, from=4-2, to=4-3]
\arrow[""{name=19, anchor=center, inner sep=0}, from=4-2, to=4-3]
\arrow["{{{{{{{{\widehat{G}_{F(X\otimes Y)\otimes'F(\one)}\,\otimes''\,\id_{G(\one')}}}}}}}}}"{pos=0.5}, from=4-2, to=5-2]
\arrow["{{{{{{{{GF(r_{X\otimes Y})\,\otimes''\,G(\id_{\one'})}}}}}}}}",draw=none, shift left=1, from=4-3, to=4-4]
\arrow[""{name=20, anchor=center, inner sep=0}, from=4-3, to=4-4]
\arrow["{{{{{{{{\widehat{G}_{F(X\otimes Y\otimes\one)}\,\otimes''\,\id_{G(\one')}}}}}}}}}"{pos=0.5}, from=4-3, to=5-3]
\arrow["{{{{{{{{\widehat{G}_{F(X\otimes Y)}\,\otimes''\,\id_{G(\one')}}}}}}}}}"{pos=0.5}, from=4-4, to=5-4]
\arrow["{{{{{{{{G(F^{(2)}_{X,Y})\,\otimes''\,G(\id_{\one'})}}}}}}}}"'{pos=0.5}, shift right=1, draw=none, from=4-5, to=4-4]
\arrow[""{name=21, anchor=center, inner sep=0}, from=4-5, to=4-4]
\arrow["{{{{{{{{\widehat{G}_{F(X)\otimes'F(Y)}\otimes''\id_{G(\one')}}}}}}}}}", from=4-5, to=5-5]
\arrow["{{{{{{{{G(F^{(2)}_{X\otimes Y,\one})\,\otimes''\,G(\id_{\one'})}}}}}}}}", shift left=1, draw=none, from=5-2, to=5-3]
\arrow[""{name=22, anchor=center, inner sep=0}, from=5-2, to=5-3]
\arrow[""{name=23, anchor=center, inner sep=0}, "{{{{{{{{G^{(2)}_{F(X\otimes Y)\otimes' F(\one),\one'}}}}}}}}}"{pos=0.5}, from=5-2, to=6-2]
\arrow["{{{{{{{{GF(r_{X\otimes Y})\,\otimes''\,G(\id_{\one'})}}}}}}}}"{pos=0.5},draw=none, shift left =1, from=5-3, to=5-4]
\arrow[""{name=24, anchor=center, inner sep=0},  from=5-3, to=5-4]
\arrow["{{{{{{{{G^{(2)}_{F(X\otimes Y\otimes\one),\one'}}}}}}}}}"{pos=0.5}, from=5-3, to=6-3]
\arrow["{{{{{{{{G^{(2)}_{F(X\otimes Y),\one'}}}}}}}}}"{pos=0.5}, from=5-4, to=6-4]
\arrow["{{{{{{{{G(F^{(2)}_{X,Y})\,\otimes''\,G(\id_{\one'})}}}}}}}}"'{pos=0.5}, draw=none,shift right=1, from=5-5, to=5-4]
\arrow[""{name=25, anchor=center, inner sep=0},from=5-5, to=5-4]
\arrow["{{{{{{{{G^{(2)}_{F(X)\otimes'F(Y),\one'}}}}}}}}}"{pos=0.5}, from=5-5, to=6-5]
\arrow[""{name=26, anchor=center, inner sep=0}, "{{{{{{{{G(F^{(2)}_{X,Y})}}}}}}}}"{description, pos=0.5}, from=5-6, to=7-6]
\arrow["{{{{{{{{G^{(2)}_{F(X\otimes Y),F(\one)}}}}}}}}}"{pos=0.4},shift right =1.5,draw=none, from=6-1, to=7-2]
\arrow[from=6-1, to=7-2]
\arrow[""{name=27, anchor=center, inner sep=0}, "{{{{{{{{(GF)^{(2)}_{X,Y}}}}}}}}}"{description, pos=0.8}, curve={height=90pt}, from=6-1, to=7-3]
\arrow["{{{{{{{{G(F^{(2)}_{X\otimes Y,\one}\,\otimes'\,\id_{\one'})}}}}}}}}"{pos=0.5},draw=none, shift left =1, from=6-2, to=6-3]
\arrow[""{name=28, anchor=center, inner sep=0},  from=6-2, to=6-3]
\arrow["{{{{{{{{G(r'_{F(X\otimes Y)\otimes'F(\one)})}}}}}}}}"{pos=0.5}, from=6-2, to=7-2]
\arrow[ "{{{{{{{{G(F(r_{X\otimes Y})\,\otimes'\,\id_{\one'})}}}}}}}}",draw=none, shift left =1, from=6-3, to=6-4]
\arrow[""{name=29, anchor=center, inner sep=0}, from=6-3, to=6-4]
\arrow["{{{{{{{{G(r'_{F(X\otimes Y\otimes\one)})}}}}}}}}", from=6-3, to=7-3]
\arrow["{{{{{{{{G(r'_{F(X\otimes Y)})}}}}}}}}",draw=none, shift right=1.5, from=6-4, to=7-6]
\arrow[from=6-4, to=7-6]
\arrow["{{{{{{{{G(r'_{F(X)\otimes'F(Y)})}}}}}}}}"'{description,pos=0.5}, curve={height=27pt}, from=6-5, to=5-6]
\arrow[ "{{{{{{{{G(F^{(2)}_{X,Y}\,\otimes'\,\id_{\one'})}}}}}}}}"'{pos=.5},draw=none, shift right =1, from=6-5, to=6-4]
\arrow[""{name=30, anchor=center, inner sep=0},  from=6-5, to=6-4]
\arrow["{{{{\textcolor{blue}{(24)}}}}}"{description,pos=.7},shift right=15, draw=none, from=5-6, to=7-6]
\arrow[ "{{{{{{{{G(F^{(2)}_{X\otimes Y,\one})}}}}}}}}", shift left=1,draw=none, from=7-2, to=7-3]
\arrow[""{name=31, anchor=center, inner sep=0}, from=7-2, to=7-3]
\arrow[""{name=32, anchor=center, inner sep=0}, "{{{{{{{{GF(r_{X\otimes Y})}}}}}}}}", from=7-3, to=7-6]
\arrow["{{{{\textcolor{blue}{(2)}}}}}"{description, pos=0.7}, shift right=10, draw=none, from=1-1, to=2-1]
\arrow["{{{{\textcolor{blue}{(1)}}}}}"{description, pos=0.5}, shift left=8, draw=none, from=1-2, to=1-4]
\arrow["{{{{\textcolor{blue}{(3)}}}}}"{description, pos=0.5}, shift left=14, draw=none, from=2-1, to=2-4]
\arrow["{{{{\textcolor{blue}{(4)}}}}}"{description, pos=0.5}, shift right=3, draw=none, from=1-4, to=1-5]
\arrow["{{{{\textcolor{blue}{(10)}}}}}"{description, pos=.2}, shift right=3, draw=none, from=2-4, to=2-5]
\arrow["{{{{\textcolor{blue}{(12)}}}}}"{description, pos=0.8}, shift left=5, curve={height=18pt}, draw=none, from=5, to=16]
\arrow["{{{{\textcolor{blue}{(5)}}}}}"{description, pos=0.5}, shift left=1, draw=none, from=1-5, to=3-6]
\arrow["{{{{\textcolor{blue}{(6)}}}}}"{description, pos=0.2}, shift right=10, curve={height=12pt}, draw=none, from=4-1, to=6-1]
\arrow["{{{{\textcolor{blue}{(7)}}}}}"{description,pos=.8}, shift left=15, draw=none, from=2-1, to=3-1]
\arrow["{{{{\textcolor{blue}{(14)}}}}}"{description, pos=0.5}, shift right=10, draw=none, from=3-3, to=4-3]
\arrow["{{{{\textcolor{blue}{(13)}}}}}"{description,pos=.5}, draw=none, from=4-1, to=5-2]
\arrow["{{{{\textcolor{blue}{(16)}}}}}"{description, pos=0.5}, shift left=-8, draw=none, from=4-3, to=5-3]
\arrow[""{name=33, anchor=center, inner sep=0}, "{{{{\textcolor{blue}{(17)}}}}}"{description}, shift left=-6,  draw=none, from=4-4, to=5-4]
\arrow["{{{{\textcolor{blue}{(18)}}}}}"{description}, shift left=-8, draw=none, from=4-5, to=5-5]
\arrow["{{{{\textcolor{blue}{(19)}}}}}"{description}, shift left=-10, draw=none, from=5-3, to=6-3]
\arrow["{{{{\textcolor{blue}{(20)}}}}}"{description}, shift left=-10, draw=none, from=5-4, to=6-4]
\arrow["{{{{\textcolor{blue}{(21)}}}}}"{description}, shift right=10, draw=none, from=5-5, to=6-5]
\arrow["{{{{\textcolor{blue}{(25)}}}}}"{description,pos=.2}, shift left=8, draw=none, from=5-6, to=7-6]
\arrow["{{{{\textcolor{blue}{(22)}}}}}"{description}, shift left=-10,  draw=none, from=6-3, to=7-3]
\arrow["{{{{\textcolor{blue}{(23)}}}}}"{description}, draw=none, from=29, to=32]
\arrow["{{{{\textcolor{blue}{(26)}}}}}"{description, pos=0.05}, shift left=5, curve={height=-30pt}, draw=none, from=7-2, to=27]
\arrow["{{{{\textcolor{blue}{(15)}}}}}"{description, pos=0.5}, shift right=10, draw=none, from=3-4, to=4-4]
\end{tikzcd}
}
\caption{$GF$ satisfies  Definition~\ref{def:eFMfunc}(c)(iii).} 
\label{GF satisfies c)iii)}
\end{figure}
\end{landscape}




\clearpage

\section*{Acknowledgements}

We thank the anonymous referee for their careful feedback, time, and consideration. We would like to thank Harshit Yadav for sharing some of their graphical arguments in  Appendix~\ref{sec:HF}.
Czenky was partially supported by Simons Collaboration Grant No. 999367. Walton was partially supported by the US NSF grant \#DMS-2348833, and by an AMS Claytor-Gilmer Research Fellowship. 

\bibliography{ExtFrobMonStruc}
\bibliographystyle{alpha}

\end{document}